\documentclass[letterpaper, 11pt,  reqno]{amsart}

\usepackage{amsmath,amssymb,amscd,amsthm,amsxtra, esint}

\usepackage{mathrsfs}

\usepackage[implicit=true]{hyperref}

\allowdisplaybreaks[2]

\usepackage{color}

\newtheorem{theorem}{Theorem} [section]

\newtheorem{lemma}[theorem]{Lemma}
\newtheorem{proposition}[theorem]{Proposition}
\newtheorem{remark}[theorem]{Remark}

\newtheorem{definition}[theorem]{Definition}
\newtheorem{corollary}[theorem]{Corollary}

\DeclareMathOperator{\med}{med}

\newcommand{\noi}{\noindent}
\newcommand{\Z}{\mathbb{Z}}
\newcommand{\R}{\mathbb{R}}
\newcommand{\C}{\mathbb{C}}
\newcommand{\T}{\mathbb{T}}

\let\P= \undefined
\newcommand{\P}{\mathbf{P}}

\newcommand{\E}{\mathbb{E}}

\newcommand{\F}{\mathcal{F}}

\newcommand{\al}{\alpha}

\newcommand{\dl}{\delta}

\newcommand{\eps}{\varepsilon}
\newcommand{\kk}{\kappa}
\newcommand{\g}{\gamma}
\newcommand{\G}{\Gamma}
\newcommand{\ld}{\lambda}

\newcommand{\ft}{\widehat}

\newcommand{\wt}{\widetilde}
\newcommand{\cj}{\overline}

\newcommand{\dt}{\partial_t}

\renewcommand{\l}{\ell}
\renewcommand{\o}{\omega}
\renewcommand{\O}{\Omega}

\newcommand{\les}{\lesssim}
\newcommand{\ges}{\gtrsim}

\newcommand{\jb}[1]
{\langle #1 \rangle}
\newcommand{\jbb}[1]
{[\hspace{-0.6mm}[ #1 ]\hspace{-0.6mm}]}

\newcommand{\ind}{\mathbf 1}

\newcommand{\N}{\mathbb{N}}

\newtheorem*{ackno}{Acknowledgements}

\numberwithin{equation}{section}
\numberwithin{theorem}{section}

\begin{document}
\baselineskip = 13.9pt

\title[Gibbs dynamics for FNLS]{Gibbs Dynamics for Fractional Nonlinear Schr{\"o}dinger Equations with Weak Dispersion}

\author[R. Liang, Y. Wang]
{Rui Liang and Yuzhao Wang}

\address{
Rui Liang\\
School of Mathematical Sciences\\
South China Normal University\\
Guangzhou\\
Guangdong\\
510631\\
P. R. China\\
and 
School of Mathematics\\
Watson Building\\
University of Birmingham\\
Edgbaston\\
Birmingham\\
B15 2TT\\ United Kingdom}
\email{RXL833@alumni.bham.ac.uk}

\address{
Yuzhao Wang\\
School of Mathematics\\
Watson Building\\
University of Birmingham\\
Edgbaston\\
Birmingham\\
B15 2TT\\ United Kingdom}
\email{y.wang.14@bham.ac.uk}

\subjclass[2010]{35Q55, 37K99}

\keywords{Gibbs dynamics; 
weakly dispersive fractional NLS; 
strong solution;
random averaging operator}

\begin{abstract}
  We consider the Cauchy problem for the one-dimensional periodic cubic
  nonlinear fractional Schr{\"o}dinger equation (FNLS) with initial data
  distributed via its associated Gibbs measure. 
  We construct global strong solutions with the flow property for the FNLS on the support of the Gibbs measure in the full dispersive range,  thus resolving a question proposed by Sun-Tzvetkov (2021). As a byproduct, we prove the invariance of the Gibbs measure and almost sure global well-posedness for FNLS. 
\end{abstract}

{\maketitle}

\section{Introduction}

In this paper, we study the problem of invariant Gibbs measure for the fractional cubic nonlinear Schr{\"o}dinger equation (FNLS) on the torus $\mathbb{T}$:
\begin{equation}
  \label{FNLS} 
  \begin{cases}
    i  \partial_t u - {\rm D}_x^{\alpha} u = \pm | u |^2 u,\\
    u (0) = u_0,
  \end{cases} 
  \quad (t, x) \in \mathbb{R} \times \mathbb{T}
\end{equation}

\noi 
where $\mathbb{T}=\mathbb{R}/ (2 \pi \mathbb{Z})$ is the one-dimensional torus
and ${\rm D}_x^{\alpha} = (- \partial_{x x})^{\alpha / 2}$ is the
spatial fractional derivative defined as the Fourier multiplier $\F_x (
{\rm D}_x^{\alpha} f)(k) : = | k |^{\alpha} \ft f (k)$, $k \in \Z$. 
We use the symbol $\pm$ to indicate the type of nonlinearity in FNLS  \eqref{FNLS}, which can be either defocusing ($+$) or focusing ($-$). 
The parameter $\alpha > 0$ measures the dispersion effect in FNLS  \eqref{FNLS}. 
The dispersion becomes stronger as $\alpha$ increases. 
When $\al = 2$, FNLS  \eqref{FNLS} reduces to the well-known cubic nonlinear Schr\"odinger equation (NLS), which models nonlinear phenomena in optics and plasma physics; see \cite{SS99} for more details. 
When $\al \neq 2$, FNLS \eqref{FNLS} arises in fractional quantum mechanics {\cite{Laskin00}}, and also in the water wave models {\cite{IP14}}. 
Quantum mechanics can be seen as the theory of functionals on measure generated by the Brownian motions in the path-integral approach; see \cite{Co65,GS98}. 
Fractional quantum mechanics is a generalization of quantum mechanics by replacing the Brownian motions with L\'evy processes \cite{Laskin00}. 
Moreover, it is shown in {\cite{KLS13}} that FNLS \eqref{FNLS} describes the continuum limit of lattice interactions when $\al \in (1, 2]$. Specifically, the case when $1 < \alpha < 2$ corresponds to the long-range lattice interactions, while the case when $\al = 2$ corresponds to the short-range or quick-decaying interactions. 
We also note that the cubic nonlinear half-wave equation ($\al = 1$) has various physical applications, such as wave turbulence \cite{MMT1997,CMMT01} and gravitational collapse \cite{ES07}. However, when $\al = 1$, the equation \eqref{FNLS} is no longer dispersive. 
In this paper, we focus on FNLS  \eqref{FNLS} for $\al \in (1,2)$, where the dispersion is weaker than that of the cubic nonlinear Schr\"odinger equation (NLS) (for \eqref{FNLS} with $\al = 2$).

The equation \eqref{FNLS} is a Hamiltonian PDE with the conserved Hamiltonian:
\begin{align}
\label{Hamiltonian}
  H (u) = \int_{\mathbb{T}} \bigg( \big| {\rm D}_x^{\alpha / 2} u \big|^2 \pm \frac{1}{2} | u |^4 \bigg) d x. 
\end{align}

\noi 
The dynamics of the equation \eqref{FNLS} also preserves the mass:
\begin{align}
\label{mass}
  M (u) =  \int_{\mathbb{T}} |u|^2d x. 
\end{align}

\noi 
The well-posedness of FNLS  \eqref{FNLS} in the low regularity setting has
been extensively studied. If $\alpha = 2$, \eqref{FNLS} is well-posed in
$L^2 (\mathbb{T})$ in the deterministic sense {\cite{B93a}}. However, when
$\alpha < 2$, where the equation is less dispersive, the deterministic local
well-posedness in $L^2 (\mathbb{T})$ is not available. Nevertheless, it has been proved in {\cite{De15,CHKL15}} that FNLS \eqref{FNLS} with $1 < \alpha < 2$ is
locally well-posed in $H^s (\mathbb{T})$ for $s \ge \frac{2 - \alpha}{4}$. 
The main purpose of this paper is to study FNLS \eqref{FNLS} from a probabilistic point of view. 
In particular, we focus on the invariance of the Gibbs measure \eqref{Gibbs} under the flow of \eqref{FNLS}, where the Gibbs measure is a probability measure on distributions on $\T$. See Subsection \ref{SUB:Gibbs} for a precise definition.

The existence and invariance of Gibbs measures for nonlinear PDEs are fascinating topics that have attracted a lot of attention in recent years \cite{Friedlander85,LRS88,BO94,BO96,B97,Tz06,Tz08,O1,O2,BOB1,BOB2,BOB3,OST22,DNY2,DNY3,DNY4,ORSW21,OOT1,ST20,ST21,LW22,BDNY}. Here we did not aim to cover all the relevant literature on this topic. 
We would like to acknowledge the pioneering works of Bourgain \cite{BO94,BO96} that inspire many subsequent studies on the invariant measure for Hamiltonian PDEs, especially those by Burq and Tzvetkov \cite{Tz06,Tz08,BT2008a,BT2008b}, that popularized this line of research. One of the main motivations for studying such measures is to understand the long-time behaviour of solutions to nonlinear PDEs, especially those that exhibit chaotic or turbulent phenomena. 
One such example is the famous question posed by Zakharov; see \cite{Friedlander85}.
``Numerical experiments demonstrated [that the 1-d periodic cubic NLW] possesses
the `returning' property, i.e. solutions appear to be very close to the initial state $\cdots$, after some time of rather chaotic evolution. The problem is to explain this phenomenon."
The key ingredients to explain this phenomenon are the existence of an invariant Gibbs measure and the flow property for the dynamics, 
which enable the use of Poincar\'e recurrence theorem. However, proving the existence and invariance of such a measure is highly nontrivial, as it requires overcoming several analytical and probabilistic challenges. 
See \cite{Friedlander85,LRS88,BO94,BO96,BDNY} and references therein for further discussion.
We also refer readers to \cite{Oh10} for a nice review of other invariant measures.

For the NLS case, i.e. \eqref{FNLS} with $\al = 2$, Lebowitz-Rose-Speer \cite{LRS88} considered the Gibbs measure of the form 
\begin{align}
\label{Gibbs0}
    d\rho = Z^{-1} e^{-\beta H(u)} du
\end{align}

\noi 
where $Z$ is a normalization constant and $H(u)$ is the Hamiltonian given in \eqref{Hamiltonian}.
In particular, 
they showed that the Gibbs measure \eqref{Gibbs0} is a well-defined probability measure on $H^{\frac12-} (\T)$ for the defocusing case; while for the focusing case, their result only holds with the $L^2$-cutoff $\ind_{\{\|u\|_{L^2(\mathbb{T})} \le K\}}$ for some $K>0$. 
Bourgain \cite{BO94} proved the Gibbs measure \eqref{Gibbs0} is invariant under the flow of NLS and global well-posedness almost surely on the statistical ensemble.

For the fractional NLS \eqref{FNLS} with $\al \in (1,2)$, the invariance of the Gibbs measure and the ``returning" property for the dynamics are not yet fully understood.
When $\al > \frac43$, these problems can be settled by using the deterministic theory \cite{De15,CHKL15} since the local well-posedness holds on the support of the Gibbs measure. 
To go beyond the threshold $\frac43$, Sun-Tzvetkov \cite{ST20} exploited a probabilistic argument, where they managed to handle the case $\al > \frac65$.
When $\al \in (1, \frac65]$, they also proved the convergence of the Galerkin approximation scheme for the FNLS by using the Bourgain-Bulut approach \cite{BOB1,BOB2,BOB3}. 
However, this argument is insufficient to show the flow property for the limiting dynamics, thus preventing us from applying the Poincar\'e recurrence theorem. 
Recently, Sun-Tzvetkov \cite{ST21} further improved their results in \cite{ST20} by using the theory of random averaging operators introduced by Deng-Nahmod-Yue \cite{DNY2}, and extended the range of $\al$ to $\al > \frac{31-\sqrt{233}}{14}$. They also conjectured the existence of global strong solutions to \eqref{FNLS} for all $\al > 1$ (see \cite[Question 1.1]{ST21}).

The main goal of this paper is to affirmatively answer Sun-Tzvetkov's question. 
Specifically, we build strong solutions, satisfying the flow property, to FNLS \eqref{FNLS} in the full range $\alpha > 1$. 
The range $\alpha > 1$ is also expected to be optimal. As a matter of fact, we shall show that there is a dramatic change in the regularity of the second Picard iterate between $\al > 1$ and $\al = 1$. See Theorem \ref{THM:div} for further discussion.

Before proceeding, we recall the notion of Gibbs measure.

\subsection{Gibbs measures}
\label{SUB:Gibbs}

The Gibbs measure, denoted by $d \rho$, can be formally written as
\begin{align}
  \label{Gibbs}
  d \rho = Z^{-1} e^{- H (u) - M(u)} d u,
\end{align}
where $Z$ is a normalization constant. Here, $H (u)$ and $M(u)$ are the Hamiltonian and mass of \eqref{FNLS} given in \eqref{Hamiltonian} and \eqref{mass}, respectively. 
Also, $du$ is formally an infinite-dimensional Lebesgue measure, which does not make sense as it is. We should understand $du$ as part of the Gaussian measure; see \eqref{GFM} in the following.
As both Hamiltonian and mass are conserved along the dynamics of \eqref{FNLS}, and the volume $du$ is also formally invariant, it is expected that Gibbs measure \eqref{Gibbs} is invariant under the dynamics of \eqref{FNLS}.
By adding the mass term to \eqref{Gibbs}, we avoid some technical issues at zero frequency. 
We may define the Gibbs measure $d \rho$ as an absolutely continuous
measure with respect to the following massive Gaussian measure:
\begin{align} 
\label{GFM}
  d \mu = e^{- \int_{\mathbb{T}} 
  (| {\rm D}_x^{\alpha / 2} u |^2 + | u |^2) d x} d u = \prod_{k \in \Z} e^{- \jbb{k}^{\al} |u_k|^2} du_k d\cj u_k,
\end{align}
where $u_k$ is the $k$-th Fourier coefficient of $u$.
The above Gaussian measure given in \eqref{GFM} is the induced probability measure under the map
\begin{equation}
  \label{GFF} 
  \omega \longmapsto u_0^{\omega} (x) = \sum_{k \in
  \mathbb{Z}} \frac{g_k (\omega)}{\jbb{k}^{\frac\alpha 2}} e^{i k x} .
\end{equation}
Here $\jbb{k} := (1 + | k |^{\alpha})^{\frac1\al}$ and $( g_k )_{k \in
\mathbb{Z}}$ is a sequence of independent standard complex-valued Gaussian
random variables on a probability space $(\Omega, \mathcal{F}, \mathbb{P})$.
One notes that $\jbb{k}$ depends on $\al$ in a non-essential way and $\jbb{k} \sim \jb{k}$; see \eqref{Eqn:br}.
Typical elements on the support of $\mu$ can be represented as the random Fourier series given in \eqref{GFF}. 
An easy computation shows that 
\begin{align} 
\label{GFF:reg} 
 u^{\omega}_0 \in H^{\frac{\alpha - 1}{2} - } (\mathbb{T}) \setminus H^{\frac{\al -1}2} (\mathbb T) \textup{ almost surely, }
\end{align}
where
\[
  H^{\frac{\alpha - 1}{2} - } (\mathbb{T})  : = \bigcap_{s < \frac{\alpha - 1}{2} } H^s (\mathbb{T}).
\]

With the above notations, 
the defocusing Gibbs measure $\rho$ can be recast as the following weighted 
Gaussian measure
\begin{align}
    \label{Gibbs2}
  d \rho = Z^{-1} e^{- \frac{1}{2} \int_{\mathbb{T}} | u |^4 
  d x} d \mu.
\end{align} 

\noi 
When $\alpha > 1$, we note that $\| u \|_{L^4 (\mathbb{T})} < \infty$ almost surely with respect to the Gaussian measure $\mu$. Thus, the defocusing Gibbs measure \eqref{Gibbs2} is a well-defined probability measure on $H^s (\mathbb{T})$ for $s < \frac{\alpha - 1}{2}$, absolutely continuous with respect to the Gaussian measure $\mu$. 
When $\alpha \le  1$, one has $\| u \|_{L^4 (\mathbb{T})} = \infty$ almost surely, thus a renormalization is needed. 
We refer the readers to {\cite[Section 1.3]{ST20}} for further discussion.
On the other hand, the focusing Gibbs measure formally given by
\[
d \rho = Z^{-1} e^{  \frac{1}{2} \int_{\mathbb{T}} | u |^4 
  d x} d \mu
\]
cannot be normalized as a probability measure since
\[
\E_{\mu} \Big[ e^{\frac{1}{2} \int_{\mathbb{T}} | u |^4 
  d x} \Big] = \infty.
\]

\noi 
Inspired by Lebowitz-Rose-Speer \cite{LRS88} and Oh-Sosoe-Tolomeo \cite{OST22},
the authors \cite{LW22} considered the following focusing Gibbs measure with a cutoff,
\begin{align}
    \label{fGibbs}
    d \rho = Z^{-1} \ind_{\{\|u\|_{L^2(\T)} \le K\}}e^{\frac{1}{2} \int_{\mathbb{T}} | u |^4 
  d x} d \mu,
\end{align}

\noi 
associated with the focusing FNLS \eqref{FNLS}.
In particular, the authors showed that the Gibbs measure \eqref{fGibbs} is a probability measure for any finite cutoff size $K$, provided $\al > 1$.
See \cite{LW22} for further discussion.

\subsection{Gauge transform}

Given a solution $u \in C ([- T, T] ; L^2 (\mathbb{T}))$ to \eqref{FNLS},
we introduce the following invertible gauge transform
\begin{equation}
  \label{Eqn:gauge1} 
  u (t) \longmapsto \mathcal{G} (u) (t) :=
  e^{ 2 i t  \fint | u |^2 d x } u (t),
\end{equation}

\noi 
where 
\[ 
\fint  f(x) dx := \frac{1}{2\pi} \int_\T f(x) dx
\]
denotes the integration with respect to the normalized Lebesgue measure $(2\pi)^{-1} dx$ on $\T$.
A direct computation  with the mass conservation shows that the gauged function,  which we still denote by $u$,  solves the following renormalized defocusing FNLS:
\begin{equation}
\label{FNLS2}  
i \dt u  - {\rm D}_x^{\alpha} u =  \Big( |u|^2 -2  \fint |u|^2 dx\Big) u . 
\end{equation}

\noi 
Note that the gauge transform $\mathcal{G}$ in \eqref{Eqn:gauge1} is invertible in $C ([- T, T] ;
L^2 (\mathbb{T}))$. In particular, we can freely convert solutions to
\eqref{FNLS} into solutions to \eqref{FNLS2} and vice-versa as long as
they are in $C ([- T, T]; L^2 (\mathbb{T}))$. 
By rewriting \eqref{FNLS2}
in the Duhamel formulation, we have
\begin{equation}
  \label{Duhamel} u (t) = S (t) u_0^{\omega} - i  \int^t_0 S (t - t')
  \mathcal{N} (u) (t') d t',
\end{equation}
where $S (t) = e^{- i  t {\rm D}_x^{\alpha}}$ denotes
the linear evolution and
\[ 
  \mathcal{N} (u) = \Big( |u|^2 -2  \fint |u|^2 dx\Big) u . 
\]
We decompose the nonlinearity $\mathcal{N} (u)$ into non-resonant and resonant parts i.e. 
\[
  \mathcal{N} (u) = \mathcal{Q} (u,u,u) + \mathcal{R} (u,u,u),
\]

\noi 
where the trilinear forms are defined as
\begin{align} 
\begin{split}
    \mathcal{Q} (u, v, w) (t, x) & := \sum_{k \in \mathbb{Z}}
    \sum_{
    \G (k) 
    } u_{k_1} (t) \overline{v_{k_2} (t)} w_{k_3} (t)
    e^{i  (k_1 - k_2 + k_3) x} ;\\
    \mathcal{R} (u, v, w) (t, x) & := \sum_{k \in \mathbb{Z}} u_k (t)
    \overline{v_k (t)} w_k (t) e^{i  k x}.
\end{split}
\label{trilinear}
\end{align}

\noi 
Here $u_k, v_k, w_k$ denote the spatial Fourier transform of $u, v, w$ respectively, and $\G (k)$ denotes the hyperplane of $\Z^3$,
\begin{align} 
\label{hyperplane}
  \G (k)  := \left\{ (k_1, k_2, k_3) \in \mathbb{Z}^3 ; k = k_1 - k_2 + k_3 \text{ and } k_1, k_3 \neq k_2 \right\}. 
\end{align}
When all the arguments are the same, we simply denote by $\mathcal{Q} (u) =
\mathcal{Q} (u, u, u)$ and $\mathcal{R} (u) = \mathcal{R} (u, u, u)$. The term
$\mathcal{Q} (u)$ denotes the non-resonant part of the renormalized
nonlinearity $\mathcal{N} (u)$, and $\mathcal{R} (u)$ denotes the resonant
part. Then, we have
\[ 
  \mathcal{N} (u) = \mathcal{Q} (u) + \mathcal{R} (u). 
\]
\begin{remark}
  {\rm
    When $\alpha > 1$, the local-in-time solution to \eqref{Duhamel} we construct will be in $H^{\frac{\alpha - 1}{2}-} (\mathbb{T})$, almost surely. 
    See Theorem \ref{THM:main} below. 
    Thus the solution to \eqref{Duhamel} lies in $L^2(\T)$, where the gauge transform
    $\mathcal{G}$ is invertible; i.e. the renormalized FNLS \eqref{FNLS2}
    is equivalent to the original FNLS \eqref{FNLS}. 
    The use of gauge transformation \eqref{Eqn:gauge1} removes some troublesome resonances, which improves the
    main counting estimates in Lemma \ref{LEM:counting1} and thus the tensor
    estimates in Lemma \ref{LEM:tensor}.
  }
\end{remark}

\subsection{Main results}

In what follows, we consider the Cauchy problem of defocusing FNLS  \eqref{FNLS2} with
Gaussian random data $u_0^{\omega}$ given in \eqref{GFF}. 
Let $N\in 2^{\Z_{\ge 0}} \cup \{1/2\}$ be a dyadic number,
define projections $\Pi_N$ such that
\begin{align} 
  (\Pi_N f)_k =\ind_{\jb{k} \le  N} \cdot f_k ,
\label{Pi_N}
\end{align}

\noi 
and $\Delta_N := \Pi_N - \Pi_{N / 2}$.
In particular, we note
\begin{align} 
\label{P1/2}
 (\Pi_{1 / 2} f)_k = 0. 
\end{align} 

\noi 
We start with a truncated version of \eqref{FNLS2}, with truncated random
initial data
\begin{equation}
  \label{FNLS2_N} 
  \begin{cases}
    \displaystyle i  \partial_t u_N - {\rm D}_x^{\alpha} u_N = \Pi_N \bigg[ \bigg(|
    u_N |^2  - 2 {{  \fint | u_N |^2 d
    x}} \bigg) u_N \bigg],\\
    \displaystyle  u_N (0) = \Pi_N u^{\omega}_0 = \sum_{\jb{k} \le N} \frac{g_k (\omega)}{\jbb{k}^{\frac\alpha 2}} e^{i  k x},
  \end{cases} \quad (t, x) \in \mathbb{R} \times \mathbb{T}.
\end{equation} 
From \eqref{P1/2} and \eqref{FNLS2_N},
we also note that
\begin{align*}
    u_{1/2} = 0. 
\end{align*}

\noi 
In this paper, we show that \eqref{FNLS2} admits strong solutions, which are the unique limits of solutions to \eqref{FNLS2} and satisfying the flow property. 
It follows from \cite{Th17} (or see \cite[Theorem 4]{ST20}) that \eqref{FNLS2_N} is globally well-posed with truncated (and thus smooth) initial data.

Our first goal
is to construct local-in-time solutions to \eqref{FNLS2} almost surely
with respect to the random initial data \eqref{GFF}, which is addressed in
the following theorem.

\begin{theorem}
\label{THM:main}
Let $\al >1$ and $u_0^\o$ be as in \eqref{GFF}.
Let $u_N$ be the solution to the truncated system \eqref{FNLS2_N}.
Then for $0 < T \ll 1$,
there exists a set $\O_T\subset \O$ with $\P(\O_T^c) \le C_\theta e^{-T^{-\theta}}$, where $\theta > 0$ is a small constant determined by $\al$, such that for $\o \in \O_T$, the sequence $\{u_N\}$ converges to a limit $u$ in $C ([-T,T]; H^{\frac{\al-1}2 -} (\T))$. The limit $u(t)$ solves \eqref{FNLS2}.
\end{theorem}

\begin{remark}
\label{Rmk:known}
\rm
Theorem \ref{THM:main} has been shown for $\alpha > \frac{31 - \sqrt{233}}{14} \approx 1.124$ in \cite{ST21} by using the theory of random averaging operators developed by Deng-Nahmod-Yue \cite{DNY2}, with a precursor \cite{Br21}. 
In what follows, we only focus on the range $\al \in (1,2)$. 
The random averaging operator can be seen as the dispersive version of the well-known paracontrolled distribution method developed by Gubinelli-Imkeller-Perkowski \cite{GIP}.
The key observation is to decompose the solution into a combination of linear, para-linear, and nonlinear parts, such that the para-linear term (or para-controlled term) possesses a randomness structure and the nonlinear part is smoother. 
\end{remark}

\begin{remark}
\rm
\label{RMK:unique}
The solution $u$ constructed in Theorem \ref{THM:main} is the unique limit of canonical smooth approximations; i.e. the limit $u$ obtained in Theorem \ref{THM:main} does not depend on the choice of `good' approximations of \eqref{FNLS2}. For further discussion on this issue, see \cite[Remark 1.6 (3)]{DNY2} and \cite[Remark 1.4]{DNY3}. 
We also refer the readers to \cite{ST20} for some different approximations. 
\end{remark}

\begin{remark}\rm
    The main difficulty in establishing Theorem \ref{THM:main} comes from the weak dispersion when $\al$ is small.
    Besides the random averaging operator, we need to exploit several new ideas and techniques to overcome the difficulty arising from the weaker dispersion. First of all, we employ the random tensor estimates introduced in \cite{DNY3,DNY4}, which allow us to simplify the multilinear estimates and exploit more refined counting results. See Subsection \ref{SUB:TNE}, Subsection \ref{SUB:redu}, and Subsection \ref{Sub:pre} for more details. Secondly, we improve the counting estimates by exploiting some key properties of the random structure. In particular, we show that the worst-case scenario of bad countings does not happen when considering the interaction between high and low frequencies, which is one of the crucial observations. We make crucial use of the decay fact in the basic counting Lemma \ref{LEM:counting1}. This enables us to reduce two dimensions in our key counting estimates; see Lemma \ref{LEM:S}. This contrasts sharply with the half-wave equation, where at most one dimension can be reduced. See Theorem \ref{THM:div}. We also use the $\Gamma$-condition to enhance some of the counting arguments. See Subsection \ref{SUB:count} for further discussion. Last but not least, we take advantage of the crucial cancellation that arises from the conservation structure of the random averaging operator, namely the unitary property. See Subsection \ref{SUB:unitary} for further discussion. It is important to note that the counting arguments become more complicated in higher dimensions. Specifically, the $\Gamma$-condition does not eliminate the impact of high-frequency terms due to the inner product structure. We plan to address these issues in our future work.
\end{remark}

Theorem \ref{THM:main} shows that the FNLS  \eqref{FNLS} is almost surely
locally well-posed with the random initial data \eqref{GFF}. To extend the local solution globally, one may adapt Bourgain's invariant measure argument {\cite{BO94,BO96}}. See also \cite[Section 6]{DNY2} for a detailed argument in the random averaging operator setting. More precisely, we use the invariance of the
fractional Gibbs measure under the finite-dimensional approximation of the
FNLS  flow to obtain a uniform control on the solutions and then apply a PDE
approximation argument to extend the local solutions to \eqref{FNLS}
obtained from Theorem \ref{THM:main} to global ones. The proof of Theorem \ref{THM:main}, using the theory of random averaging operators, also implies the stability of the truncated solution map. See \cite[Section 5]{DNY2}. As a consequence, we obtain the group property for the flow map and the invariance of the fractional Gibbs measure under the resulting global flow of the FNLS  \eqref{FNLS}.

\begin{theorem}
\label{THM:main2}
  Let $\alpha > 1$. Then, the defocusing renormalized FNLS  \eqref{FNLS2} on $\mathbb{T}$ is globally well-posed almost surely with respect to the
  Gibbs measure $\rho$ in \eqref{Gibbs}. Moreover, the flow maps satisfy the usual group (flow) property and keep the Gibbs measure $d\rho$ invariant.
\end{theorem}

Once we have Theorem \ref{THM:main}, the proof of Theorem \ref{THM:main2} is standard by now, thus we omit the proof. 
See Deng-Nahmod-Yue \cite[Section 6]{DNY2} for a proof of the 2D NLS case; also see Bourgain {\cite{BO94,BO96}}, Deng-Nahmod-Yue \cite{DNY4}, and Sun-Tzvetkov {\cite{ST20}} for more discussion. 
The crucial point is that, in the proof of Theorem \ref{THM:main}, we established a probabilistic local well-posedness result with a fine structure of the solution of \eqref{FNLS}. See Subsection \ref{SUB:idea} for further discussion.
In what follows, we shall focus on the proof of Theorem \ref{THM:main}.

\begin{remark}
  {\rm
  Theorem \ref{THM:main2} is for FNLS \eqref{FNLS} (or \eqref{FNLS2}) with the defocusing nonlinearity $| u |^2 u$. 
  The local-in-time theory still holds if we replace the defocusing
  nonlinearity with the focusing one $- | u |^2 u$. In the focusing case, to use Bourgain's invariant measure argument extending local-in-time
  dynamics globally, one needs to construct the associated focusing Gibbs measure. The authors {\cite{LW22}} constructed the focusing Gibbs measure for \eqref{FNLS}. See also {\cite{RS22}} for a different construction of the focusing Gibbs measure from quantum many-body mechanics.
  }
\end{remark}

As pointed out by Sun-Tzvetkov \cite{ST21}, the second Picard iterate enjoys some smoothing effect when $\al > 1$.
We further observe that the high-high interaction of the second Picard iterate exhibits an even stronger smoothing property; see Theorem \ref{THM:div} (i) and also Subsection \ref{SUB:HHCCC}, which plays an important role in our analysis. However, when $\al = 1$, the second Picard iterate fails to gain any smoothing; see Theorem \ref{THM:div} (ii).
To illustrate these ideas, we consider the following truncated version of the second Picard iterate of high-high interaction.
Given $N\in \N$, define the following second Picard iterate
\begin{align}
    \label{Picard}
    Z^{(2)}_N (t) = \int_0^t  e^{i(t-t'){\rm D}_x^\al} \Pi_N \bigg[\bigg( |z_N (t')|^2 - 2 \fint |z_N (t')|^2 dx \bigg) z_N (t') \bigg] dt',
\end{align}

\noi 
where 
\[
z_N (t) = \sum_{N^{1-\dl} < k < N} \frac{g_k(\o)}{\jbb{k}^{\frac\al2}} e^{it|k|^\al + ikx} 
\]

\noi 
is the truncation of the random linear solution. 
In \eqref{Picard}, we only see the high-high interactions, i.e. $N_{\med} \ges N_{\max}^{1-\dl}$. 
The following theorem shows a sharp contrast of the smoothing property of \eqref{Picard} for $\al > 1$ and for $\al =1$, respectively.
More precisely, we have the following result.

\begin{theorem}
\label{THM:div}
Let $Z^{(2)}_N (t)$ be defined by \eqref{Picard} with $N \gg 1$ and $|t| \sim 1$.
\begin{itemize}
    \item[(i)] When $\al > 1$, there exists $\dl >0$, such that $(\mathbb{E}[\|Z^{(2)}_N (t)\|^2_{L^2 (\T)}])^{\frac12} \les N^{-\frac12 -}$.
    \item[(ii)] When $\al = 1$, we have
    $(\mathbb{E}[\|Z^{(2)}_N (t) \|^2_{L^2(\T)})^{\frac12} \ges \delta^{\frac32} (\log N)^{\frac{3}{2}}$ with the constant independent of $\dl$.
\end{itemize} 

\end{theorem} 

\begin{remark} \label{RMK:phaset}
    \rm 
    Theorem \ref{THM:div} shows a dramatic change in the regularity of the second Picard iterate between $\al > 1$ and $\al = 1$. In particular, Theorem \ref{THM:div} quantifies the phase transition of the smoothing effect for Picard iterates at $\al = 1$.
    As the smoothing property of the second Picard iterate plays a crucial role in the proof of Theorem \ref{THM:main},  Theorem \ref{THM:div} also implies that the strategy in proving Theorem \ref{THM:main} does not work for the half-wave equation, i.e. \eqref{FNLS2} with $\al =1$.
\end{remark}

\begin{remark}\label{RMK:nosmooth}
    \rm 
    Picard iterate's lack of smoothing property for other weakly dispersive models has been observed. Oh \cite{Oh11} considered the Szeg\"o equation and proved that the first nontrivial Picard iterate does not gain regularity compared to the initial data. See also Camps-Gassot-Ibrahim \cite{CGI22} for a similar observation for the cubic Schr\"odinger half-wave equation.
\end{remark} 

\subsection{Main ideas}
\label{SUB:idea}
We explain the key ideas in proving Theorem \ref{THM:main}, i.e. how to construct local-in-time solutions to \eqref{FNLS} with initial data distributed according to the Gibbs measure.
We note that the Gibbs measure $d\rho$ in \eqref{Gibbs2} (or in \eqref{fGibbs}) is absolutely continuous with respect to the Gaussian measure $d\mu$ in \eqref{GFM}. 
Therefore, to prove Theorem \ref{THM:main}, it only suffices to consider \eqref{FNLS} with initial data distributed according to $d\mu$, i.e. $u|_{t=0} = u_0^\o$ given in \eqref{GFF}. 
The main difference between the weakly dispersive case $\al \in (1,2)$ and the standard case $\al  =2$ comes from the counting estimates. For instance, the values of
\[
|k_1|^\al - |k_2|^\al + |k_3|^\al - |k|^\al  
\]

\noi 
may be dense in an interval of size $1$,
under the constraint of $(k_1,k_2,k_3) \in \Gamma (k)$. This causes the loss of regularity in establishing linear and multilinear estimates regarding solutions to \eqref{FNLS2}.

\subsubsection{The Ansatz}
In this subsection, we recall the theory of random averaging operators recently developed by Deng-Nahmod-Yue \cite{DNY2}.
The idea is to include the high-low interactions in the ansatz, and write them as a low-frequency operator applied to the high-frequency. 
See \cite{DNY2} for further discussion. 

To construct the solution to \eqref{FNLS2_N} locally, 
we introduce a time cutoff.
Let $\eta$ be any Schwartz function such that $\eta (t) = 1$ for $| t |
\le  1$ and $\eta (t) = 0$ for $| t | \ge  2$. 
To simplify the notation, we will denote by
\begin{align}
    \label{chi}
    \chi (t) = \eta_T (t) : = \eta\bigg(\frac{t}T \bigg),
\end{align}

\noi 
for some $T \ll 1$ to be determined later.
Applying the interaction representation to the unknown $u_N (t)$, still denoted by $u_N(t)$,
\begin{equation*}
  u_N (t)  =
  e^{i  t {\rm D}_x^{\alpha}} u_N (t),
\end{equation*}

\noi 
the Duhamel formulation of \eqref{FNLS2_N}, restricted to $|t| \le T$, becomes
\begin{equation}
  \label{Duhamel_N} 
  \begin{split}
    u_N (t) = & \chi (t) \Pi_N u^{\omega}_0 - i  \chi (t)
    \int^t_0 \Pi_N \mathcal{M} (u_N) (t') d t'  + i 
    \chi (t) \int^t_0 \Pi_N \mathcal{R} (u_N) (t') d t',
  \end{split}
\end{equation}
where $\mathcal R$ is as in \eqref{trilinear} attaching a factor $\chi(t)$, 
and $\mathcal M (u) = \mathcal M(u,u,u)$ is the trilinear form defined by
\begin{equation}
  \label{Eqn:M} 
  \mathcal{M} (u, v, w)_k (t) = {{\chi (t)}}
  \sum_{\substack{
    k_1 - k_2 + k_3 = k\\
    k_2 \not\in \{ k_1, k_3 \}
  }} e^{i  t \Phi} \cdot u_{k_1} (t) \overline{v_{k_2}} (t)
  w_{k_3} (t),
\end{equation}
where
\begin{equation}
  \label{Eqn:Phi} 
  \Phi = \Phi (\bar k) := |k_1|^{\alpha} - |k_2|^{\alpha} +
  |k_3|^{\alpha} - |k|^{\alpha} .
\end{equation}

\noi 
In what follows, 
we focus on the formulation \eqref{Duhamel_N} in proving Theorem \ref{THM:main}. We also note that $u_{1 /
2} = 0$.

Let $y_N$ be as
\begin{equation}
  \label{Eqn:yNdef} 
  y_N = u_N - u_{N / 2}. 
\end{equation}

\noi 
Then from \eqref{Duhamel_N} we see that $y_N$ in \eqref{Eqn:yNdef} satisfies
\begin{equation}
  \label{Eqn:yN} 
  \begin{split}
    & y_N (t) = \chi (t) F_N - i  \sum_{N_{\max} = N} \chi (t) \int^t_0
    \Pi_N \mathcal{M} (y_{N_1}, y_{N_2}, y_{N_3}) (t') d t'\\
    & \hspace{1.6cm} - i  \sum_{N_{\max} \le  N / 2} \chi (t)
    \int^t_0 \Delta_N \mathcal{M} (y_{N_1}, y_{N_2}, y_{N_3}) (t') d
    t'\\
    & \hspace{1.6cm} + i  \chi (t) \int^t_0 (\Pi_N \mathcal{R} (u_N) -
    \Pi_{N / 2} \mathcal{R} (u_{N / 2})) (t') d t',
  \end{split}
\end{equation}
where
\begin{equation}
  \label{Eqn:FN} F_N = \Delta_N u_0^{\omega} = \Pi_N u^{\omega}_0 - \Pi_{N/2} u^{\omega}_0,
\end{equation}
and $N_{\max} = \max (N_1, N_2, N_3)$. 
To construct $y_N$ perturbatively, 
the main difficulty comes from the
low-low-high interactions of the second term in \eqref{Eqn:yN}.
To remove this bad interaction, we introduce a random averaging operator. 
To be more precise, let \ $0 < \delta
\ll 1$ be determined later, and denote by
\begin{align} 
\label{LN}
L_N = L(N) := \max \{ L \in 2^{\mathbb{Z}} ; L < N^{1 - \delta} \}.
\end{align}

\noi 
We define $\psi_{N,L_N}$ to be the solution to the linear equation 
\begin{equation}
  \label{Eqn:method}  
    \psi_{N, L_N} (t) = \chi (t) F_N - 2 i  \chi (t) \int^t_0 \Pi_N
    \mathcal{M} (u_{L_N}, u_{L_N}, \psi_{N, L_N}) (t') d t'  ,
\end{equation}

\noi 
which is expected to consist of bad high-low interactions from $y_N$.
We note that the local-in-time existence of solutions to \eqref{Eqn:method} is guaranteed by a fixed point argument, as the norm of the random averaging operator \eqref{Eqn:l} can be made small provided $T \ll 1$; see \eqref{smallP}.
We then further decompose $y_N$ into
\begin{align}
    \label{zN} 
    y_N =  \psi_{N, L_N} + z_N .
\end{align}

\noi 
It is, therefore, expected that the remainder $z_N$ behaves
better than $y_N$. 
From \eqref{Eqn:yN}, \eqref{LN}, \eqref{Eqn:method}, and \eqref{zN}, we note that the remainder
$z_N$ satisfies
\begin{equation}
  \label{Eqn:newzN} 
  \begin{split}
    & z_N (t) = - i  \sum_{N_{\max} = N_{\med} = N} \chi (t) \int^t_0 \Pi_N
    \mathcal{M} (y_{N_1}, y_{N_2}, y_{N_3}) (t') d t'\\
    & \hspace{1.6cm} - 2 i  \sum_{N_1, N_2 \le  N / 2} \chi (t)
    \int^t_0 \Pi_N \mathcal{M} (y_{N_1}, y_{N_2}, z_N) (t') d t'\\
    & \hspace{1.6cm} - 2 i  \sum_{L_N <  N_{\med} \le  N / 2} \chi (t) \int^t_0
    \Pi_N \mathcal{M} (y_{N_1}, y_{N_2}, \psi_{N, L_N}) (t') d t'\\
    & \hspace{1.6cm} - i  \sum_{N_1, N_3 \le  N / 2} \chi (t)
    \int^t_0 \Pi_N \mathcal{M} (y_{N_1}, z_N, y_{N_3}) (t') d t'\\
    & \hspace{1.6cm} - i  \sum_{N_1, N_3 \le  N / 2} \chi (t)
    \int^t_0 \Pi_N \mathcal{M} (y_{N_1}, \psi_{N, L_N}, y_{N_3}) (t')
    d t'\\
    & \hspace{1.6cm} - i  \sum_{N_{\max} \le  N / 2} \chi (t)
    \int^t_0 \Delta_N \mathcal{M} (y_{N_1}, y_{N_2}, y_{N_3}) (t') d
    t'\\
    & \hspace{1.6cm} + i  \chi (t) \int^t_0 (\Pi_N \mathcal{R} (u_N) -
    \Pi_{N / 2} \mathcal{R} (u_{N / 2})) (t') d t' .
  \end{split}
\end{equation}

\noi 
The equation \eqref{Eqn:newzN} will be our main concern when constructing $z_N$.
The main advantage of this formulation is the appearance of the lower bounds
of $N_{\med}$ in the third term of the right-hand side
of \eqref{Eqn:newzN}, i.e. $L_N <  N_{\med} \le  N / 2$. Therefore, there is essentially no low-low-high interaction from this formulation. This lower bound of $N_{\med}$ comes from the removal of $\psi_{N, L_N}$ from $y_N$.  
In what follows, we shall construct $z_N$ in an induction manner. See Section \ref{Sec:proof} for further discussion.

\subsubsection{The random averaging operator}
Let us have a closer look at $\psi_{N,L_N}$ defined in \eqref{Eqn:method}.
For any $0 \le   L \le  N / 2$ of dyadic $L \in 2^{\mathbb{Z}_{\ge 
0}} \cup \{ 0 \}$, consider the linear equation for $\Psi$:
\begin{equation}
  \label{Eqn:l+} 
  \partial_t \Psi (t) = - 2 i  \Pi_N \mathcal{M} (u_L,
  u_L, \Psi) (t),
\end{equation}
where $u_L$ are solutions to \eqref{Duhamel_N} with $N$ being replaced by $L$, and $\mathcal{M}$ is the trilinear form given in \eqref{Eqn:M}. We remark that $u_L (t)$, the solution to \eqref{FNLS2_N}, uniquely exists for $t\in \R$. 
If \eqref{Eqn:l+} has initial
data $\Psi (0) = \Psi_0$, then we can rewrite \eqref{Eqn:l+} as
\begin{equation}
  \label{Eqn:l+1} 
  \Psi (t) - \chi (t) \Psi_0 = \mathcal{P}^{N, L} [\Psi](t) ,
\end{equation}
for $|t| \le T$,
where $\mathcal{P}^{N, L} : X^b_N (J) \to X^b_N (J)$, for $J = [-T,T]$, are the linear operators defined by
\begin{equation}
  \label{Eqn:l} 
  \mathcal{P}^{N, L} [\Psi](t) = - 2 i  \chi (t) \int^t_0
  \Pi_N \mathcal{M} (u_L, u_L, \Psi) (t') d t' .
\end{equation}

\noi 
Here the space $X^b_N (J)$ is the projection of $X^b (J)$ to its finite Fourier mode, i.e. $\jb{k} \le N$.
Here $\jb{k}$ is defined in \eqref{Eqn:br}.
The operator $\mathcal P^{N,L}$ is known as the random averaging operator, which roughly acts as averaging over low-frequency objects. 
Once the operator $\mathcal{P}^{N, L}$ is properly controlled, as in Definition \ref{Def:LocalM}, we can solve \eqref{Eqn:l+1} and get
\begin{equation}
  \label{Eqn:l+2} 
  \Psi (t) = \mathcal{H}^{N, L}   [\chi (\cdot) \Psi_0] (t)  ,
\end{equation}
where
\begin{equation}
  \label{Eqn:calH} 
  \mathcal{H}^{N, L} := (1 - \mathcal{P}^{N, L})^{- 1} .
\end{equation}

\noi 
We denote $H_{k k'}^{N, L}$ the kernel of $\mathcal{H}^{N, L} $ in the following sense, for a function $u:\R \times \T \to \C$,
\begin{equation*}
  \mathcal F_t (\mathcal{H}^{N, L} [u])_k (\tau) = \int_\R \sum_{k'} \ft{H_{k k'}^{N, L}} (\tau, \tau') \ft u_{k'} (\tau') d\tau'.
\end{equation*}
Then, the solution \eqref{Eqn:l+2} can be expressed as
\[ 
\begin{split}
  \Psi (t,x) & = \mathcal{H}^{N, L} [\chi (\cdot) \Psi_0] (t,x) \\ 
  & = c \int_\R \sum_k 
 \bigg[ \sum_{k'}  \Big( \int_\R  \ft{H_{k k'}^{N, L}} (\tau, \tau') \ft \chi(\tau') d\tau'\Big) (\Psi_0)_{k'} \bigg]e^{i  (k x + t\tau)} d\tau. 
\end{split}
\]
In what follows, 
we set the initial data $\Psi_0 = F_N$ defined in 
\eqref{Eqn:FN}, and we denote $\psi_{N, L}$ to be the corresponding solution, i.e.
\begin{equation}
  \label{Eqn:psiNLdef} 
  \begin{split}
  \psi_{N, L} & (t,x) = \mathcal{H}^{N, L} [\chi (\cdot)  F_N] (t,x) \\ 
  & = c \int_\R \sum_k 
 \bigg[ \sum_{\frac{N}{2} < \jb{k'} \le  N}  \Big( \int_\R  \ft{H_{k k'}^{N, L}} (\tau, \tau') \ft \chi(\tau') d\tau'\Big) \frac{g_{k'}}{\jbb{ k'}^{\alpha / 2}} \bigg]e^{i  (k x + t\tau)} d\tau.
  \end{split}
\end{equation}

\noi  
From \eqref{Eqn:l+1} and \eqref{Eqn:l}, we note that $\psi_{N, L}$ defined in
\eqref{Eqn:psiNLdef} coincides with that in \eqref{Eqn:method} when $L$ is
chosen to be $L_N$ in \eqref{LN}.
We define the operator 
\begin{equation}
  \label{Eqn:hNL}  
    \mathfrak{h}^{N, L}  :=  {\mathcal{H}}^{N, L}   - {\mathcal{H}}^{N, L/2}   .
\end{equation}
By denoting by $ \zeta^{N, L} := \psi_{N, L} - \psi_{N, L / 2} $,
we have
\begin{align}
    \label{zetaN}
    \begin{split} 
    \zeta^{N, L} &:= \psi_{N, L} - \psi_{N, L / 2}  
    =  {\mathfrak{h}}^{N, L} [\chi(\cdot) F_N]  .
    \end{split} 
\end{align}

We also define the flow map version of the operator $\mathcal H^{N,L}$, denoted by $\wt{\mathcal{H}}^{N, L} (t)$ and defined as
\begin{align}
    \wt{\mathcal{H}}^{N, L} (t) [\phi] (x) := \mathcal{H}^{N, L} [ \chi (\cdot) \phi] (t,x) = \sum_k e^{ikx} \sum_{k'} \wt{{H}}^{N, L}_{kk'} (t) \phi_{k'},
    \label{wtH}
\end{align}

\noi 
where $\{\wt{{H}}^{N, L}_{kk'} (t)\}_{kk'}$ is the kernel of the operator $\wt{\mathcal{H}}^{N, L} (t)$. Then denote by
\begin{equation}
  \label{wthNL}  
    \wt{\mathfrak{h}}^{N, L} (t):= \wt{\mathcal{H}}^{N, L}  (t) - \wt{\mathcal{H}}^{N, L/2}  (t) ,
\end{equation}

\noi 
whose kernel is given by 
\begin{equation}
  \label{Eqn:keh} 
  h_{k k'}^{N, L} (t)  =  \wt{{H}}^{N, L}_{kk'} (t) - \wt{{H}}^{N, L/2}_{kk'} (t) .
\end{equation}

\noi 
One of the key observations is that $ h_{k k'}^{N, L}$ and $\wt H_{k k'}^{N, L} $ are Borel functions of $(g_k (\omega))_{\jb{k} \le L}$, and are thus independent
of the Gaussians $F_N$ given by \eqref{Eqn:FN} and \eqref{GFF}. 
Such an independent structure enables one to use probabilistic tools to exploit the high-moment cancellation structure of $F_N$; see Proposition \ref{Prop:te}.

\subsubsection{The solution structure} 
\label{SUBS:SS}
We note that, from \eqref{zN}, \eqref{zetaN}, and the fact that $\psi_{N,1/2} = F_N$, 
we have the following formal expansion of $y_N$ on $[-T,T]$. 
\begin{equation}
  \label{Eqn:deco}  
  \begin{split}
    y_N (t) & = \chi (t) \psi_{N, L_N} (t) + z_N (t)\\
    & = \chi (t) F_N + \sum_{1 \le  L \le  L_N} \zeta^{N, L} (t) + z_N (t)\\
    & = \chi (t) F_N + \sum_{1 \le  L \le  L_N}   {\mathfrak{h}}^{N, L} [\chi(\cdot) F_N]
    (t) + z_N (t)\\
    & = \chi (t) F_N + \sum_{1 \le  L \le  L_N}  \wt {\mathfrak{h}}^{N, L} [ F_N]
    (t) + z_N (t).
  \end{split}
\end{equation}

\noi 
We will construct $z_N$ and $\mathfrak{h}^{N, L}$, for all $N \ge \frac12$ and $1 \le L \le L_N$, 
by an induction argument.
As a matter of fact, the random averaging operator $\mathfrak{h}^{N, L}$ in \eqref{Eqn:deco}  for $L\le L_N$ only depend on $(u_{L})_{L\ll N}$, and thus $(y_{L})_{ L\ll N}$ due to \eqref{Eqn:yNdef}.
Then, from \eqref{zN}, we see that $(y_{L})_{ L\ll N}$ are again determined by $(z_{L'})_{1 \le L' \le L}$ and $(\mathfrak h^{L',R'})_{L'\le L, R' < (L')^{1-\dl}}$. In all, we see that $\mathfrak{h}^{N, L}$ is determined by $(z_{L'})_{L' \ll N}$ and $(\mathfrak h^{L',R'})_{L', R'\ll N}$. On the other hand, from \eqref{Eqn:newzN}, we see that $z_N$ depends on $(z_{N'})_{N' \le N}$ and $(\mathfrak h^{N',L})_{N'\le N, L < (N')^{1-\dl}}$. See Definition \ref{Def:LocalM} for further details. 

Finally, from \eqref{Eqn:yNdef}, \eqref{Eqn:deco}  , and the fact that $u_{1/2} = F_{1/2} = 0$, we arrive at:
\begin{equation}
  \label{Eqn:Ansatz} 
  u (t) = \chi(t) u_0^{\omega} + \sum_{N \in
  2^{\mathbb{Z}_{\ge  0}}} \sum_{1 \le  L \le  L_N}
  \wt{\mathfrak{h}}^{N, L} [F_N] (t) + \sum_{N \in 2^{\mathbb{Z}_{\ge  0}}}
  z_N (t),
\end{equation}

\noi 
provided the summations converge in a proper sense. 
From \eqref{Eqn:Ansatz}, we see that the expected solution $u$ in Theorem \ref{THM:main} has the following structure
    \begin{equation*}  
      u = \text{random linear term} + 
      \text{random averaging term}
      + 
      \text{smoother term}
      .
    \end{equation*}

\noi 
We shall construct the smoother term in $H^{\frac12+} (\T)$, much smoother than the random linear and random averaging terms, which lie in $H^{\frac{\al -1}2-} (\T)$, as $\al$ is close to $1$.
The random averaging term bridges the random linear term and the smoother term. 
On the one hand, it records quasilinear roughness from high-low interactions in $u$, thus guaranteeing that the ``smoother term" has higher regularity. On the other hand, it preserves the independent structure of the initial data, such that we may apply probabilistic arguments. 

\medskip 

\text{{\bfseries{Notation conventions.}}} Throughout the paper, we will denote
by $A \lesssim B$, $A \lesssim_{a_1, \cdots, a_n} B$ and $A \ll B$ any
estimate of the form $A \le  C B$, $A \le  C_{a_1, \cdots, a_n} B$ and
$A \le  c B$ respectively, where $C >0$ is an absolute constant,
$C_{a_1, \cdots, a_n} >0$ is a constant that only depends on parameters $a_1,
\cdots, a_n$, and $c > 0$ is a small constant. We also denote by $A \sim_{a_1, \cdots, a_n} B$ if $A \lesssim_{a_1, \cdots, a_n} B$ \& $B \lesssim_{a_1, \cdots, a_n} A$ hold. We also use $a +$ (and $a -$) to mean $a + \varepsilon$ (and $a -
\varepsilon$, respectively) for arbitrarily small $\varepsilon > 0$. 
We denote by
\begin{equation}
  \label{Eqn:br} 
  \begin{cases}
    \langle k \rangle := \sqrt{1 + | k |^2} ;\\
    \jbb{k} := (1 + | k |^{\alpha})^{1 / \alpha} .
  \end{cases}
\end{equation}
We note that $\jb{k} = \jbb{k}$ for $\al = 2$.
We denote
by $\bar{z}$ the complex conjugate of $z$. We also use $(\mathcal F_x f) (k) $ to denote the
Fourier coefficients of $f$, i.e.
\begin{align*}
  (\mathcal{F}_x f) (k) = \int_{\T} f(x) e^{-ikx} dx,
\end{align*}  

\noi 
which will be abbreviated as $f_k$, i.e. $f_k = (\mathcal{F}_x f) (k)$. We may also abuse notations and write $f = (f_k)$.
The Fourier transform is invertible such that
\begin{align*} 
    f(x) = \frac1{2\pi} \sum_{k \in \Z}  f_k e^{ikx}.
\end{align*}  

\noi 
Let $u = u(t,x)$ be a space-time distribution. We denote $u_k (t) : = \mathcal{F}_x (u(\cdot,t)) (k) $ to be the spatial Fourier coefficient of $u(\cdot, t)$. Given $k\in \Z$,  the Fourier
coefficient $u_k$ is still a function of time. 
We use $\widehat{u_k}$ to represent the temporal Fourier transform of $u_k$ only, i.e.
\[ 
  \widehat{u_k} (\tau) := \mathcal{F}_t (u_k) (\tau) = \int_\R u_k (t) e^{ - it\tau} dt. 
\]
We use $t$, $t'$ to
denote temporal variables and $\tau$, $\tau'$, $\tau_i$ to denote Fourier
variables of time. 
When we work with space-time function spaces, we use short-hand notations such
as $C_T H_x^s = C ([0, T] ; H^s (\mathbb{T}))$. 

Let $N_i \ge  \frac12$ for $i
= 1, 2, 3$ be dyadic numbers. We use $N_{\max}$,
$N_{\med}$, and $N_{\min}$ to denote the largest,
second largest, and smallest number among $N_1$, $N_2$, and $N_3$,
respectively. 
Let $k_i$ be the frequency associated with the dyadic number $N_i$ in the sense that $\jb{k_i} \le N_i$.
Please note that we only attach an upper bound to $\jb{k_i}$ here. 
If the object under consideration is of type (G)\footnote{See Subsection \ref{Sub:refo} for definitions of type (G), (C), and (D).}, then we may further restrict $\jb{k_i}$ to the dyadic block $N_i/2 < \jb{k_i} \le N_i$, which is not necessarily true for type (C) and (D) objects. 
Define $k_{\max}$ to be the $k_i$ such that $N_i = N_{\max}$.
When $N_{\max} = N_{\med}$, the roles of $k_{\max} $ and $k_{\med}$ are interchangeable in this context; i.e. if $N_1 = N_2 > N_3$, then we can designate either $k_1$ or $k_2$ as be $k_{\max}$ and the other to be $k_{\med}$ without changing other parts of the proof.
Similarly, we define $k_{\med}$ and $k_{\min}$. We use $a \wedge b$ and $a \vee b$ to denote $\min (a, b)$ and
$\max (a, b)$, respectively. We denote by $\#S$ or $|S|$ the cardinal number of the set
$S$. 
We also need the sharp cutoff 
\begin{align} 
\label{ind}
\ind_{B} (x)
= \begin{cases}
    1, \quad & x\in B;\\
    0, & \textup{otherwise}.
\end{cases}
\end{align}

Let $X_i$ for $i = 1,2$ be normed space equipped with norm $\|\cdot \|_{X_i}$. 
Then we define
\[
\| \cdot \|_{X_1 \cap X_2} = \| \cdot \|_{X_1 }  + \| \cdot \|_{ X_2} .
\]


\noi 
Similarly, we can define the intersection of more than two normed spaces.

For any $N$, we denote by $\mathcal{B}_{\le  N}$ the
$\sigma$-algebra generated by the random variables $( g_k )_{\jb{k}
\le  N}$.

The rest of the paper is organized as follows. In Section
\ref{Sec:preliminary}, we recall some basic definitions and tools and we also
prove some crucial tensor estimates. In Section \ref{Sec:proof}, we prove the
main theorem. In Sections \ref{Sec:raos} and \ref{Sec:rems}, we prove Proposition
\ref{PROP:main}. In Section \ref{Sec:alphagtr1}, we discuss the
sharpness of the condition $\alpha > 1$.

\section{Preliminary}\label{Sec:preliminary}

In this section, we collect some results that will be used throughout the paper.

\subsection{Norms}
Let $A$ be a finite index set.
We denote by $k_A = \{k_j; j\in A\} \in \Z^A$. 
A tensor $H =  H_{k_A}$ is a function $\Z^A \to \C$ with variables $k_A$. In what follows, we assume $H_{k_A} \in \l^2_{k_A}$, i.e.
\[
\| H\|_{\l^2_{A}} : = \| H_{k_A}\|_{k_{A}} = \bigg( \sum_{k_A} |H_{k_A}|^2 \bigg)^{\frac12} < \infty.
\]
The support of $H$ is the set of $k_A$ such that $H_{k_A} \neq 0$.
If $(B, C)$ is a partition of $A$, namely $B \cap C = \emptyset$ and $B \cup
C = A$, we define the tensor norm $\| H \|_{k_B \to  k_C}$ such that
\begin{equation}
  \label{Eqn:opnorm} 
  \| H \|_{k_B \to  k_C}^2 = \sup \bigg\{ \sum_{k_C}
  \Big| \sum_{k_B} H_{k_A} z_{k_B} \Big|^2 ; \sum_{k_B} | z_{k_B} |^2 = 1
  \bigg\} .
\end{equation}
Let us remark that
\begin{equation} 
\label{CSinq}
  \| H \|_{k_B \to  k_C} = \| H \|_{k_C \to 
  k_B} \le \| H \|_{k_A}.
\end{equation}

\noi 
Let $f \in \l^2_{B} $ be a function $f : \Z^B \to \C$ such that $\|f({k_B})\|_{\l^2_B} < \infty$. 
We may use $H$ to define an operator, still denoted by $H : \l^2_{k_B} \to \l^2_{k_C}$, given by
\[
H[f] (k_C) = \sum_{k_B} H_{k_A} f ({k_B}).
\]
Then the tensor norm $\| H \|_{k_B \to  k_C}$ defined in \eqref{Eqn:opnorm} can be recast as an operator norm.
\[
\| H \|_{k_B \to  k_C} = \| H \|_{\l_{k_B}^2 \to  \l_{k_C}^2} : = \sup_{\|f\|_{\l^2_{k_B}} =1} \| H[f]\|_{\l^2_{k_C}} .
\]

\noi 
We remark that if $C = \emptyset$, then $B = A$ and $\| H \|_{k_B \to  k_C} = \| H \|_{k_A}$.

For a space-time function $u \in C(\R; L^2 (\T))$, we have
\[ 
u (t,x) = \frac1{2\pi} \sum_{k \in \mathbb{Z}} u_k (t)
e^{i  k x}.
\]

\noi 
We define the Fourier restriction norm 
\[
\| u \|_{X^b} := \bigg( \int_{\mathbb{R}} \langle \tau \rangle^{2 b} \|
    \widehat{u_k} (\tau) \|_k^2 d \tau \bigg)^{\frac12}  = \bigg(\int_{\mathbb{R}} \langle \tau \rangle^{2 b} \sum_{k} |
    \widehat{u_k} (\tau) |^2 d \tau\bigg)^{\frac12}.
\]

\noi 
For a time dependent operator $\mathcal H (t)$ given by its  kernel $\mathcal{H} (t) = \{ H_{k k'} (t) \}_{k
k'}$, i.e.
\[
\mathcal{H} (t) [f] = \frac1{2\pi} \sum_{k \in \Z} \Big( \sum_{k'\in \Z}H_{k k'} (t) f_{k'} \Big) e^{ikx},
\]

\noi 
we define the operator norm
\begin{align}
    \label{Yb}
    \begin{split}
\| \mathcal{H} \|_{Y^b} & := \big\| \langle \tau \rangle^{b} \widehat{H_{k k'}} (\tau)  \big\|_{\ell_k^2 L_{\tau}^2 \to \ell_{k'}^2} \\
& = \sup \bigg\{ \bigg( \int_{\mathbb{R}} \langle \tau \rangle^{2 b} \sum_{k}
\Big| \sum_{k'} \ft{H_{kk'}} (\tau) f_{k'} \Big|^2 d\tau\bigg)^{\frac12}; \,\, \sum_{k'} |f_{k'}|^2 = 1 \bigg\} \\
& = \sup_{\|f\|_{L^2} = 1} \big\|\mathcal H (t) [f] \big\|_{X^b}, 
\end{split}
\end{align} 

\noi 
and the Hilbert–Schmidt norm of the operator $\mathcal H$,
\begin{align}
    \label{Zb}
    \| \mathcal{H} \|_{Z^b} := \bigg( \int_{\mathbb{R}} \langle \tau \rangle^{2
    b} \| \widehat{H_{k k'}} (\tau) \|_{k k'}^2 d \tau \bigg)^{\frac12} = \bigg( \int_{\mathbb{R}} \langle \tau \rangle^{2 b} \sum_{k,k'} \big|
     \widehat{H_{k k'}} (\tau) \big|^2 d \tau\bigg)^{\frac12}. 
\end{align} 

\noi 
Now we consider an operator $\mathcal P$ over space-time functions given by  
\[
\mathcal P [u] (t,x) = \frac1{(2\pi)^2} \sum_{k \in \Z} \int_{\R} \bigg( \sum_{k'\in \Z} \int_\R \ft {P_{k k'}} (\tau,\tau') \ft u_{k'} (\tau') d\tau' \bigg)  e^{i(t\tau + kx)} d\tau,
\]

\noi 
where $\ft {P_{k k'}} (\tau,\tau')$ is the double temporal Fourier transform given by
\[
\ft {P_{k k'}} (\tau,\tau') = \int_{\R^2} P_{k k'} (t, t') e^{-i(t\tau + t' \tau')} dt dt'.
\]

\noi 
We define the operator norm of $\mathcal P$ by
\begin{align}
    \label{Ybb}
    \begin{split}
    \| \mathcal{P} \|_{Y^{b, b'}} & := \big\| \langle \tau \rangle^{ b}
    \langle \tau' \rangle^{-  b'} \widehat{P_{k k'}} (\tau, \tau')
    \big\|_{\ell_k^2 L_{\tau}^2 \to  \ell_{k'}^2 L_{\tau'}^2} \\ 
    & = \sup_{\|u\|_{X^{b'}} = 1} \frac1{(2\pi)^2} \bigg( \int_{\mathbb{R}} \langle \tau \rangle^{2 b} \sum_{k} \bigg|  \sum_{k'} \int_\R  \widehat{P_{k k'}} (\tau,\tau') \ft u_{k'} (\tau') d\tau' \bigg|^2 d \tau \bigg)^{\frac12} \\
     & = \sup_{\|u\|_{X^{b'}} = 1} \big\| \mathcal P [u] \big\|_{X^b},
    \end{split}
\end{align}

\noi 
and the Hilbert–Schmidt norm of the operator $\mathcal P$ 

\begin{align} 
\label{Zbb} 
\begin{split} 
    \| \mathcal{P} \|_{Z^{b, b'}} 
    & = \bigg( \int_{\mathbb{R}^2} \langle \tau
    \rangle^{2 b} \langle \tau' \rangle^{- 2 b'} \sum_{k,k'} \big| \widehat{P_{k k'}} (\tau,
    \tau') \big|^2 d \tau d \tau' \bigg)^{\frac12}.
    \end{split}
\end{align}

For any finite interval $I$, define the corresponding localized norms
\[ 
  \| u \|_{X^b (I)} := \inf \left\{ \| v \|_{X^b} ; v = u \text{ on } I
  \right\} 
\]
and similarly define $Y^{b, b'} (I)$ and $Z^{b, b'} (I)$. By abusing
notations, we will call the above $v$ an  extension of $u$,
and still denotes $v$ by $u$. We will use $X^b$ to denote $X^b (I)$ for simplicity unless otherwise specified. 

To conclude this subsection, we record some basic estimates from \cite{DNY1,DNY3}. Define the original and truncated integral
operators
\begin{align}
    \label{Duhamel_I}
  \mathcal{I} v (t) = \int^t_0 v (t') d t', \quad
  \mathcal{I}_{\eta} v (t) = \eta (t) \int^t_0 \eta (t') v (t') d t', 
\end{align}  

\noi 
where $\eta(t)$ is a smooth cutoff function, such that $\eta(t) = 1$ for $|t|\le 1$; and $\eta(t) = 0$ for $|t|\ge 2$. 
By comparing equations \eqref{Duhamel} with \eqref{Duhamel_N}, 
we observe that the Duhamel operator from \eqref{Duhamel} roughly reduces to the integral operator $\mathcal I$ defined in \eqref{Duhamel_I} under the transformation $u (t)  \to
e^{i  t {\rm D}_x^{\alpha}} u (t)$. 
Therefore, the operators $\mathcal I$ and $\mathcal I_\eta$ are also referred to as Duhamel operator and truncated Duhamel operator, respectively. 
See \cite[Section 4.1]{DNY3} for further details.
We have the following kernel estimates. See \cite[Lemma 3.1]{DNY1} and \cite[Lemma 4.1]{DNY3} for a proof.

\begin{lemma}
  \label{LEM:kerneles}
  We have the formula
  \begin{equation}
    \label{Eqn:Ik} 
    \widehat{\mathcal{I}_{\eta} v} (\lambda) =
    \int_{\mathbb{R}} \mathcal{K} (\lambda, \lambda') \hat{v} (\lambda')
    d \lambda',
  \end{equation}
  where the kernel $\mathcal{K}$ satisfies the following pointwise estimates
  \begin{equation}
    \label{Eqn:bk} 
    | \mathcal{K} (\ld,\ld') | + | \partial_{\lambda, \lambda'}
    \mathcal{K} (\ld,\ld')  | \lesssim \bigg( \frac{1}{\langle \lambda \rangle^3} +
    \frac{1}{\langle \lambda - \lambda' \rangle^3} \bigg) \frac{1}{\langle
    \lambda' \rangle} \lesssim \frac{1}{\langle \lambda \rangle \langle
    \lambda - \lambda' \rangle} .
  \end{equation}
\end{lemma}

We also need the following short-time bound on Fourier restriction norm from \cite[Proposition 2.7]{DNY2}.

\begin{proposition}
  \label{Prop:stb}{{\em (Short time bounds).\/}} 
  Let $\eta$ be any Schwartz
  function. Recall that $\eta_T (t) = \eta (T^{- 1} t)$ for $T \ll
  1$. Then for any $u \in X^{b_1}$, we have
  \[ 
    \| \eta_T \cdot u \|_{X^b} \lesssim T^{b_1 - b} \| u \|_{X^{b_1}}, 
  \]
  provided either $0 < b \le  b_1 < 1 / 2$, or $u_k (0) = 0$ and $1 / 2 <
  b \le  b_1 < 1$.
\end{proposition}

\subsection{Counting estimates}
\label{SUB:count}
Counting estimates are crucial in our analysis.
In this subsection,
we shall show  some almost sharp counting estimates of solutions to some fractional Diophantine equations. 
To be more precise,
given dyadic numbers $N$, $N_1$, $N_2$, $N_3 > 0$ and some fixed number $m \in \R$, define the set
  \begin{equation}
    \label{Eqn:S} 
    S = \left\{
    \begin{aligned} 
      & (k, k_1, k_2, k_3) \in \mathbb{Z}^4 , \quad k_2 \not\in \{ k_1, k_3 \}, \\  
      & k = k_1 - k_2 + k_3 ,\\
      & |{k_1}|^{\alpha} - |{k_2}|^{\alpha} + |{k_3}|^{\alpha} - |{k}|^{\alpha}
      = m + O (1), \\
      & | k | \le  N, \,\, | k_j | \le  N_j\, \textup{ for } j \in \{ 1, 2, 3
      \} ,
    \end{aligned}
    \right.
  \end{equation}

  \noi 
  which is a subset of the hyperplane $k = k_1 - k_2 + k_3 $ in $\Z^4$.
  Denote by $S_k$ the set of $(k_1, k_2, k_3) \in \Z^3$ such that $(k, k_1, k_2, k_3) \in S$ when $k$ is fixed.
  It is easy to see that \[
  S_k \subset \Gamma (k),
  \]
  
  \noi  
  where $\Gamma (k)$ is the hyperplane of $\Z^3$ defined in \eqref{hyperplane}. Similarly, we can define $S_{k_1}$, $S_{k k_1}$, $S_{kk_1k_2}$, and etc. 
  For a set $A$, we use $|A|$ or $\# A$ to denote its cardinality. If $A$ is a finite set, then $|A|$ denotes the number of elements of $A$.
  The main purpose of this subsection is to count the  cardinality of these sets defined above.

\medskip 

We start with a basic counting estimate from {\cite[Lemma 2.6]{ST21}}.
\begin{lemma}
    \label{LEM:basic}
    Let $I$, J be two intervals and $\phi$ be a
  real-valued  $C^1$ function defined on $I$. Then, we have
  \[ 
    \# \{ k \in I \cap \mathbb{Z}; \phi (k) \in J \} \lesssim 1 + \frac{| J
    |}{\operatorname{ inf}\limits_{\xi \in I} | \phi' (\xi) |}. 
  \]
\end{lemma}

In this paper, we will focus on a special function $\phi$. Let 
\begin{align}
    \label{phib}
    \phi_{b,\pm} (x) = | x |^{\alpha} \pm | x - b |^{\alpha},
\end{align}

\noi 
where $\al \in (1,2)$ and $b \in \R$. Then, we see that $\phi_{b,\pm} \in C^1$ for $\al > 1$, and we have
\begin{align}
    \label{phib'}
    \phi'_{b, \pm} (x) = \alpha \big(\ensuremath{\operatorname{sgn}} (x) | x
    |^{\alpha - 1} \pm \ensuremath{\operatorname{sgn}} (x - b) | x - b |^{\alpha
    - 1} \big).
\end{align}
Then, we have the following elementary estimates.

\begin{lemma}
  \label{LEM:c}
  Let $\phi_{b,\pm} $ be as in \eqref{phib}.   
  \begin{itemize}
  \item[(i)] 
  Let $\al \in (1,2)$.
  Then, we have
    \begin{equation}
    \label{Eqn:lbm} | \phi'_{b, -} (x) | \gtrsim_{\alpha} \min ( | b |  | x |^{\alpha - 2}, |b|^{\al -1}) 
  \end{equation}
  provided $x \neq 0$.

  \item[(ii)]  
  Let $\al \in (1,2)$ and $|b| \ge 1$. Then, we have
  \begin{equation}
    \label{Eqn:lowb+}
    | \phi_{b, +}' (x) | \gtrsim_{\alpha}  |b|^{\alpha-1}
\end{equation} 
for $|2x-b| \ges |b| $.
  For $|2x-b| \ll |b|$, if we further assume that $ | 2 x - b  | \ges |b|^{1-\frac\al2} $, then we have
  \begin{equation}
    \label{Eqn:sharp}
    | \phi_{b, +}' (x) | \gtrsim_{\alpha} |b|^{\frac\al2 -1}.
  \end{equation} 
  \end{itemize} 
\end{lemma}

\begin{proof}
  (i) We first consider $b < 0$, for which we use \eqref{phib'} to get
  \[ 
    \phi'_{b, -} (x) = 
    \begin{cases}
      \alpha (- | x |^{\alpha - 1} + | x - b |^{\alpha - 1}), & x \in (-
      \infty, b];\\
      \alpha (- | x |^{\alpha - 1} - | x - b |^{\alpha - 1}), & x \in (b, 0)
      ;\\
      \alpha (| x |^{\alpha - 1} - | x - b |^{\alpha - 1}), & x \in (0, +
      \infty).
   \end{cases} 
  \]
  
  \noi 
  When $x \in (b, 0)$, it is easy to see that 
  \begin{align*}
    | \phi'_{b, -} (x) | = \alpha (| x |^{\alpha - 1} + | x - b |^{\alpha - 1}) \gtrsim_{\alpha} |b|^{\al -1}, 
  \end{align*}

  \noi  
  due to the fact that $\al > 1$, which is sufficient for \eqref{Eqn:lbm}. 
  We turn to the case when $x \in (- \infty, b] \cup (0, + \infty)$, where we have that  $x$ and $x-b$ have the same sign.
  Therefore, by the Fundamental Theorem of Calculus, we have
  \begin{align} 
  \label{phib-2}
    \begin{split}
      & | \phi'_{b, -} (x) | = \alpha (\alpha - 1) \bigg| \int_b^0 | x - b + s
      |^{\alpha - 2} d s \bigg|\\
      & \quad \gtrsim_{\alpha} | b | \cdot \min (| x |^{\alpha - 2}, | x - b
      |^{\alpha - 2}) ,
    \end{split} 
  \end{align}
  where we used the fact that $\al - 2 < 0$. If $x \in (-\infty, b]$, then $|x-b| \le |x|$ and thus $\min (| x |^{\alpha - 2}, | x - b |^{\alpha - 2}) = | x |^{\alpha - 2}$, which, together with \eqref{phib-2}, gives \eqref{Eqn:lbm}. In the following, let us assume that $x \in (0, \infty)$. Then we have $|x-b| \ge |x|$. If we further assume that $|x-b| \sim |x|$, then  $\min (| x |^{\alpha - 2}, | x - b |^{\alpha - 2}) \ges | x |^{\alpha - 2}$, which is again sufficient for \eqref{Eqn:lbm}.
  It remains to consider the case $|x-b| \gg |x|$, which implies $|x| \ll |b|$ and thus $|x-b| \sim |b|$. Therefore, we have $\min (| x |^{\alpha - 2}, | x - b |^{\alpha - 2}) = | x -b|^{\alpha - 2} \sim |b|^{\al - 2}$, which, together with \eqref{phib-2}, gives \eqref{Eqn:lbm} again. 
  Thus we finish the proof of \eqref{Eqn:lbm} for $b <0$. The proof for the case when $b > 0$ is similar and therefore omitted.

  (ii) We only consider the case $b <0$, as the proof for the case $b >0$ is similar. 
  From \eqref{phib'}, we have
  \begin{align} 
  \label{phib+1}
    \phi'_{b, +} (x) = 
    \begin{cases}
      \alpha (- | x |^{\alpha - 1} - | x - b |^{\alpha - 1}), & x \in (-
      \infty, b];\\
      \alpha (- | x |^{\alpha - 1} + | x - b |^{\alpha - 1}), & x \in (b, 0)
      ;\\
      \alpha (| x |^{\alpha - 1} + | x - b |^{\alpha - 1}), & x \in (0, +
      \infty).
   \end{cases} 
  \end{align}
  When $x \in (- \infty, b] \cup (0, \infty)$, from \eqref{phib+1} we have
  \begin{equation*} 
    | \phi_{b, +}' (x) | = \alpha | | x |^{\alpha - 1} + | b -
    x |^{\alpha - 1} | \gtrsim_{\alpha} |b|^{\al -1},  
  \end{equation*}

  \noi 
  since $\al > 1$.
  From now on, we assume $x \in (b,0]$.
  From \eqref{phib+1}, we have
  \begin{align}
      \label{phib+2}
      \begin{split}
      |\phi_{b,+}' (x)| & = \alpha \big| | x |^{\alpha - 1} - | x - b |^{\alpha - 1} \big| \\
      & =  {\al}{(\al - 1)} \bigg| \int_x^0 t^{\al -2} dt - \int_0^{x-b} t^{\al -2} dt \bigg| \\ 
      & = {\al}{(\al - 1)} \bigg| \int_{-x}^{x-b} t^{\al -2} dt \bigg|\\
      & \ge  {\al}{(\al - 1)}  | 2x-b| |b|^{\al -2},
      \end{split}
  \end{align}

\noi 
where we used the facts that $|x|, |x-b| \le |b|$. 
It is easy to see that \eqref{phib+2} implies \eqref{Eqn:lowb+}, provided $|2x-b| \ges |b|$.
On the other hand, if $ | 2 x - b  | \ges |b|^{1-\frac\al2} $, we have
\begin{align}
    \label{phib+3}
     {\al}{(\al - 1)}  | 2x-b| |b|^{\al -2} \ges |b|^{\frac\al2 -1},
\end{align}
which, together with \eqref{phib+2}, gives \eqref{Eqn:sharp}.
Thus we finish the proof of \eqref{Eqn:sharp}.
\end{proof}

Now we are ready to show our main counting estimates.

\begin{lemma}
  \label{LEM:counting1} 
  Let $\alpha \in (1,2)$ and $1 \le N_1, N_2, N_3 \le N$. 
  Then we have the following counting estimates
  \begin{align} 
  \label{counting1}
    \begin{split} 
      |S_{k k_1}| & \lesssim (N_2 \wedge N_3)^{2 - \alpha} \jb{k_1 - k}^{- 1} +
      1 ;\\
      |S_{k k_3}| & \lesssim (N_1 \wedge N_2)^{2 - \alpha} \jb{k - k_3}^{- 1} +
      1 ;\\
      |S_{k_1 k_2}| & \lesssim N_3^{2 - \alpha} \jb{k_1 - k_2}^{- 1} + 1 ;\\
      |S_{k_2 k_3}| & \lesssim N_1^{2 - \alpha} \jb{k_2 - k_3}^{- 1} + 1 ,
    \end{split} 
    \end{align}
    
\noi 
where $\jb{k} = (1+|k|^2)^{\frac{1}{2}}$ is given in \eqref{Eqn:br}. 
\end{lemma}

\begin{proof}
We start with the bound for $|S_{kk_1}|$. 
We first note that $k \not\in \{ k_1, k_3 \}$ is equivalent to $k_2 \not\in \{ k_1, k_3 \}$ on the hyperplane $k = k_1 - k_2 + k_3$.
Moreover, on this hyperplane, for fixed $(k,k_1)$ (such that $|k|\le N$, $|k_1|\le N_1$, and $k \neq k_1$), once we further fix $k_3$, then $k_2$ is uniquely determined. Therefore, we have
\begin{align*} 
    \begin{split}
    |S_{k k_1}|  & =   | \{ k_3 \in \mathbb{Z}; k_3 \neq  k , |k_3|\le N_3, | k_1 + k_3 - k | \le N_2, \\
    & \hspace{2cm}| k_3 |^{\alpha} - | k_1 + k_3 - k |^{\alpha} = | k |^{\alpha} - |
      k_1 |^{\alpha} - m + O (1) \} |\\        
     & =  | \{ k_3 \in \mathbb{Z}; k_3 \neq  k , |k_3|\le N_3, | k_1 + k_3 - k | \le N_2, \\
    & \hspace{2cm} \phi_{k - k_1, -} (k_3) = | k |^{\alpha} - | k_1 |^{\alpha} - m + O (1) \} |,
    \end{split}
\end{align*}

\noi 
where $\phi_{k - k_1,-}$ is given in \eqref{phib} with $b = k-k_1$. 
We now apply Lemma \ref{LEM:basic} and then Lemma \ref{LEM:c} to get
\begin{align}
\begin{split}
     |S_{k k_1}|  & \les 1 + \frac{1}{\inf_{|k_3| \les N_3 } |\phi'_{k-k_1,-} (k_3)|} \\ 
    & \les 1 + \big( \min ( \inf_{|k_3| \le N_3} |k-k_1| |k_3|^{\al -2}, |k-k_1|^{\al -1}) \big)^{-1}\\  
    & \les  1 +  \jb{k-k_1}^{-1} N_3^{2-\al}, 
\end{split}
\label{Skk1_2}
\end{align}

\noi 
where $k \neq k_1$ and $\al \in (1,2)$.  

By switching $k_2$ and $k_3$ in the above argument, we have 
\begin{align*} 
    \begin{split}
    |S_{k k_1}|  & =   | \{ k_2 \in \mathbb{Z}; k_2 \neq  k_1 , |k_2|\le N_2, | k_1 - k_2 - k | \le N_3, \\
    & \hspace{2cm}| k_2 |^{\alpha} - | k_1 - k_2 - k |^{\alpha} = | k_1 |^{\alpha} - |
      k |^{\alpha} + m + O (1) \} | \\ 
    & =  | \{ k_2 \in \mathbb{Z}; k_2 \neq  k_1 , |k_2|\le N_2, | k_1 - k_2 - k | \le N_3, \\
    & \hspace{2cm} \phi_{k_1 - k, -} (k_2) = | k_1 |^{\alpha} - |
      k |^{\alpha} + m + O (1) \} |.        
    \end{split}
\end{align*}

\noi 
Then by the same argument as in \eqref{Skk1_2}, we get
\begin{align}
    \label{Skk1_4}
    |S_{k k_1}| \les  1 +  \jb{k-k_1}^{-1} N_2^{2-\al}. 
\end{align}

\noi 
Thus we finish the proof of \eqref{counting1} for the bound $S_{kk_1}$ by combining \eqref{Skk1_2} and \eqref{Skk1_4}.

The proof for the rest of \eqref{counting1} is similar; thus, we omit their details.
\end{proof}

\medskip

We point out that the decay $\jb{k-k_1}^{-1}$ in \eqref{counting1} plays a crucial role in our later analysis. 
We have the following estimates as a consequence of Lemma \ref{LEM:counting1}.

\begin{corollary}
  \label{COR:counting} 
  Let $\alpha \in (1,2)$ and $1 \le N_1, N_2, N_3 \le N$. 
  Then, we have 
  \begin{align} 
  \label{counting2}
    \begin{split}
      |S_k| & \lesssim \min ((N_2 \wedge N_3)^{2 - \alpha} \log N_1  + N_1, (N_1 \wedge N_2)^{2 - \alpha} \log N_3 + N_3) ; \\
      |S_{k_1}| & \lesssim N_3^{2 - \alpha} \log N_2 +
      N_2 ;\\
      |S_{k_2}| & \lesssim \min (N_3^{2 - \alpha}
      \log N_1  + N_1, N_1^{2 - \alpha} \log N_3 
      + N_3) ;\\
      |S_{k_3}| & \lesssim  N_1^{2 - \alpha} \log N_2  +
      N_2 .
    \end{split} 
    \end{align}
\end{corollary}

\begin{proof}
We start with the bound for $|S_k|$. We first note that
  \begin{align} 
  \label{Sk}
    |S_k| \lesssim \min \bigg( \sum_{|k_1| \le  N_1} |S_{k k_1}|,
    {\sum_{|k_3| \le  N_3}} 
    |S_{k k_3}| \bigg) . 
  \end{align}

  \noi 
  By using Lemma \ref{LEM:counting1}, we have
  \begin{align} 
  \label{Sk_1}
    \begin{split}
       \sum_{|k_1| \le  N_1} |S_{k k_1}| & \lesssim (N_2 \wedge N_3)^{2 -
      \alpha} \sum_{|k_1| \le  N_1} \jb{k_1 - k}^{- 1} + \sum_{|k_1| \le  N_1} 1\\
      &  \lesssim (N_2 \wedge N_3)^{2 - \alpha} \log N_1 + N_1,
    \end{split} 
  \end{align}

  \noi 
  and similarly,
  \begin{align} 
  \label{Sk_2}
    \begin{split}
       \sum_{|k_3| \le  N_3} |S_{k k_3}| & \lesssim (N_1 \wedge N_2)^{2 -
      \alpha} \sum_{|k_3| \le  N_3} \jb{k - k_3}^{- 1} + \sum_{|k_1| \le  N_3} 1 \\
      &  \lesssim (N_1 \wedge N_2)^{2 - \alpha} \log N_3 + N_3.
    \end{split} 
  \end{align}

  \noi 
  Therefore, the first bound in \eqref{counting2} follows from \eqref{Sk}, \eqref{Sk_1}, and \eqref{Sk_2}.

By a similar argument, we have
\begin{align*} 
\begin{split}
|S_{k_1}| & \lesssim  \sum_{|k_2| \le  N_2} |S_{k_1 k_2}| \\
& \lesssim N_3^{2 -
\alpha} \sum_{|k_2| \le  N_2} \jb{k_1 - k_2}^{- 1} + \sum_{|k_2| \le  N_2} 1 \\
& \lesssim N_3^{2 - \alpha} \log N_2 + N_2.
\end{split}
\end{align*}

\noi 
Thus, we finish the proof for the second estimate of \eqref{counting2}.
  
The proof for the rest of \eqref{counting2} can be handled similarly. Thus, we omit them.
\end{proof}

For the countings of $|S_{k_1k_3}|$ and $|S_{kk_2}|$, the argument in Lemma \ref{LEM:counting1} is not sufficient.
We need the following further decomposition.
  \begin{equation}
    \label{Eqn:bad1} 
    \begin{cases}  
    S^{\rm bad}_{k_1 k_3} = \left\{ (k, k_2) \in
      S_{k_1k_3} ;   | 2k - (k_1 + k_3) |
      \ll  |k_1 + k_3| \right\} ;\\
      S^{\rm good}_{k_1 k_3} = \left\{ (k, k_2) \in
      S_{k_1k_3} ;  | 2k - (k_1 + k_3)  | \ges
      |k_1 + k_3| \right\} ;\\
      S^{\rm bad}_{k k_2} = \left\{ (k_1, k_3) \in
      S_{kk_2} ;  | 2 k_1 - (k + k_2)  |
      \ll  |k + k_2| \right\}
      ;\\
      S^{\rm good}_{k k_2} = \left\{ (k_1, k_3) \in
      S_{kk_2}  ;  | 2 k_1 - (k + k_2)  | \ges 
      |k + k_2|
      \right\} .
    \end{cases}
  \end{equation}

\noi 
With these notations, we have the following improved counting estimates.

\begin{lemma}
\label{LEM:counting2} 
Let $\alpha \in (1,2)$. 
Then we have the following counting estimates
  \begin{align} 
    \begin{split}
      |S_{k k_2}^{\rm good}|,
      |S_{k_1 k_3}^{\rm good}| & \lesssim 1 ;\\
      |S_{k k_2}^{\rm bad}| & \lesssim | k + k_2 |^{1
      - \frac{\alpha}{2}} +1 ; \\ 
      |S_{k_1 k_3}^{\rm bad}|  & \lesssim | k_1 + k_3
      |^{1 - \frac{\alpha}{2}} +1.
    \end{split} 
  \label{counting3}
  \end{align}
\end{lemma}

\begin{proof}
We start with the estimate of $|S_{k k_2}^{\rm good}|$ with $k+k_2 \neq 0$.
Similar to the proof of Lemma \ref{LEM:counting1}, we may recast the set $S_{k k_2}^{\rm good}$ as
\begin{align*}
\begin{split}
|S_{k k_2}^{\rm good}| & = |S_{k k_2} \cap \{ (k_1,k_3) \in S_{k k_2}; |2k_1 - (k  + k_2)| \ges |k + k_2|  \}| \\ 
& =  |\{ k_1 \in \mathbb{Z}; k_1 \neq  k , |k_1|\le N_1, | k_1 - k_2 - k | \le N_3, \\
& \hspace{2.5cm} \phi_{k + k_2, +} (k_1) = | k |^{\alpha} + | k_2 |^{\alpha} - m + O (1) \}  \\
  & \hspace{0.5cm} \cap \{ (k_1,k_3) \in S_{k k_2}; |2k_1 - (k  + k_2)| \ges |k + k_2|  \}|,
\end{split}
\end{align*}

\noi 
where $\phi_{k + k_2,+}$ is given in \eqref{phib} with $b = k +k_2$.
Thus by using Lemma \ref{LEM:basic} and then \eqref{Eqn:lowb+}, we have
\begin{align*}
\begin{split}
|S_{k k_2}^{\rm good} |  \les 1 +  |k + k_2|^{1-\al} \les 1,
\end{split} 
\end{align*}

\noi 
since $k+k_2 \neq 0$.
Thus, we finish the proof of the bound for $|S_{k k_2}^{\rm good} | $. 

Now we turn to $|S_{k k_2}^{\rm bad} | $.  
We distinguish two cases.
We first consider the case when $|2k_1 - (k  + k_2)| \le |k + k_2|^{1-\frac{\al}2}$, for which we have the crude estimate
\begin{align} 
\begin{split}
|S_{k k_2}^{\rm bad} & \cap \{ (k_1,k_3) \in S_{kk_2}; |2k_1 - (k  + k_2)| \le |k + k_2|^{1-\frac{\al}2}  \}| \\
& \le | \{ k_1 \in \Z;  |2k_1 - (k  + k_2)| \le |k + k_2|^{1-\frac{\al}2}  \}| \\ 
& \les |k + k_2|^{1-\frac{\al}2} + 1,
\end{split}
\label{Sbad1}
\end{align}

\noi 
which is sufficient for our purpose. From now on, we assume $|2k_1 - (k  + k_2)| > |k + k_2|^{1-\frac{\al}2}$.
Similar to the above, we may rewrite 
\begin{align*}
\begin{split}
|S_{k k_2}^{\rm bad} & \cap \{ (k_1,k_3) \in S_{kk_2}; |2k_1 - (k  + k_2)| > |k + k_2|^{1-\frac{\al}2}\}|\\  
& = |S_{k k_2} \cap \{ (k_1,k_3) \in S_{k k_2};  |k + k_2|^{1-\frac{\al}2}  \le |2k_1 - (k  + k_2)| \ll |k + k_2|  \}| \\ 
& =  |\{ k_1 \in \mathbb{Z}; k_1 \neq  k , |k_1|\le N_1, | k_1 - k_2 - k | \le N_3, \\
& \hspace{2.5cm} \phi_{k + k_2, +} (k_1) = | k |^{\alpha} + | k_2 |^{\alpha} - m + O (1) \}  \\
  & \hspace{0.5cm} \cap \{ (k_1,k_3) \in S_{k k_2}; |k + k_2|^{1-\frac{\al}2}  \le  |2k_1 - (k  + k_2)| \ll |k + k_2|  \}|.
\end{split}
\end{align*}

\noi 
Then we may apply Lemma \ref{LEM:basic}, together with \eqref{Eqn:sharp}, to get 
\begin{align}
\begin{split}
|S_{k k_2}^{\rm bad} & \cap \{ (k_1,k_3) \in S_{kk_2}; |k + k_2|^{1-\frac{\al}2}  \le  |2k_1 - (k  + k_2)|\}| \\
& \les 1 + |k + k_2|^{1- \frac{\al}2},
\end{split}
\label{Sbad2}
\end{align}

\noi 
which is again sufficient for our purpose. By collecting \eqref{Sbad1} and \eqref{Sbad2}, we proved the estimate for $|S_{k k_2}^{\rm bad}|$ .

The proof for the rest of \eqref{counting3} is similar and thus omitted. 
\end{proof}

\begin{corollary}
\label{Cor:bad} 
The following bounds hold.
\[
|S^{\rm bad}_{k k_2}| \lesssim ( N_1 \wedge N_3 )^{1-\frac\alpha2};\qquad |S^{\rm bad}_{k_1 k_3} | \lesssim ( N \wedge N_2 )^{1-\frac\alpha2},
\]
  where $S^{\rm bad}_{k k_2}$ and $S^{\rm bad}_{k_1 k_3}$ are defined in \eqref{Eqn:bad1}.
\end{corollary}
\begin{proof}
We only consider the estimate for $|S^{\rm bad}_{k k_2}|$, as the proof for the bound of $|S^{\rm bad}_{k_1 k_3}|$ is similar. We first note that for $(k_1,k_3) \in S^{\rm bad}_{k k_2}$, we have
\[
|2k_3 - (k+k_2)| = |2k_1 - (k+k_2)| \ll |k+k_2|,
\]

\noi 
which implies that 
\begin{align} 
\label{bad_kk2}
|k_1| \sim |k+k_2| \sim |k_3|.
\end{align}

\noi 
Then recall that $|k_1| \le N_1$ and $|k_3| \le N_3$, which together with Lemma \ref{LEM:counting2} and \eqref{bad_kk2}, implies
\[
|S^{\rm bad}_{k k_2}| \les |k+k_2|^{1-\frac\al2} +1 \les \min (N_1^{1-\frac\al2}, N_3^{1-\frac\al2}).
\]

\noi 
Thus, we finish the proof.
\end{proof}

\begin{remark}\rm
\label{RMK:Skk2}
As a consequence of Lemma \ref{LEM:counting2} and Corollary \ref{Cor:bad}, we have shown that
\begin{align*} 
    |S_{k k_2}| \lesssim ( N_1 \wedge N_3 )^{1-\frac\alpha2},\qquad |S_{k_1 k_3} | \lesssim ( N \wedge N_2 )^{1-\frac\alpha2}.
\end{align*}
With this improvement, we may deduce further upper bounds for $|S_{k_1}|$ and $|S_{k_3}|$ as follows
\begin{align}
\label{impro_13}
|S_{k_1}| \les N_3 ( N \wedge N_2 )^{1-\frac\alpha2}, \qquad |S_{k_3}| \les N_1 ( N \wedge N_2 )^{1-\frac\alpha2}.  
\end{align}

\end{remark}

Finally, we are ready to estimate the size of $|S|$.

\begin{lemma}
    \label{LEM:S}
    Let $S$ be given in \eqref{Eqn:S} with $\al \in (1,2)$ and $1 \le N_1, N_2, N_3 \le N$. Then, we have
    \begin{align}  
    \begin{split}
      |S| & \lesssim \min \big (N_3^{2 - \alpha} (N_1 \wedge N_2) \log (N_1 \vee
      N_2) + N_1 N_2, \\ 
      & \hspace{1.5cm}  N_1^{2 -
      \alpha} (N_2 \wedge N_3) \log (N_2 \vee N_3) + N_2 N_3 \big) .
      \end{split}
      \label{coun:S}
    \end{align}
\end{lemma}

\begin{proof}
We first observe that
  \begin{align} 
    |S| \lesssim \min \bigg( \sum_{k_1, k_2} |S_{k_1 k_2}|, \sum_{k_2, k_3} |S_{k_2
    k_3}| \bigg). 
  \label{|S|}
  \end{align} 

  \noi 
  We shall estimate the right-hand side of \eqref{|S|} term by term. 
  From \eqref{counting1}, we have
  \begin{align} 
    \begin{split}
      \sum_{k_1, k_2} & |S_{k_1 k_2}| \lesssim N_3^{2 - \alpha} \sum_{k_1, k_2}
      \jb{k_1 - k_2}^{- 1} + \sum_{k_1, k_2} 1\\
      & \lesssim N_3^{2 - \alpha} \sum_{| k_1 | \le  N_1} \sum_{| k_2 |
      \le  N_2} \jb{k_1 - k_2}^{- 1} + N_1N_2\\
      & \lesssim N_3^{2 - \alpha} (N_1 \wedge N_2) \log (N_1 \vee N_2) + N_1 N_2 .
    \end{split} 
  \label{|S|_1}
  \end{align}

  \noi 
  Similarly, we have
  \begin{align}
    \label{|S|_3}
    \sum_{k_2, k_3} |S_{k_2
    k_3}|  \lesssim N_1^{2 - \alpha} (N_2 \wedge N_3) \log (N_2 \vee N_3) + N_2 N_3.
  \end{align}

  \noi 
  Then \eqref{coun:S} follows from \eqref{|S|_1} and \eqref{|S|_3}.
\end{proof}

Another structure that we shall exploit is the so-called $\Gamma$-condition. 
To be more precise,
if there exists a positive number $\Gamma$ such that 
\[ 
|k_{\max}| \le \Gamma < |k|,
\]

\noi 
where $k_{\max}
$ is the frequency corresponding to $N_{\max}$,
then we call that $S$ given in \eqref{Eqn:S} satisfies the $\Gamma$-condition. To simplify the notation, let
\begin{align} 
\label{gammaB}
B_\Gamma = \{ (k,k_1,k_2,k_3) \in S; |k_{\max}| \le \Gamma < |k|\}.
\end{align}

\noi 
From the fact that $k = k_1 - k_2 + k_3$, we see that
\begin{equation}
  \label{Eqn:gamma} 
  \begin{cases}
    \Gamma \le  | k | \le  \Gamma + 2
    N_{\med} ;\\
    \Gamma - 2 N_{\med} \le  | k_{\max} | \le 
    \Gamma.
  \end{cases}
\end{equation}
Thus, both $| k |$ and $| k_{\max} |$ locate in an interval of length $2
N_{\med}$.

\medskip

With the $\Gamma$-condition, we can improve the previous counting estimates.

\begin{lemma}
  \label{LEM:gamma} 
  Let $S$ be given in \eqref{Eqn:S} with $\al \in (1,2)$, $B_\Gamma$ be as in \eqref{gammaB}, and $1 \le N_1, N_2, N_3 \le N$. Then, we have
  \begin{align} 
      |{B_\Gamma} \cap S| & \lesssim N_{\min} N_{\med} ; \label{Sgamma}\\
      |({B_\Gamma} \cap S)_k|, |({B_\Gamma} \cap S)_{k_i}| &
      \lesssim N_{\med}, \,\,\text{ for }\,\, i = 1,2,3 ;\label{Skgamma}\\
      |({B_\Gamma} \cap S)_{k_{\min} k_{\med}}| &
      \lesssim N_{\med} . \label{Skkgamma}
  \end{align}
  Here $({B_\Gamma} \cap S)_k$ is the set of $(k_1,k_2,k_3) \in \Z^3$ such that $(k, k_1,k_2,k_3) \in {B_\Gamma} \cap S$ when $k$ is fixed. The definitions of $({B_\Gamma} \cap S)_{k_i}$ and $({B_\Gamma} \cap S)_{k_{\min} k_{\med}}$ are similar.
\end{lemma}

\begin{proof}
The key observation is that under the $\Gamma$-condition \eqref{Eqn:gamma}, both $|k|$ and $|k_{\max}|$ are confined in some intervals of size $2N_{\med}$.
We first prove \eqref{Sgamma}.
Recall that 
  \begin{align}
  \label{Sgamma1}
    \begin{split}
     & |{B_\Gamma} \cap S| \\
     & \lesssim \min \bigg( \sum_{k, k_1} |({B_\Gamma} \cap S)_{k k_1}|, \sum_{k, k_3}
     |({B_\Gamma} \cap S)_{k k_3}|, \sum_{k_1, k_2}  |({B_\Gamma} \cap S)_{k_1 k_2}|, \sum_{k_2, k_3}  |({B_\Gamma} \cap S)_{k_2 k_3}|  \bigg)  \\ 
     & \les \min \bigg( \sum_{k, k_1} \big( (N_2 \wedge N_3)^{2 - \alpha} \jb{k_1 - k}^{- 1} +
      1 \big), \sum_{k, k_3}
     \big( (N_1 \wedge N_2)^{2 - \alpha} \jb{k - k_3}^{- 1} +
      1\big) , \\
      & \hspace{2cm} \sum_{k_1, k_2}  \big( N_3^{2 - \alpha} \jb{k_1 - k_2}^{- 1} +  1 \big), \sum_{k_2, k_3} \big( N_1^{2 - \alpha} \jb{k_2 - k_3}^{- 1} +  1 \big)\bigg),
   \end{split}   
  \end{align}

  \noi 
  where the summations are under \eqref{Eqn:gamma}.
  We then distinguish three cases.
  If $N_1 = N_{\min}$, then from \eqref{Eqn:gamma} and \eqref{Sgamma1} we have
  \begin{align}
    \label{G_N1}
    \begin{split}
    |{B_\Gamma} \cap S| & \lesssim 
    \sum_{k, k_1} \big( (N_2 \wedge N_3)^{2 - \alpha} \jb{k_1 - k}^{- 1} +
      1 \big) \\ 
    & \les \sum_{ k_1} \big[ (N_2 \wedge N_3)^{2 - \alpha} \log N_{\med} +  N_{\med} \big] \\
      & \les N_{\min} N_{\med}^{2 - \alpha} \log N_{\med} +
      N_{\min} N_{\med} \\ 
      & \les N_{\min} N_{\med},
    \end{split}
  \end{align}

  \noi 
  where we used the $\Gamma$-condition \eqref{Eqn:gamma} in the first step when summing over $k$. Thus we finish the proof for the case $N_1 = N_{\min}$.
  If $N_3 = N_{\min}$, the argument is similar and thus omitted.
  In the following, we assume that $N_2 = N_{\min}$. If we further assume that $N_3 = N_{\med}$, then similar computation as in \eqref{G_N1} yields
  \begin{align*} 
    \begin{split}
    |{B_\Gamma} \cap S| & \lesssim 
    \sum_{k_1, k_2} \big( N_3^{2 - \alpha} \jb{k_1 - k_2}^{- 1} +  1 \big) \\ 
    & \les \sum_{ k_{2}} \big( N_{\med}^{2 - \alpha} \log N_{\med} +  N_{\med} \big) \\
      & \les N_{\min} N_{\med}^{2 - \alpha} \log N_{\med} +
      N_{\min} N_{\med} \\ 
      & \les N_{\min} N_{\med}.
    \end{split}
  \end{align*}

  \noi 
  It remains to consider the case when $N_2 = N_{\min}$ and $N_1 = N_{\med}$.
  Then, from \eqref{Eqn:gamma} and \eqref{Sgamma1} we have 
  \begin{align*} 
    \begin{split}
    |{B_\Gamma} \cap S| & \lesssim 
    \sum_{k_2, k_3} \big( N_1^{2 - \alpha} \jb{k_2 - k_3}^{- 1} +  1 \big) \\ 
    & \les \sum_{ k_{2}} (N_{\med}^{2 - \alpha} \log N_{\med} +
       N_{\med}) \\
      & \les N_{\min} N_{\med}^{2 - \alpha} \log N_{\med} +
      N_{\min} N_{\med} \\ 
      & \les N_{\min} N_{\med}.
    \end{split}
  \end{align*}  

  \noi 
  Thus we finish the proof. 

  Now we consider \eqref{Skgamma}. We first note that
  \begin{align*}
    \begin{split}
    & |({B_\Gamma} \cap S)_k|  \les \min \bigg( \sum_{k_1} |({B_\Gamma} \cap S)_{k k_1}|,   \sum_{k_3}  |({B_\Gamma} \cap S)_{k k_3}|  \bigg)  \\ 
     & \les \min \bigg( \sum_{k_1} \big( (N_2 \wedge N_3)^{2 - \alpha} \jb{k_1 - k}^{- 1} +
      1 \big),  \sum_{k_3} \big( (N_1 \wedge N_2)^{2 - \alpha} \jb{k - k_3}^{- 1} +  1 \big)\bigg) \\
      & \les \sum_{k_{\med}} \big( N_{\med}^{2 - \alpha} \jb{k - k_{\med}}^{- 1} +  1 \big)  \les N_{\med}^{2 - \alpha + \eps }  + N_{\med } \les N_{\med}.
      \end{split}
  \end{align*}

  \noi 
  Similarly, we have
  \begin{align*}
    \begin{split}
   |({B_\Gamma} \cap S)_{k_1}|  &  \lesssim  \min \bigg( \sum_{k_2}  |({B_\Gamma} \cap S)_{k_1 k_2}|, \,    \sum_{k}  |({B_\Gamma} \cap S)_{k k_1}| \bigg) \\ 
     & \les  \min \bigg( \sum_{k_2}
     \big(  N_3^{2 - \alpha} \jb{k_1 - k_2}^{- 1} + 1\big) , \, \sum_{k}
     \big(  N_2^{2 - \alpha} \jb{k - k_1}^{- 1} + 1\big) \bigg)  \\
      & \les N_{\med},
      \end{split}
  \end{align*}

  \noi 
  where in the last step, we used \eqref{Eqn:gamma}, which suggests the summation over $k_2$ is at most over an interval of size $N_{\med}$.
  The rest of \eqref{Skgamma} follows similarly. 

  Finally, we turn to  \eqref{Skkgamma}. 
  We note that
  \begin{align*}
\begin{split}
|(B_\Gamma \cap S)_{k_{\min} k_{\med}}|  & = |S_{k_{\min} k_{\med}} \cap \{ \eqref{Eqn:gamma}  \}| \\ 
& \le  | \{ k_{\max} \textup{ range in an interval of size at most } 2N_{\med} \}| \\ 
& \le 2N_{\med}.
\end{split}
\end{align*}
  
\noi 
Thus, we finish the proof of \eqref{Skkgamma}.
\end{proof}

\subsection{Tensor norm estimates} 
\label{SUB:TNE}
This subsection will introduce the base tensor operator and then discuss its operator norm estimates. 
Given $m\in \Z$, $\al \in (1,2)$, and dyadic numbers $1 \le N_1, N_2, N_3 \le N$,
we define the base tensor ${\rm T}^{{\rm b},m}$ as
\begin{align}
    \label{baseT}
    {\rm T}^{{\rm b},m} = {\rm T}^{{\rm b},m}_{kk_1k_2k_3} = \ind_{S} (k,k_1,k_2,k_3),
\end{align}

\noi 
where $\ind_S$ is the indicator function defined in \eqref{ind}, and $S$ is the subset of $\Z^4$ given in \eqref{Eqn:S}.
Here, with a slight abuse of notation, we will not distinguish the tensor operator ${\rm T}^{{\rm b},m}$ and its kernel ${\rm T}^{{\rm b},m}_{kk_1k_2k_3}$.
From \eqref{Eqn:S} we see that
\[
\begin{split}
{\rm T}^{{\rm b},m}_{kk_1k_2k_3} & = \ind_{k-k_1+k_2-k_3 = 0} \cdot \ind_{k_2 \notin \{k_1,k_3\}} \cdot \ind_{\{|k|^\al - |k_1|^\al + |k_2|^\al - |k_3|^\al  = m + O(1)\}} \\
& \hphantom{XXXXXX} \times \ind_{|k|\le N} \cdot \prod_{i=1}^3 \ind_{|k_i| \le N_i}.
\end{split}
\]

\noi 
Recall the operator norms defined in \eqref{Eqn:opnorm}. For example,
\begin{align} 
\label{baseT1}
\| {\rm T}^{{\rm b},m}_{k k_1 k_2 k_3} \|^2_{kk_1 k_2 k_3} = \sum_{ k,k_1, k_2, k_3 \in \mathbb{Z}}  \big| {\rm T}^{{\rm b},m}_{k k_1 k_2 k_3} \big|^2 = |S|,
\end{align}
or 
\[ 
  \begin{split}
    & \| {\rm T}^{{\rm b},m}_{k k_1 k_2 k_3} \|^2_{ k_2 k_3
    \to   k k_1} = \sup \bigg\{ \sum_{k,k_1 \in \Z} \Big| \sum_{ k_2, k_3 \in \mathbb{Z}}
    {\rm T}^{{\rm b},m}_{k k_1 k_2 k_3} z_{k_2 k_3} \Big|^2 ;
    \sum_{k_2, k_3 \in \mathbb{Z}} | z_{k_2 k_3} |^2 = 1 \bigg\} .
  \end{split} 
\]
Now we are ready to state the estimates for tensor norms, which are the main tools of our later analysis.
We start with the Hilbert–Schmidt norm estimate of the base tensor operator ${\rm T}^{{\rm b},m}$, which is a consequence of Lemma \ref{LEM:S}.
\begin{lemma}
\label{LEM:tensor} 
Let ${\rm T}^{{\rm b},m}$ be the base tensor defined in \eqref{baseT}. 
Then we have
\begin{align}
\begin{split}
\| {\rm T}^{{\rm b},m}_{k k_1 k_2 k_3} \|_{k
k_1 k_2 k_3}^2 & \lesssim \min  \big(N_3^{2 - \alpha} (N_1 \wedge N_2) \log (N_1 \vee
N_2) + N_1 N_2, \\ 
& \hspace{1.5cm}  N_1^{2 -
\alpha} (N_2 \wedge N_3) \log (N_2 \vee N_3) + N_2 N_3) .
\end{split}
  \label{tensor_E1}
\end{align} 
\end{lemma}

\begin{proof}
The estimate \eqref{tensor_E1} is a consequence of \eqref{baseT1} and Lemma \ref{LEM:S}.
\end{proof}

\begin{lemma}
\label{LEM:tensor2}
Let ${\rm T}^{{\rm b},m}$ be the base tensor defined in \eqref{baseT}. 
Then we have
\begin{align}
\begin{split}  
\| {\rm T}^{{\rm b},m}_{k k_1 k_2 k_3} \|_{k k_1 \to  k_2 k_3}^2 & \lesssim (N_2 \wedge N_3)^{2 - {\alpha} } N_1^{2 - \al}, \\
\| {\rm T}^{{\rm b},m}_{k k_1 k_2 k_3} \|_{k k_3 \to  k_1 k_2}^2 & \lesssim (N_1 \wedge N_2)^{2 - \al} N_3^{2 - \al} ,\\
\| {\rm T}^{{\rm b},m}_{k k_1 k_2 k_3} \|_{k k_2 \to  k_1 k_3}^2 & \lesssim (N_1 \wedge N_3)^{1 -
\frac\al2} (N \wedge N_2)^{1 - \frac\al2}.
\end{split}
  \label{tensor2}
\end{align} 
\end{lemma}

\begin{proof}
We start with the first estimate of \eqref{tensor2}.
By the Schur's test, we have
  \[ 
    \begin{split}
      \| {\rm T}^{{\rm b},m}_{k k_1 k_2 k_3} \|^2_{k k_1
      \to  k_2 k_3} & \lesssim \bigg( \sup_{k, k_1 \in \mathbb{Z}}
      \sum_{k_2, k_3 \in \mathbb{Z}} | {\rm T}^{{\rm b},m}_{k k_1 k_2
      k_3} | \bigg) \times \bigg( \sup_{k_2, k_3 \in \mathbb{Z}}
      \sum_{k, k_1 \in \mathbb{Z}} | {\rm T}^{{\rm b},m}_{k k_1 k_2
      k_3} | \bigg) .
    \end{split} 
  \]

  \noi 
  which gives the first estimate of \eqref{tensor2} by using Lemma \ref{LEM:counting1}. 
  Same argument as above works for the second estimate of \eqref{tensor2}.
  For the third estimate, we use Remark \ref{RMK:Skk2} instead of Lemma \ref{LEM:counting1}.
\end{proof}

\begin{lemma} 
\label{LEM:tensor3}
Let ${\rm T}^{{\rm b},m}$ be the base tensor defined in \eqref{baseT}. 
Then we have
\begin{align}
\begin{split} 
\| {\rm T}^{{\rm b},m}_{k k_1 k_2 k_3} \|_{k \to  k_1 k_2 k_3}^2 & \lesssim \min ((N_2 \wedge N_3)^{2 - \alpha} \log N_1  + N_1, (N_1 \wedge N_2)^{2 - \alpha} \log N_3 + N_3) ;\\
\| {\rm T}^{{\rm b},m}_{k k_1 k_2 k_3}
\|_{k_1 \to  k k_2 k_3}^2 & \lesssim \min ( N_3^{2 - \alpha} \log N_2 +  N_2, N_3 (N \wedge N_2)^{1 -
\frac\al2})  ;\\
\| {\rm T}^{{\rm b},m}_{k k_1 k_2 k_3}
\|_{k_2 \to  k k_1 k_3}^2 & \lesssim \min (N_3^{2 - \alpha} \log N_1  + N_1, N_1^{2 - \alpha} \log N_3  + N_3)  ;\\
 \| {\rm T}^{{\rm b},m}_{k k_1 k_2 k_3}
\|_{k_3 \to  k k_1 k_2}^2 & \lesssim \min( N_1^{2 - \alpha} \log N_2  +  N_2, N_1 (N \wedge N_2)^{1 - \frac\al2}).\\
\end{split} 
\label{tensor3}
\end{align} 
\end{lemma}

\begin{proof}
Similar to in the proof of Lemma \ref{LEM:tensor2}, the estimates of \eqref{tensor3} are consequences of \eqref{counting2}. We only show how to get the first estimate, as others follow similarly\footnote{For the second and fourth estimates, we may need to use \eqref{impro_13} as well.}.

By the Schur's test, we have
  \begin{align}
      \label{tensor3_1} 
    \begin{split}
      \| {\rm T}^{{\rm b},m}_{k k_1 k_2 k_3} \|^2_{k
      \to  k_1 k_2 k_3} & \lesssim \bigg( \sup_{k  \in \mathbb{Z}}
      \sum_{k_1, k_2, k_3 \in \mathbb{Z}} | {\rm T}^{{\rm b},m}_{k k_1 k_2
      k_3} | \bigg) \times \bigg( \sup_{k_1, k_2, k_3 \in \mathbb{Z}}
      \sum_{k \in \mathbb{Z}} | {\rm T}^{{\rm b},m}_{k k_1 k_2
      k_3} | \bigg) .
    \end{split} 
  \end{align}

  \noi 
  We note that once $k_1,k_2,k_3$ are fixed, then $k$ is uniquely determined provided $(k,k_1,k_2,k_3) \in S$, where $S$ is given in \eqref{Eqn:S}. Therefore, we have
  \begin{align}
      \label{tensor3_2}
  \sup_{k_1, k_2, k_3 \in \mathbb{Z}}\sum_{k \in \mathbb{Z}} | {\rm T}^{{\rm b},m}_{k k_1 k_2 k_3} | \le 1.
  \end{align}

  \noi 
  On the other hand, for fixed $k\in \Z$, from \eqref{baseT} and \eqref{counting2}, we note
  \begin{align}
      \label{tensor3_3}
      \begin{split}
      \sum_{k_1, k_2, k_3 \in \mathbb{Z}} & | {\rm T}^{{\rm b},m}_{k k_1 k_2
      k_3} | = |S_k| \\
      & \les \min ((N_2 \wedge N_3)^{2 - \alpha} \log N_1  + N_1, (N_1 \wedge N_2)^{2 - \alpha} \log N_3 + N_3).
      \end{split}
  \end{align}
  
  \noi 
  Then, the first estimate of \eqref{tensor3} follows from \eqref{tensor3_1},  \eqref{tensor3_2}, and \eqref{tensor3_3}.
\end{proof}

We will need some estimates of the base tensor ${\rm T}^{{\rm b},m}$ subject to further restrictions,
for which we expect better estimates.

\begin{lemma}
\label{LEM:tensor4}
Let ${\rm T}^{{\rm b},m}$ be the base tensor defined in \eqref{baseT}. 
Then we have
\begin{align} 
\label{tensor4}
\begin{split}
  \| \ind_{|k_1 +   k_3| < |k_2|}
  {\rm T}^{{\rm b},m}_{k k_1 k_2 k_3} \|_{k k_1 k_2 k_3}^2 & \lesssim N_1 N_3 ;\\
  \| \ind_{|k_1 +   k_3| < |k_2|}
  {\rm T}^{{\rm b},m}_{k k_1 k_2 k_3} \|_{k k_2 \to  k_1 k_3}^2 & \lesssim (N_1 \wedge N_3)^{1-\frac\al2} ;\\
  \| \ind_{|k_1 +   k_3| < |k_2|}
  {\rm T}^{{\rm b},m}_{k k_1 k_2 k_3} \|_{k_1
  \to  k k_2 k_3}^2 & \lesssim N_3 ;\\
  \| \ind_{|k_1 +   k_3| < |k_2|}
  {\rm T}^{{\rm b},m}_{k k_1 k_2 k_3} \|_{k_3
  \to  k k_1 k_2}^2 & \lesssim N_1 .
\end{split}
\end{align}
\end{lemma}

\begin{proof}
We start with the first bound in \eqref{tensor3}.
We have 
\begin{align}
\label{tensor4_1}
\begin{split}
\| \ind_{|k_1 +   k_3| < |k_2|}
  {\rm T}^{{\rm b},m}_{k k_1 k_2 k_3} \|_{k k_1 k_2 k_3}^2 & = \sum_{k_1,k_3} \sum_{k,k_2} | \ind_{|k_1 +   k_3| < |k_2|}
  {\rm T}^{{\rm b},m}_{k k_1 k_2 k_3} |^2 \\ 
  & = \sum_{k_1,k_3} |S_{k_1k_3} \cap \{|k_1 +   k_3| < |k_2|\} |.
\end{split}
\end{align}

\noi 
We observe that under the condition $|k_1 +   k_3| < |k_2|$, we have
\[
| 2k - (k_1 + k_3) | = | 2k_2 - (k_1 + k_3) | \ge  2 |k_2| - |k_1 + k_3| > |k_2| > |k_1 +   k_3|,
\]

\noi 
which, together with \eqref{Eqn:bad1}, implies that
\begin{align}
\label{tensor4_2}
 S_{k_1k_3} \cap \{|k_1 +   k_3| < |k_2|\}  \subset  S_{k_1k_3}^{\rm good}.
\end{align}

\noi 
Thus, from \eqref{tensor3_1}, \eqref{tensor3_2}, \eqref{tensor4_1}, \eqref{tensor4_2}, and Lemma \ref{LEM:counting2}, we get
\[
\| \ind_{|k_1 +   k_3| < |k_2|}
  {\rm T}^{{\rm b},m}_{k k_1 k_2 k_3} \|_{k k_1 k_2 k_3}^2 \le \sum_{k_1,k_3} |S_{k_1k_3}^{\rm good}| \les N_1N_3,
\]

\noi 
which finishes the proof of the first estimate of \eqref{tensor3}.

We turn to the estimate of $\| \ind_{|k_1 +   k_3| < |k_2|}  {\rm T}^{{\rm b},m}_{k k_1 k_2 k_3} \|_{k k_2 \to  k_1 k_3}$. 
By a similar Schur's test argument as in \eqref{tensor3_1}, we have
  \[ 
    \begin{split}
      \| & \ind_{|k_1 +   k_3| < |k_2|} {\rm T}^{{\rm b},m}_{k k_1 k_2 k_3} \|^2_{k k_2
      \to  k_1 k_3} \\ 
      & \lesssim \Big( \sup_{k, k_2 \in \mathbb{Z}}
      \sum_{k_1, k_3 \in \mathbb{Z}} | {\rm T}^{{\rm b},m}_{k k_1 k_2
      k_3} | \Big) \times \Big( \sup_{k_1, k_3 \in \mathbb{Z}}
      \sum_{k, k_2 \in \mathbb{Z}} | \ind_{|k_1 +   k_3| < |k_2|} {\rm T}^{{\rm b},m}_{k k_1 k_2
      k_3} | \Big)\\
      & = \Big(\sup_{k, k_2 \in \mathbb{Z}} |S_{kk_2}| \Big) \cdot \Big(\sup_{k_1, k_3 \in
      \mathbb{Z}} |S^{\rm good}_{k_1 k_3}|\Big), 
    \end{split} 
  \]

  \noi  
  which, together with Lemma \ref{LEM:counting2}, gives the desired bound.

 For the third estimate $\| \ind_{|k_1 +   k_3| < |k_2|}
  {\rm T}^{{\rm b},m}_{k k_1 k_2 k_3} \|_{k_1
  \to  k k_2 k_3}$,
  we proceed similarly as the above to get 
    \begin{align}
      \label{tensor4_3} 
    \begin{split}
      \| \ind_{|k_1 +   k_3| < |k_2|} {\rm T}^{{\rm b},m}_{k k_1 k_2 k_3} \|^2_{k_1
      \to  k k_2 k_3} & \lesssim   \sup_{k_1  \in \mathbb{Z}}
      \sum_{k, k_2, k_3 \in \mathbb{Z}} |\ind_{|k_1 +   k_3| < |k_2|} {\rm T}^{{\rm b},m}_{k k_1 k_2
      k_3} |  \\ 
      & \les \sup_{k_1  \in \mathbb{Z}}
      \sum_{ k_3 \in \mathbb{Z}} |S_{k_1k_3}^{\rm good} | \les N_3,
    \end{split} 
  \end{align}

  \noi 
  where we used \eqref{tensor4_2} and Lemma \ref{LEM:counting2}.
  
  The last estimate of \eqref{tensor4} can be handled similarly as in \eqref{tensor4_3}. Thus, we omit the details.
\end{proof}

We also record the following result, which is a consequence of Lemma \ref{LEM:gamma}.

\begin{corollary}
\label{COR:gammaT}
Let ${\rm T}^{{\rm b},m}$ be the base tensor defined in \eqref{baseT}. 
Then we have
\begin{align*} 
\begin{split}
  \| \ind_{B_{\Gamma}}
  {\rm T}^{{\rm b},m}_{k k_1 k_2 k_3} \|_{k k_1 k_2 k_3}^2 & \lesssim N_{\min} N_{\med} ,\\ 
  \| \ind_{B_{\Gamma}}
  {\rm T}^{{\rm b},m}_{k k_1 k_2 k_3} \|_{k_1
  \to  k k_2 k_3}^2 & \lesssim N_{\med} ,\\
  \| \ind_{B_{\Gamma}}
  {\rm T}^{{\rm b},m}_{k k_1 k_2 k_3} \|_{k_3
  \to  k k_1 k_2}^2 & \lesssim N_{\med} ,\\
  \| \ind_{B_{\Gamma}}
  {\rm T}^{{\rm b},m}_{k k_1 k_2 k_3} \|_{k_2
  \to  k k_1 k_3}^2 & \lesssim N_{\med} .
\end{split}
\end{align*}
\end{corollary}

In the rest of this subsection, 
we recall some bilinear tensor estimates from {\cite[Proposition 4.11]{DNY3}}, which serve as ``nonlinear" estimates for operator norms of tensors.

\begin{proposition}
  \label{PROP:contr}
  Consider two tensors $h_{k_{A_1}}^{(1)}$ and
  $h_{k_{A_2}}^{(2)}$, where $A_1 \cap A_2 = C$. Let $A_1 \Delta A_2 = A$ and
  define the semi-product
  \[ 
    H_{k_A} = \sum_{k_C} h_{k_{A_1}}^{(1)} h_{k_{A_2}}^{(2)} . 
  \]
  For any partition $(X, Y)$ of $A$, let $X \cap A_i = X_i$, $Y \cap A_i
  = Y_i$ for $i = 1, 2$. Then, we have
  \[ 
    \| H \|_{k_X \to  k_Y} \le  \| h^{(1)} \|_{k_{X_1 \cup C} \to 
    k_{Y_1}} \cdot \| h^{(2)} \|_{k_{X_2} \to  k_{C \cup Y_2}} . 
  \]
\end{proposition}

The following estimate for special tensor semi-products can be proved using Proposition \ref{PROP:contr}. However, we demonstrate a fundamental proof to showcase the idea of proving  Proposition \ref{PROP:contr}.

\begin{lemma}
  \label{LEM:tech}
  The following holds.
  \[ 
    \| h^{(1)}_{k_1 k_1'} h^{(2)}_{k_2 k_1'} \|_{k_1' \to  k_1 k_2}
    \le  \| h^{(1)}_{k_1 k_1'} \|_{k_1' \to  k_1} \| h^{(2)}_{k_2
    k_1'} \|_{k'_1 \to  k_2} . 
  \]
\end{lemma}

\begin{proof}
By using the definition of the tensor norm \eqref{Eqn:opnorm}, we have
\[ 
  \begin{split}
    \| h^{(1)}_{k_1 k_1'} h^{(2)}_{k_2 k_1'} \|_{k_1' \to  k_1 k_2} &  = \sup_{\|a_{k_1'} \|_{k_1'} = 1} 
    \bigg\| \sum_{k_1'} h^{(1)}_{k_1 k_1'} h^{(2)}_{k_2 k_1'} a_{k_1'}
    \bigg\|_{k_1 k_2} \\ 
    & \le  \| h^{(1)}_{k_1 k_1'} \|_{k_1' \to 
    k_1} \sup_{\|a_{k_1'} \|_{k_1'} = 1}  \| h^{(2)}_{k_2 k_1'} a_{k_1'} \|_{k'_1 k_2}\\
    &  \le \| h^{(1)}_{k_1 k_1'} \|_{k_1' \to  k_1} \| h^{(2)}_{k_2
    k_1'} \|_{k_1' \to  k_2}.
  \end{split} 
\]
Thus, we finish the proof.
\end{proof}

We conclude this subsection with a weighted estimate, whose proof can be found in {\cite[Proposition 2.5]{DNY2}}. Also, see \cite[Proposition 2.9]{DNY4}.

\begin{proposition} 
  \label{Prop:co} 
  Suppose that matrices $h = h_{k k''}$, $h^{(1)} = h^{(1)}_{k
  k'}$ and $h^{(2)} = h_{k' k''}^{(2)}$ satisfy that
  \[ 
    h_{k k''} = \sum_{k'} h_{k k'}^{(1)} h_{k' k''}^{(2)}, 
  \]
  and $h_{k k'}^{(1)}$ is supported in $\{| k - k' | \lesssim L\}$, then we have
  \[ 
    \Big\| \Big( 1 + \frac{| k - k'' |}{L} \Big)^{\kappa} h_{k k''}
    \Big\|_{\ell_{k k''}^2} \lesssim \| h^{(1)}_{kk'} \|_{k \to  k'}
    \cdot \Big\| \Big( 1 + \frac{| k' - k'' |}{L} \Big)^{\kappa}
    h^{(2)}_{k' k''} \Big\|_{\ell_{k' k''}^2} . 
  \]
\end{proposition}

\medskip

\subsection{Random tensor estimates}

In this subsection, we collect some probability results.
We start with the Wiener chaos estimate (see {\cite[Theorem I.22]{Simon}}).

\begin{lemma}
  \label{LEM:wc}
  Let $\{ g_n \}_{n \in \mathbb{N}}$ be a sequence of
  independent standard real-valued Gaussian random variables. Given $k \in
  \mathbb{Z}_{\ge  0}$, let $\{ P_j \}_{j \in \mathbb{N}}$ be a sequence
  of polynomials in $\bar{g} = \{ g_n \}_{n \in \mathbb{N}}$ of degree at most
  $k$. Then, for $p \ge  2$, we have
  \[ 
    \bigg\| \sum_{j \in \mathbb{N}} P_j (\bar{g}) \bigg\|_{L^p (\Omega)} 
    \le  (p - 1)^{\frac{k}{2}} \bigg\| \sum_{j \in \mathbb{N}} P_j
     (\bar{g}) \bigg\|_{L^2 (\Omega)} . 
  \]
\end{lemma}

We recall the definition of $A$-certain.

\begin{definition}
\label{DEF:certain}
Let $\theta$ be the constant as in Theorem \ref{THM:main}.
If some statement $S$ for a random variable holds with probability $\P (S) \ge 1- C_{\theta} e^{-A^{\theta}}$ for some $A > 0$,
we say that this statement $S$ is $A$-certain.
\end{definition}

For a complex number $z$, we define $z^+ = z$ and $z^- = \bar z$.
For a finite index set $A$,
we also use the notation $z^{\zeta_j}$ for 
$\{ \zeta_j \}_{j \in A}$ with $\zeta_j \in \{
\pm \}$.   

\begin{definition}
    \label{DEF:pair}
    We say that $(k_i,k_j)$ for $i,j \in A$ is a pairing in $k_A$,
    if $k_i = k_j$ and $\zeta_i + \zeta_j = 0$.
    We say a pairing is over-paired if $k_i = k_j = k_\l$ for some $\l \in A \backslash \{i,j\}$.
\end{definition}

We also recall the following Large deviation inequality for multilinear Gaussians. It is a special case of {\cite[Lemma 4.1]{DNY2}}. We give a simpler proof for completeness.

\begin{lemma}
\label{LEM:LD}
  Let $E \subset \mathbb{Z}$ be a finite set, and $a = a_{k_1 \cdots k_r}
  (\omega)$ be a random tensor such that the collection $\{ a_{k_1 \cdots k_r}
  \}$ is independent with the collection $\{ g_k (\omega) ; k \in E \}$. Let
  $\zeta_j \in \{ \pm \}$ and assume that in the support of $a_{k_1 \cdots
  k_r}$ there is no pairing in $\{ k_1, \cdots, k_r \}$ associated with the
  signs $\zeta_j$. Define the random variable $X$ as
  \begin{align*} 
    X (\omega) := \sum_{k_1, \cdots, k_r} a_{k_1 \cdots k_r} \prod_{j = 1}^r
    \eta_{k_j} (\omega)^{\zeta_j}, 
  \end{align*}
  where $\eta_{k_j} \in \{ g_{k_j} , |g_{k_j}|^2-1 \}$.
  Then, for any $A > |E|$, we have $A$-certainly that
  \begin{align} 
  \label{LDE}
    | X (\omega) |^2 \le  A^{\theta} \cdot \sum_{k_1, \cdots, k_r} |
    a_{k_1 \cdots k_r} (\omega) |^2 . 
  \end{align}
\end{lemma}

 \begin{proof}
By using Lemma \ref{LEM:wc}, we have
\[
\E [|X|^p] \les 
p^{\frac{2rp}2} \big( \E [|X|^2] \big)^{\frac{p}2}.
\]

\noi 
Due to the no-pairing assumption and independence, we have
\[
\begin{split}
\E [|X|^2] & = \E \bigg[  \Big| \sum_{k_1, \cdots, k_r} a_{k_1 \cdots k_r} \prod_{j = 1}^r
\eta_{k_j} (\omega)^{\zeta_j} \Big|^2\bigg] \\
& =  \sum_{k_1, \cdots, k_r} | a_{k_1 \cdots k_r}|^2 \prod_{j = 1}^r \E \big[  |\eta_{k_j} (\omega)|^2 \big] \\
& \les \sum_{k_1, \cdots, k_r} | a_{k_1 \cdots k_r}|^2,
\end{split}
\]

\noi 
where we used that $\E[|g_{k}|^2] \sim \E [|(|g_k|^2-1)|^2] \sim 1$.
Thus we have shown that
\[
\E [|X|^p] \les p^{\frac{2rp}2} \bigg( \sum_{k_1, \cdots, k_r} |a_{k_1 \cdots k_r}|^2 \bigg)^{\frac{p}2}.
\]

\noi 
Then, \eqref{LDE} follows by a standard Nelson's argument.
See, for example, \cite[Lemma 4.5]{Tz10}.
 \end{proof}

Now we are ready to state the random tensor estimate, whose proof can be found in {\cite[Proposition 4.14]{DNY3}}. See also \cite[footnote 39]{DNY4} and \cite{BDNY} for related discussions.

\begin{proposition}
  \label{Prop:te} 
  Let $A$ be a finite set and $h_{b c k_A} = h_{b c k_A}
  (\omega)$ be a tensor, where each $k_j$ and $(b, c) \in \mathbb{Z}^q$ \ for
  some integer $q \ge  2$. Given signs $\zeta_j \in \{ \pm \}$, we also
  assume that $\jb{b-b_0}, \jb{c-c_0} \lesssim M$ and $\jb{k_j } \lesssim M$ for some fixed $b_0, c_0 \in \Z$ and all $j \in A$, where $M$ is a dyadic number, and that in support of $h_{b c k_A}$
  there is no pairing in $k_A$. Define the tensor
  \begin{equation}
    \label{Eqn:Hbc} 
    H_{b c} = \sum_{k_A} h_{b c k_A} \prod_{j \in A}
    \eta^{\zeta_j}_{k_j},
  \end{equation}
  where we restrict $k_j \in E$ in \eqref{Eqn:Hbc}, $E$ being a finite set
  such that $\{ h_{b c k_A} \}$ is independent with $\{ \eta_k ; k \in E \}$, and $\eta_{k_j} \in \{ g_{k_j} , |g_{k_j}|^2-1 \}$.  
  Then $T^{-1}M$-certainly, we have
  \[ 
    \| H_{b c} \|_{b \to  c} \lesssim T^{- \theta} M^{\theta} \cdot
    \max_{(B, C)} \| h \|_{b k_B \to  c k_C}, 
  \]
  for some $T\ll 1$,
  where $(B, C)$ runs over all partitions of $A$.
\end{proposition}

\begin{remark}\rm 
\label{RMK:abs}
In what follows, we may ignore the constant $T^{-\theta} M^\theta$, as these constants may be absorbed by choosing $\theta \ll 1$ in our analysis.
For instance, the factor $T^{-\theta}$ can be absorbed by Proposition \ref{Prop:stb} by choosing $\theta \ll b_1 - b$.
See also \cite{DNY3,DNY4} for similar arguments. 
\end{remark}

\section{Proof of Theorem \ref{THM:main}}
\label{Sec:proof}

We are ready to prove the main result of this paper, i.e. Theorem \ref{THM:main}.
As explained in Subsection \ref{SUBS:SS}, we shall proceed with an induction argument.
To be more precise,
to construct the solution $u(t)$ to \eqref{FNLS} with the formal expansion \eqref{Eqn:Ansatz}, we need to (i) construct $z_N$ and $\mathfrak h^{N,L}$ with $1\le L < N^{1-\dl}$ for each dyadic $N\ge 1$; (ii) show the convergence of \eqref{Eqn:Ansatz} in a proper space. Now we illustrate how to construct $z_N$ and $\mathfrak h^{N,L}$ in an induction manner. 
We first note that $\Pi_{1/2} = 0$\footnote{Here we used the fact that $\{k\in \Z; \jb{k} \le 1/2\} = \emptyset$.}, which implies that $u_{1/2} = y_{1/2} = F_{1/2} = z_{1/2} = 0$,  $\mathcal H^{N,1/2} = \Pi_N$, and $\psi_{N,1/2} = \chi(t)F_N$.
Then we move to $z_1$. 
We note that from \eqref{zN} that
\begin{align}
    \label{y1}
    y_1 = \psi_{1,1/2} + z_1 = F_1 + z_1.
\end{align}
From \eqref{Eqn:newzN} and \eqref{y1}, it follows that
\begin{align}
    \label{z1}
    \begin{split}
    z_1(t) & = -i \chi (t) \bigg( \int^t_0 \Pi_1
    \mathcal{M} (y_{1}, y_{1}, y_{1}) (t') d t' + \int^t_0 \Pi_1 \mathcal{R} (y_1) (t') d t' \bigg) \\
    & = -i \chi (t) \sum_{w_1^{(i)} \in \{F_1, z_1\}} \bigg( \int^t_0 \Pi_1
    \mathcal{M} (w_{1}^{(1)}, w_{1}^{(2)}, w_{1}^{(3)}) (t') d t' \\
    & \hphantom{XXX} + \int^t_0 \Pi_1 \mathcal{R} (w_{1}^{(1)}, w_{1}^{(2)}, w_{1}^{(3)}) (t') d t' \bigg) 
    \end{split}
\end{align} 

\noi 
which can be solved, provided $T\ll 1$, as $\Pi_1$ is the projection to zero frequency. 
Given $z_1$ solved from \eqref{z1}, we obtain $y_1$, from \eqref{zN} and $\psi_{1,1/2} = F_1$, and thus $u_1$ by \eqref{Eqn:yNdef}. 
Next, we can use \eqref{Eqn:method} to construct $\psi_{N,1}$, which is a linear equation that can be solved locally for small $t$\footnote{See \eqref{Eqn:method} for more details.}, where we used $u_1$ that we obtained earlier.  
By plugging $\psi_{2,1}$ and $y_1$ into \eqref{Eqn:newzN}, we can solve for $z_2$\footnote{This can be done via a fixed point argument as it is essentially an ODE system. Also see Section \ref{Sec:rems} for further discussion.}.
Similarly, we can get $y_2$ from \eqref{zN} and then $u_2$ from \eqref{Eqn:yNdef}.
We repeat the above process to construct $\psi_{N,2}$ using \eqref{Eqn:method} and $u_2$. 
By iterating this procedure, we obtain a sequence of $\{z_L, u_L, \psi_{N,L}\}$ for all dyadic numbers $L,N$, such that $L < N^{1-\dl}$. To show point (ii), which is the convergence of \eqref{Eqn:Ansatz} in some space, we need to show that each term in the sequence $\{z_L, u_L, \psi_{N,L}\}$ satisfies some proper bounds, which are elaborated in the following definition. 

\begin{definition}
  \label{Def:LocalM} 
  Given $0 < T \ll 1$, $b > \frac12$ sufficiently closed to $\frac12$, and let $J = [- T, T]$ and $\chi(t) = \eta_T(t)$. For any
  dyadic number $M$, consider the following statements, which we call
  $\text{{{\em {\textsf{Local}}\/}}} (M) :$
  \begin{itemize}
    \item[(i)] For the operator $\wt{\mathfrak{h}}^{N, L}$ defined in \eqref{wthNL},
    where $L < \min (M, N^{1 - \delta})$, we have
    \begin{equation}
      \label{Eqn:1} 
      \| \wt{\mathfrak{h}}^{N, L} \|_{Y^b} \le  L^{- \delta_0},
      \text{\quad} \| \wt{\mathfrak{h}}^{N, L} \|_{Z^b} \le 
      N^{\frac{1}{2} + \gamma_0} L^{- \frac{1}{2}},
    \end{equation}
    as well as
    \begin{equation}
      \label{Eqn:2}
      \Big\| \Big( 1 + \frac{| k - k' |}{L}
      \Big)^{\kappa} h_{k k'}^{N, L} \Big\|_{Z^b} \le 
      {{N}}.
    \end{equation} 
    Here the $Y^b$ and $Z^b$ norms are defined in \eqref{Yb} and \eqref{Zb}, respectively.
    
    \item[(ii)] For the terms $z_N$, where $N \le  M$, we have
    \begin{equation}
      \label{Eqn:3} \| z_N \|_{X^b (J)} \le 
      N^{- \frac{1}{2} - \gamma}.
    \end{equation}
    
    \item[(iii)] For any dyadic numbers $L_1$, $L_2 < M \le N$ the operator
    $\mathcal{P}^{\pm}$ defined by
    \begin{equation}
      \label{Eqn:5} 
      \begin{split}
        & \mathcal{P}^+ [\psi] (t) = \mathcal{P}^{+}_{N,L_1,L_2} [\psi] (t) : = - i  \chi(t) \int^t_0 \Pi_N \mathcal{M}
        (y_{L_1}, y_{L_2},  \psi) (t') d t'\\
        & \mathcal{P}^{-} [\psi] (t)  = \mathcal{P}^{-}_{N,L_1,L_2} [\psi] (t) : = - i  \chi(t) \int^t_0 \Pi_N \mathcal{M}
        (y_{L_1}, \psi, y_{L_2}) (t') d t'
      \end{split}
    \end{equation}
    whose kernel $P_{k k'}^{\pm} (t, t')$ has Fourier transform $\widehat{P_{k
    k'}^{\pm}} (\tau, \tau')$, which satisfies $|k-k'| \les L_{\max}$ and
    \begin{equation}
      \label{Eqn:6} 
      \int_{\mathbb{R}^2} \langle \tau \rangle^{2 (1 - b)}
      \langle \tau' \rangle^{- 2 b} \| \widehat{P^{\pm}_{k k'}} (\tau, \tau')
      \|_{k \to  k'}^2 d \tau d \tau' \lesssim
      T^{2\theta} L^{-  6\delta_0 }_{\max},
    \end{equation}
    \begin{equation}
      \label{Eqn:7} 
      \int_{\mathbb{R}^2} \langle \tau \rangle^{2 (1 - b)}
      \langle \tau' \rangle^{- 2 b} \| \widehat{P^{\pm}_{k k'}} (\tau, \tau')
      \|_{k k'}^2 d \tau d \tau' \lesssim T^{2\theta} N^{1 +
       \gamma_0 } L^{- 1 - 4\g_0}_{\max},
    \end{equation}
    \begin{equation}
      \label{Eqn:8} 
      \int_{\mathbb{R}^2} \langle \tau \rangle^{2 (1 - b)}
      \langle \tau' \rangle^{- 2 b} \Big\| \Big( 1 + \frac{| k - k'
      |}{{ {L_{\max}}}} \Big)^{\kappa} \widehat{P^{\pm}_{k k'}}
      (\tau, \tau') \Big\|_{k k'}^2  d \tau d \tau' \lesssim T^{2\theta} N^{2},
    \end{equation}
    where $L_{\max} = \max (L_1, L_2)$.
  \end{itemize}

\noi 
In the above definition, we choose the parameters $0 < \eps \ll \delta \ll \gamma_0  \ll \gamma \ll  \delta_0 \ll \al - 1$ and $\kk \gg 1$. 
\end{definition}

\begin{remark}
\rm 
Random averaging operators $\mathcal P^\pm$ play a crucial role in our later analysis. These operators were first introduced in \cite{DNY2} for Schr\"odinger equations; see Section \ref{Sec:raos} for further details. In \cite[Proposition 3.3]{DNY2}, it is observed that the operator $\mathcal P^+$ enjoys a better bound.
A similar better estimate for the $\mathcal P^+$ than \eqref{Eqn:7} also holds; see Section \ref{Sec:raos} for details. However, this improved estimate for $\mathcal P^+$ does not play an essential role in this paper.
\end{remark} 

Now, by using Definition \ref{Def:LocalM}, we are ready to state the main a priori estimate, which is the key ingredient in proving Theorem \ref{THM:main}. 
\begin{proposition}
  \label{PROP:main}
  Given $0 < T \ll 1$, 
  the probability that $\text{{{\em
  {\textsf{Local}}\/}}} (M / 2)$ holds but $\text{{{\em {\textsf{Local}}\/}}}
  (M)$ does not hold is less than $C_{\theta} e^{- (T^{- 1}
  M)^{\theta}}$ for some $\theta, C_{\theta} > 0$.
\end{proposition}

The proof of Proposition \ref{PROP:main} will occupy Section \ref{Sec:raos} and Section \ref{Sec:rems}.

\subsection{Proof of the main theorem}
Let us prove Theorem \ref{THM:main} by assuming Proposition \ref{PROP:main}.

\begin{proof}[Proof of Theorem \ref{THM:main}]
From Proposition \ref{PROP:main}, it follows that $T^{-1}$-certainly, the event $\text{{\textsf{Local}}} (M)$, defined in Definition \ref{Def:LocalM}, holds for any $M$. 
Recall the decomposition \eqref{Eqn:deco} that
\begin{align} 
\begin{split}
 y_N & = \chi(t) F_N + \sum_{1 \le  L \le  L_N} \mathfrak{h}^{N, L} [\chi(\cdot) F_N] + z_N (t) \\
      & = \chi(t) F_N + \sum_{1 \le  L \le  L_N} \wt{\mathfrak{h}}^{N, L}
      [F_N] + z_N (t)
    \end{split} 
    \label{decoyN}
  \end{align}
where $F_N$ is given in \eqref{Eqn:FN}, $\mathfrak{h}^{N, L}$ and $\wt{\mathfrak{h}}^{N, L}$ are the operators given in \eqref{Eqn:hNL} and \eqref{wthNL}, and $z_N$ is the solution to \eqref{Eqn:newzN}. 
From the definition of $\text{{\textsf{Local}}} (M)$ in Definition \ref{Def:LocalM}, we see that for each $N$, the decomposition of $y_N$ in the above is well defined, i.e. both $\mathfrak{h}^{N, L}$ and $z_N$ are well defined.
To prove the local well-posedness part of Theorem \ref{THM:main}, it suffices to justify the convergence of the summation $\sum y_N$ in some proper sense.

From \eqref{GFF:reg} and \eqref{Eqn:FN}, we see that 
\begin{align}
\label{con_FN}
u_0^\o = \sum_{N\ge 1} F_N  \quad \textup{ converges in }H^{\frac{\al-1}2 -} (\T). 
\end{align}
 
From \eqref{Eqn:keh}, 
we also have that 
\begin{align}
    \label{hkk}
    \wt{\mathfrak{h}}^{N, L}
      [F_N] = \sum_k
      e^{i  k \cdot x} \sum_{ k'} h_{k
      k'}^{N, L} (t) \frac{g_{k'}}{\jbb{k'}^{\frac{\alpha}{2}}} ,
\end{align}

\noi 
where $h_{k k'}^{N, L}$, the kernel of $\wt{\mathfrak{h}}^{N, L}$ given in \eqref{Eqn:keh}, is independent of $\{g_{k'}\}_{N/2 < \jb{{k'}} \le N}$.
Therefore, by using 
Proposition \ref{Prop:te} (as well as Remark \ref{RMK:abs}), we obtain
\begin{align}
    \label{con_aver}
     \begin{split}
      \| \wt{  \mathfrak{h}}^{N, L} [F_N] \|_{X^b}^2 & = \int_{\R} \jb{\tau}^{2b} \bigg\| \sum_{ k' } \ft{h_{k k'}^{N, L}} (\tau) \frac{g_{k'}}{\jbb{k'}^{\frac{\alpha}{2}}}
      \bigg\|_{\ell_k^2}^2 d \tau \\
    & \les  N^{- \al} \int_{\R} \jb{\tau}^{2b}  \big\| \ft{h_{k k'}^{N, L}} (\tau)   \big \|_{k k'}^2  d \tau \\
      &  = N^{- \al}  \| \wt{\mathfrak{h}}^{N, L} \|^2_{Z^b (J)} \\
      & \lesssim N^{ 1 -
      \alpha  + 2\gamma_0} L^{- 1} ,
    \end{split}    
\end{align}

\noi 
where we used \eqref{Eqn:1} in the last step.
Then from \eqref{Eqn:2}, by choosing $\kappa \gg 1$ large enough,  we have $|k-k'| \le LN^\eps$ for some $\eps \ll 1$. This together with \eqref{hkk} implies that the Fourier support of $\wt{\mathfrak{h}}^{N, L} [F_N]$ is $N/4 < |k| \le 2N$. Therefore, from \eqref{con_aver} we have that 
\begin{align*} 
    \begin{split}
    \bigg\| \sum_{1 \le  L \le  L_N} \wt{\mathfrak{h}}^{N, L} [F_N]
    \bigg\|_{C^0_J H^{\frac{\alpha-1}{2} - \gamma}_x } 
    & \les N^{\frac{\alpha-1}{2} - \gamma}\bigg\| \sum_{1 \le  L \le  L_N} \wt{\mathfrak{h}}^{N, L} [F_N]
    \bigg\|_{X^b(J)} \\ 
    & \les N^{\frac{\alpha-1}{2} - \gamma} \sum_{1 \le  L \le  L_N}  \| \wt{\mathfrak{h}}^{N, L} [F_N] \|_{X^b (J)} \\ 
    & \lesssim N^{ \g_0 - \g} ,
    \end{split}
\end{align*}

\noi 
where we used the fact that $b > \frac12$. Also, recalling that $\g_0 \ll \g$,  we thus conclude that
\begin{align}
    \label{con_aver3}
    \sum_{N \ge 1} \sum_{1 \le  L \le  L_N} \mathfrak{h}^{N, L}    [ \chi(\cdot) F_N] \quad \textup{ converges in } C(J; H^{\frac{\al -1}2  -\g} (\T)) \subset  C (J;L^2(\T))
\end{align}

\noi 
by choosing $\g \ll 1$. 

From \eqref{Eqn:3}, we see that 
\begin{align}
    \label{con_zN}
    \sum_{N\ge 1} z_N \quad \textup{ converges in } X^{\frac12 + \g -,b} (J),
\end{align}

\noi  
and thus in $C (J; H^{\frac12 + \g-} (\T))$ for some short time interval $J$. 

Finally, by collecting \eqref{decoyN}, \eqref{con_FN}, \eqref{con_aver3}, and \eqref{con_zN}, we conclude that the sequence
\[
u_N  = \sum_{M= 1}^N y_M
\]

\noi 
converges in $C(J; H^{\frac{\al -1}2  -\g} (\T)) \subset  C (J;L^2(\T))$; i.e. there exists an unique $u \in C(J; H^{\frac{\al -1}2  -\g} (\T)) $ such that
\[
u = \lim_N u_N = \lim_N \sum_{M= 1}^N y_M
\]

\noi 
converges in $C(J; H^{\frac{\al -1}2  -\g} (\T))$.
\end{proof}

\begin{remark}
\rm 
\label{RMK:meshing}
The rigorous proof of \eqref{con_aver} would require a {\itshape meshing argument}. 
As the application of meshing argument has become a standard procedure nowadays, we only sketch its key ideas here. 
For detailed argument and further discussion, see \cite[Remark 4.3]{DNY2}, \cite[Subsection 6.1]{DNY3}, and \cite[Subsection 3.4]{DNY4}. 
First, we may restrict to low modulation $|\tau|\leq N^{\kappa}$ for some $\kappa\gg1$, as otherwise, we may gain considerable smoothing. Second, we divide the big interval $\{\tau\in\mathbb{R};|\tau|\leq N^{\kappa}\}$ into small intervals of length $\nu:=\exp(-(\log N)^6)$. 
Third, by taking averages
on these small intervals and using Poincar\'e inequality, 
we can approximate the tensor $\widetilde{\mathfrak{h}}^{N,L}= (\widehat{h^{N,L}_{kk'}}(\tau) )_{k,k'\in \Z}$ by  $\widetilde{\mathfrak{h}}_{\text{avg}}^{N,L}=(\widehat{(h^{N,L}_{\text{avg}})_{kk'}}(\tau))_{k,k'\in \Z}$ in the sense that
\[
\|\widetilde{\mathfrak{h}}^{N,L} - \widetilde{\mathfrak{h}}^{N,L}_{\text{avg}} \|_{Z^b(J)} \les\, \nu\,\bigg( \int_{\mathbb{R}} \jb{\tau}^{2b} \|  \ft{h^{N,L}_{kk'}} (\tau)\|_{kk'}^2 d\tau + \int_{\mathbb{R}}  \jb{\tau}^{2b}\|  \partial_\tau \ft{h^{N,L}_{kk'}} (\tau)\|_{kk'}^2 d\tau \bigg)^{\frac12}.
\]
Once the above is properly controlled, it suffices to consider at most $\exp(\kappa(\log N)^6)$ different values of $\tau$. This, in particular, allows us to obtain \eqref{con_aver} after removing an exceptional set whose probability is still $O (e^{-(T^{-1}M)^{\theta}})$.
Roughly speaking, we may fix the $\tau$ variable in the computation such as \eqref{con_aver}.
Similar meshing arguments would be needed when we reduce to \eqref{Eqn:ktok}, \eqref{Eqn:eykk}, \eqref{Eqn:ktok-}, \eqref{Eqn:eykk-}, and \eqref{Eqn:ex}, from their time-dependent counterpart.
\end{remark}

\subsection{Reformulation}
\label{Sub:refo}
It remains to prove Proposition \ref{PROP:main}.
As explained at the beginning of this section, the main strategy is to proceed via induction. 
Recall that $\Pi_{1/2} = 0$, and thus $u_{1/2} = y_{1/2} = F_{1/2} = z_{1/2} = 0$,  $\mathcal H^{N,1/2} = \Pi_N$, and $\psi_{N,1/2} = \chi(t)F_N$.
Therefore, $\text{{\textsf{Local}}} \left( \frac{1}{2} \right)$ holds trivially. 
Here and in what follows, $\text{{\textsf{Local}}} \left(M \right)$ is given in Definition \ref{Def:LocalM}.

In what follows, let us assume $\text{{\textsf{Local}}} \left( M' \right)$ holds for all $M' \le \frac{M}{2}$ for some $M\ge 1$.
Then, the aim is to show that $\text{{\textsf{Local}}} \left({M} \right)$ holds with large probability; i.e. the probability of $\textup{{{{\textsf{Local}}\/}}} (M)$ does not hold is less than $C_{\theta} e^{- (T^{- 1} M)^{\theta}}$ for some $\theta, C_{\theta} > 0$.

Before proceeding to show $\text{{\textsf{Local}}} \left({M} \right)$, let us reformulate the estimates given in $\text{{\textsf{Local}}} \left( \frac{M}{2} \right)$. 
From $\text{{\textsf{Local}}} \left( \frac{M}{2} \right)$, we have \eqref{Eqn:1}, \eqref{Eqn:2}, and \eqref{Eqn:3} with $N\le \frac{M}2$ and $L< \min(\frac{M}2, N^{1-\dl})$. This together with \eqref{Eqn:deco} implies that 
\begin{align}
    \label{Eqn:ym}
     y_L (t) = \chi (t) F_L + \sum_{1 \le  R <  L^{1-\dl}} \wt{\mathfrak{h}}^{L, R}
      [F_L] (t) + z_L (t),
\end{align} 

\noi 
is well defined for all $L \le \frac{M}2$ and $t\in J = [-T,T]$. 
To show that $\text{{\textsf{Local}}} \left({M} \right)$-(ii) holds, that is, \eqref{Eqn:3}, we substitute $y_L$ for $L \le \frac{M}2$ from \eqref{Eqn:ym} into \eqref{Eqn:newzN} and then solve for $z_M$. Likewise, to show that $\text{{\textsf{Local}}} \left({M} \right)$-(iii) holds, that is, \eqref{Eqn:6}, \eqref{Eqn:7}, and \eqref{Eqn:8} with $L_1, L_2 < N \le M$, we substitute $y_L$ from \eqref{Eqn:ym} into the random averaging operators $\mathcal P^{\pm}$ in \eqref{Eqn:5}. To simplify the notations, we can assume that $y_L$ is either $\chi (t) F_L$, $\wt{\mathfrak{h}}^{L, R} (F_{L})$, or $z_L$, depending on the decomposition \eqref{Eqn:ym}.
In particular, we assume the input $w_L$ is one of the following three types:
\begin{itemize}
  \item Type (G), we have
  \[ 
    ({w_{L}})_k (t) := \ind_{L / 2 <
    \jb{k} \le  L}  \frac{g_k (\omega)}{\jbb{k}^{\alpha / 2}} {\chi} (t) . 
  \]
  \item Type (C), where
  \begin{equation}
    \label{Eqn:C1} 
    ({w_{L}})_k (t) := \sum_{ k'} {h_{k k'}^{L, R}} (t) \frac{g_{k'}
    (\omega)}{\jbb{k'}^{\alpha / 2}},
  \end{equation}
  with $h_{k k'}^{L, R}$  being  
  supported in the set $\{ (k, k') ; L / 2 < \jb{k'} \le  L \}$, $\mathcal{B}_{\le  R}$-measurable for some
  $R <  L^{1-\dl}$, and satisfying the bounds
  \begin{equation}
    \label{Eqn:e1} 
    \begin{split}
      & \| \langle \tau \rangle^b \ft{h_{k k'}^{L, R}} (\tau) \|_{L_{\tau}^2
      (\ell_k^2 \to  \ell_{k'}^2)} \lesssim R^{- \delta_0},\\
      & \| \langle \tau \rangle^b \ft{h_{k k'}^{L, R}} (\tau) \|_{L_{\tau}^2 \ell_{k k'}^2      } \lesssim L^{\frac{1}{2} + \gamma_0} R^{- \frac{1}{2}},
    \end{split}
  \end{equation}
  for some $b > \frac12$. 
  Moreover, from
  \eqref{Eqn:2} we may assume that $h_{k k'}^{L, R}$ is supported in  \{$|
  k - k' | \lesssim L^{\eps} R$\} for $\eps \ll 1$ that only depends on $\kk$.
  
  \item Type (D), where $(w_{L})_k$ is supported in $\{ k\in \Z; \jb{  k } \le L
  \}$ and satisfies
  \begin{equation}
    \label{Eqn:e3} 
    \| w_{L} \|_{X^b} \lesssim L^{- \frac{1}{2} -
    \gamma} ,
  \end{equation}

  \noi 
  for some $b > \frac12$, provided $1 \le  L \le  \frac M2$.
\end{itemize}

\begin{remark}
  {\rm
    By letting $h^{L, 1 / 2}_{k k'} := \ind_{L / 2 < \jb{k} \le L}  \ind_{k = k'} \cdot {\chi} (t)$,
    we can view type (G) input as a special case of type (C). 
  Therefore, we will only consider type (C) and type (D) in what follows.}
\end{remark}

\subsection{Unitary property}
\label{SUB:unitary}
In this subsection, we prove a crucial cancellation property of $\mathcal H^{N,L}$, originated from the $L^2$ conservation of the linear equation \eqref{Eqn:l+}.
We first show that  the operator $\wt{\mathcal H}^{N,L}$ given in \eqref{wtH} is well defined, provided 
\begin{align}
    \label{smallP}
    \big\| \mathcal P^{N,L} \big\|_{Y^{b,b}} < 1.
\end{align}   

\noi 
Here the operator norm $Y^{b,b}$ is given in \eqref{Ybb}.
From \eqref{Eqn:calH} and \eqref{smallP}, we have
\begin{align*}
    \begin{split}
    \| \wt{\mathcal H}^{N,L}  [\Psi_0] \|_{X^b} & = \| \mathcal H^{N,L} [\chi (\cdot) \Psi_0] \|_{X^b} = \| (1 - \mathcal P^{N,L} )^{-1} [ \chi (\cdot) \Psi_0] \|_{X^b} \\ 
    & \le \|   \chi (\cdot) \Psi_0  \|_{X^b}  + \sum_{n \ge 1} \| (\mathcal P^{N,L} )^{n} [ \chi (\cdot) \Psi_0] \|_{X^b} \\ 
    & \le \|   \chi (\cdot) \Psi_0  \|_{X^b}  + \sum_{n \ge 1} \|  \mathcal P^{N,L}  \|_{Y^{b,b}}^{n} \|\chi (\cdot) \Psi_0] \|_{X^b} ,
    \end{split}
\end{align*}

\noi 
which implies that $\wt{\mathcal H}^{N,L}$ is well defined, provided \eqref{smallP}. 
Since $b > \frac12$, by the Sobolev embedding, we have that $\wt{\mathcal H}^{N,L} [ \Psi_0] \in C_T L^2_x$.
We have the following group property on $[-T,T]$, i.e. 
\[
\Psi(t_2, \Psi(t_1)) = \Psi(t_1 + t_2, \Psi(0)),
\]

\noi 
where $\Psi(t, W)$ is the solution to \eqref{Eqn:l+1} with initial data $\Psi_0 = W$.
Thus $\wt{\mathcal H}^{N,L}(t)$ is invertible since $\wt{\mathcal H}^{N,L} (t) \circ \wt{\mathcal H}^{N,L} (-t) = {\rm Id}$.
Finally, we note that \eqref{smallP} is guaranteed by \eqref{Eqn:6}.
As a consequence, we also have the following unitary property.

\begin{lemma}
  \label{LEM:uni}
  Given $M > \frac12$, assume $\text{{{\em {\textsf{Local}}\/}}} (M)$ given in Definition \ref{Def:LocalM} holds.
  For $L < \min (M, N^{1-\dl})$, there exists $T \ll 1$ such that for each $|t| \le T$, the matrix
  $\wt{\mathcal{H}}^{N, L} (t) = \{ \wt{H}_{k k'}^{N, L}(t) \}_{k k'}$ is unitary, i.e.
  \begin{align}
    \label{uni}  
    \sum_k \wt{H}_{k_1 k}^{N, L} (t) \overline{\wt{H}^{N, L}_{k_2 k}} (t) d\tau = \delta_{k_1} (k_2) 
  \end{align}
  where $\wt{\mathcal{H}}^{N, L}$ is the linear mapping defined in
  \eqref{wtH}.
\end{lemma}

Bourgain \cite{B97} first observed this property in the context of Hartree-type NLS. 
See also \cite{DNY4}.
Here we give a proof for reader's convenience. 

\begin{proof}
From \eqref{Eqn:l+}, we know that
 \[
      \partial_t \sum_k | \Psi_k |^2 = 2 \sum_k \ensuremath{\operatorname{Re}} (\overline{\Psi_k} \partial_t \Psi_k) = 4\ensuremath{\operatorname{Im}} \bigg( \sum_k
      \sum_{\substack{
        k_1 - k_2 + k_3 = k\\
        k_2 \not\in \{ k_1, k_3 \}
      }} e^{i  t' \Phi} \cdot (u_L)_{k_1}
      \overline{(u_L)_{k_2}} \overline{\Psi_k} \Psi_{k_3} \bigg) .
 \]
 By swapping $(k, k_1, k_2, k_3) \mapsto (k_3, k_2, k_1, k)$ in the summand,
 we see that we are taking imaginary part for reals, which gives zero. Namely
 \[ \partial_t \sum_k | \Psi_k |^2 = 0, \]

 \noi
 for $|t| \le T$.
It then follows that
 \[ \| \Psi (t) \|_{L^2 (\mathbb{T})}^2 = \| \Psi_0 \|_{L^2 (\mathbb{T})}^2,
    \text{\quad for all } |t| \le T. \]

\noi 
This shows that $\wt{\mathcal H}^{N,L} (t)$ is isometric on $L^2$ for $t\in [-T,T]$.
Furthermore, given $t\in [-T,T]$, we know that $\wt {\mathcal H}^{N,L} (t)$ is invertible and thus surjective. Therefore, for fixed $t\in [-T,T]$, the linear operator $\wt{\mathcal H}^{N,L} (t)$ is unitary\footnote{A linear map is unitary if it is surjective and isometric.}, i.e. 
 \[
\big(\wt{\mathcal H}^{N,L}   (t) \big)^* \wt{\mathcal H}^{N,L}  (t) = \wt{\mathcal H}^{N,L}  (t) \big(\wt{\mathcal H}^{N,L}  (t) \big)^* 
 =\text{Id},
 \]

\noi 
for $|t| \le T$, which finishes the proof.
\end{proof} 

We have the following corollary of Lemma \ref{LEM:uni}.

\begin{corollary}
\label{COR:cancel}
    With the same assumption as in Lemma \ref{LEM:uni}, we have
    \begin{align} 
    \label{Can}
    \sum_{k} \sum_{L_1,L_2 \le L} h^{N,L_1}_{k_1k} (t) 
 \cj{h^{N,L_2}_{k_2k} (t) 
 } \frac1{\jbb{k}^{\frac\al2}} = \sum_{k} \sum_{L_1,L_2 \le L} h^{N,L_1}_{k_1k} (t)  \cj{h^{N,L_2}_{k_2k} (t) } \bigg( \frac1{\jbb{k}^{\frac\al2}} -\frac1{\jbb{k_1}^{\frac\al2}} \bigg)
 \end{align}

    \noi 
    provided $|t| \le T$ and $k_1 \neq k_2$.
\end{corollary}

\begin{proof}
From \eqref{Eqn:keh}, we have that
\begin{align}
\label{Can1}
\sum_{L_i \le L} h_{k_ik}^{N,L_i} (t) = \wt{\mathcal{H}}^{N, L}_{k_ik} (t) ,
\end{align}

\noi 
for $i = 1,2$.
Then it follows that 
\begin{align*}
\begin{split}
\sum_{L_1, L_2 \le L}  \sum_{k}   h_{k_1k}^{N,L_1} (t) \cj{h_{k_2k}^{N,L_2} } (t) =   \sum_{k}    \wt{\mathcal{H}}^{N, L}_{k_1k} (t) \wt{\mathcal{H}}^{N, L}_{k_2k} (t)  = \dl_{k_1} (k_2) 
\end{split}
\end{align*}

\noi 
for $|t| \le T$. Thus \eqref{Can} follows.
\end{proof}

\begin{remark}\rm
\label{RMK:cancel}
Without loss of generality, we may assume that  
\[
\sum_{k} h^{N,L_1}_{k_1k} (t) 
 \cj{h^{N,L_2}_{k_2k} (t) 
 } \frac1{\jbb{k}^{\frac\al2}} = \sum_{k} h^{N,L_1}_{k_1k} (t)  \cj{h^{N,L_2}_{k_2k} (t) } \bigg( \frac1{\jbb{k}^{\frac\al2}} -\frac1{\jbb{k_1}^{\frac\al2}} \bigg),
\]

\noi 
in our later analysis. The reason is that from Corollary \ref{COR:cancel}, we can always add
\[
0 = \sum_{k} \sum_{L_1,L_2 \le L} h^{N,L_1}_{k_1k} (t)  \cj{h^{N,L_2}_{k_2k} (t) }  \frac1{\jbb{k_1}^{\frac\al2}}  ,
\] 

\noi 
from the beginning of our analysis.
\end{remark}

\section{The random averaging operators}
\label{Sec:raos}

This and the next section will be devoted to the proof of Proposition \ref{PROP:main}.
The plan of the proof is as follows;
we first further reduce the problem to a time-dependent problem in Subsection \ref{SUB:redu};
then we prove \eqref{Eqn:6}--\eqref{Eqn:8} of $\text{{\textsf{Local}}} \left( M
\right)$ in Subsections \ref{SUB:LLH} and \ref{SUB:LHL}; in Subsection \ref{SUB:RAO},
we verify \eqref{Eqn:1} and \eqref{Eqn:2} of $\text{{\textsf{Local}}} \left( M
\right)$. 
We defer the proof of \eqref{Eqn:3} to the next section.

\subsection{Further reduction}
\label{SUB:redu}
To make the proof more clear, we can recast the estimates \eqref{Eqn:6} -- \eqref{Eqn:8} in a simpler way. The key point is that the main solution space $X^b$ for $b > \frac12$ embeds in $C_t $ in the temporal variable, so we can ignore the $t$ variable in most of our analysis.

Recall operators $\mathcal{P}^{\pm}$ defined in \eqref{Eqn:5}. Due to the factor $\chi (t) = \eta (T^{- 1} t)$, by Proposition
\ref{Prop:stb}, we can gain a small factor $T^{\theta}$. 
Recall that $\widehat{P^{\pm}_{k k'}} (\tau,
\tau')$ is the temporal Fourier transform of the kernel $P^{\pm}_{k k'} (t,
t')$ of $\mathcal{P}^{\pm}$ in \eqref{Eqn:5}, i.e.
\[ 
  \widehat{(\mathcal{P}^{\pm} \psi)_k} (\tau) = \sum_{k'}
  \int_{\mathbb{R}} \widehat{P^{\pm}_{k k'}} (\tau, \tau')
  \widehat{\psi_{k'}} (\tau') d \tau' . 
\]
We decompose $\{ | k' | \sim N_3 \}$ into intervals $B'_i$ of lengths
$L_{\max} = \max (L_1, L_2)$, and then write
\[ 
  \widehat{P^{\pm}_{k k'}} (\tau, \tau') = \sum_i \widehat{P^{\pm}_{k k'}}
  (\tau, \tau') \ind_{k' \in B_i'} . 
\]

\noi 
Here, we may further restrict to the low modulation case $||k|^\alpha-|k'|^\alpha|\lesssim L_{\max}^\kappa$ for some $\kappa \gg 1$, as otherwise, we may gain considerable regularity from an $X^{s,b}$ argument. With this further restriction, we note that the number of intervals $B_i'$ is at most of $O(L_{\max}^{c \kappa})$.
See \cite[Proposition 4.15]{DNY3} for a similar argument.
Due to the fact that $| k - k' |\lesssim L_{\max}$, we
see that
\begin{equation}
  \label{Eqn:lm} 
  \| \widehat{P^{\pm}_{k k'}} (\tau, \tau') \|_{k \to 
  k'} \lesssim \sup_i \| \widehat{P^{\pm}_{k k'}} (\tau, \tau')
  \ind_{k' \in B_i'} \|_{k \to  k'} .
\end{equation}

\noi 
Therefore,
in what follows, we may assume that $k$ and $k'$ are both localized in intervals of length $L_{\max} = \max (L_1, L_2)$.

We start with the operator $\mathcal{P}^+$.
By using \eqref{Eqn:M}, \eqref{Eqn:Ik}, and
\eqref{Eqn:5}, we may rewrite the temporal
Fourier transform of the kernel as
\begin{equation}
  \label{Eqn:pkkp} 
  \begin{split}
    & \widehat{P^+_{k k'}} (\tau, \tau')\\
    & \quad = - i  \sum_{\substack{
      k_1 - k_2 = k - k'\\
      k_2 \not\in \{ k_1, k' \}
    }} \int_{\mathbb{R}^2} \mathcal{K} (\tau, \Phi + \tau' +
    \tau_1 - \tau_2) \cdot ( \widehat{y_{L_1}} )_{k_1} (\tau_1)
    \cdot \overline{(\widehat{y_{L_2}})_{k_2}} (\tau_2)
    d \tau_1 d \tau_2\\
    & \quad = - i  \sum_{m \in \mathbb{Z}} \int_{\mathbb{R}^2} \sum_{k_1, k_2}
    \mathcal{K} (\tau, m + \Phi - [\Phi] + \tau' + \tau_1 - \tau_2)\\
    & \hspace{3cm} \times \mathrm{T}_{k k_1 k_2 k'}^{\text{b}, m} \cdot ( \widehat{y_{L_1}}
    )_{k_1} (\tau_1) \cdot \overline{(\widehat{y_{L_2}})_{k_2}} (\tau_2) d \tau_1 d \tau_2,
  \end{split}
\end{equation}
where $\Phi$ is defined in \eqref{Eqn:Phi}, $[\Phi]$ is the largest integer
that does not greater than $\Phi$, and the base tensor ${\rm T}^{{\rm b},m}$ is given by
\eqref{baseT}. Then by Minkowski inequality, \eqref{Eqn:pkkp}, and \eqref{Eqn:bk}, we have
\begin{equation}
  \label{Eqn:re2} 
  \begin{split}
   \| \mathcal P^+   & \|_{Y^{1-b,b'}}^2 \les  \int_{\mathbb{R}^2} \langle \tau \rangle^{2 (1 - b)} \langle \tau'
    \rangle^{- 2 b'} \| \widehat{P^+_{k k'}} (\tau, \tau') \|_{k \to 
    k'}^2 d \tau d \tau'\\
    & \lesssim \int_{\mathbb{R}^2} \langle \tau \rangle^{- 2 b} \langle \tau'
    \rangle^{- 2 b'} \bigg( \sum_{m \in \mathbb{Z}} \int_{\mathbb{R}^2} 
    \langle \tau_1 \rangle^{- b} \langle \tau_2 \rangle^{- b}\\
    & \hspace{1cm} \times \Big\|  \sum_{k_1, k_2} \langle \tau - m + \Phi - [\Phi]
    - \tau' - \tau_1 + \tau_2 \rangle^{- 1} \\
    & \hspace{1.5cm} \times \mathrm{T}_{k k_1 k_2
    k'}^{\text{b}, m} \cdot ( \langle \tau_1 \rangle^b \widehat{y_{L_1}}
    )_{k_1} (\tau_1) \cdot \overline{( \langle \tau_2
    \rangle^b \widehat{y_{L_2}} )_{k_2} } (\tau_2)  \Big\|_{k
    \to  k'}  d \tau_1 d \tau_2\bigg)^2 d \tau
    d \tau'\\
    & \lesssim \int_{\mathbb{R}^2} \langle \tau \rangle^{- 2 b} \langle \tau'
    \rangle^{- 2 b'} \bigg( \sum_{m \in \mathbb{Z}} \int_{\mathbb{R}^2} 
    \langle \tau - m - \tau' - \tau_1 + \tau_2 \rangle^{- 1} \langle \tau_1
    \rangle^{- b} \langle \tau_2 \rangle^{- b}\\
    & \hspace{1cm}  \times \Big\|\sum_{k_1, k_2} \mathrm{T}_{k k_1
    k_2 k'}^{\text{b}, m} \cdot ( \langle \tau_1 \rangle^b
    \widehat{y_{L_1}} )_{k_1} (\tau_1) \cdot \overline{( \langle \tau_2
    \rangle^b \widehat{y_{L_2}} )_{k_2} }(\tau_2) \Big\|_{k
    \to  k'} d \tau_1 d \tau_2 \bigg)^2 d
    \tau d \tau',
  \end{split}
\end{equation}

\noi 
where $b > b' > \frac12$.
To deal with the summation over $m$ in \eqref{Eqn:re2}, we need to localize the
frequencies $k, k_1, k_2, k'$. To be more precise, if $m = [\Phi] \in
\mathbb{Z}$ ranges over an interval of size $R$, then from our construction,
we have
\begin{equation}
  \label{Eqn:gener} 
  \sum_m \langle \tau - m - \tau' - \tau_1 + \tau_2
  \rangle^{- 1} \lesssim \log (1 + R) .
\end{equation}
In what follows, we shall use some special cases of \eqref{Eqn:gener}. Recall
$\Phi = |{k_1}|^{\alpha} - |{k_2}|^{\alpha} + |k'|^{\alpha} - |{k}|^{\alpha}$ from \eqref{Eqn:Phi} and
the relation $k = k_1 - k_2 + k'$. If we assume $| k_1 | \lesssim L_1$, $| k_2
| \lesssim L_2$, and $k \in B_i'$, where $B_i'$ is the ball considered in
\eqref{Eqn:lm}, then we have a refined bound
\begin{equation}
  \label{Eqn:sm} 
  \sum_m \langle \tau - m - \tau' - \tau_1 + \tau_2 \rangle^{-
  1} \lesssim \log (1 + L_{\max}) .
\end{equation}

\noi 
By using \eqref{Eqn:sm} and Cauchy-Schwarz inequality in $\tau_1,
\tau_2$ integrations, we may bound \eqref{Eqn:re2} by
\begin{equation}
  \label{Eqn:goal1} 
  (\log (1 + L_{\max}))^2 \Big\| \sum_{k_1, k_2}
  \mathrm{T}_{k k_1 k_2 k'}^{\text{b}, m} \cdot ( \langle \tau_1
  \rangle^b \widehat{y_{L_1}} )_{k_1} (\tau_1) \cdot \overline{( \langle \tau_2
    \rangle^b \widehat{y_{L_2}} )_{k_2} } (\tau_2)
  \Big\|^2_{L_{\tau_1 \tau_2}^2 (k \to  k')} ,
\end{equation}
where we used that $b, b' > \frac12$.
From the discussion in Subsection \ref{Sub:refo}, we proceed with estimating \eqref{Eqn:goal1} by replacing $y_{L_j}$ with an input $w_{L_j}$ of either type (C) or type (D), respectively. 

From the definition of type (C), type (D),
Remark \ref{RMK:meshing},
and \eqref{Eqn:goal1}, we may further simplify the estimate \eqref{Eqn:6} to a time-independent estimate.
To be more precise, to prove \eqref{Eqn:6},
it suffices to show
\begin{equation}
  \label{Eqn:ktok} 
  \| \mathscr{Y}_{k, k'}^{+} \|_{k \to  k'} \lesssim L^{- 4
  \delta_0 }_{\max},
\end{equation}
where $\mathscr{Y}^+_{k, k'}$ is the time-dependent random tensor given by
\begin{equation}
  \label{Eqn:ykk} 
  \mathscr{Y}_{k, k'}^{+} = \sum_{k_1, k_2} \mathrm{T}_{k k_1 k_2
  k'}^{\text{b}, m} \cdot (w_{L_1})_{k_1}  \overline{(w_{L_2})_{k_2}}.
\end{equation}

\noi 
Here $w_{L_j}$ for $j \in \{ 1, 2 \}$ are of the following two types, with a
slight abuse of notation, we still call them type (C) and type (D).  
\begin{itemize}
  \item Type (C), where
  \[ 
    (w_{L_j})_{k_j} := \sum_{ {k_j'} } h_{k_j
    k_j'}^{L_j, R_j} (\omega) \frac{g_{k_j'} (\omega)}{\jbb{k_j'}^{\alpha / 2}},
  \]
  with $h_{k_j k_j'}^{L_j, R_j}$ supported in the set $\{  k_j'; L_j /
  2 < \jb{k_j'} \le  L_j \}$ and $\mathcal{B}_{\le 
  R_j}$-measurable for some $R_j \le  L_j^{1 - \delta}$, and satisfying
  the bounds
  \begin{equation}
    \label{Eqn:ee1} 
    \begin{split}
      & \| h_{k_j k_j'}^{L_j, R_j} \|_{\ell_{k_j}^2 \to  \ell_{k_j'}^2}
      \lesssim R_j^{- \delta_0},\\
      & \| h_{k_j k_j'}^{L_j, R_j} \|_{{\ell_{k_j k_j'}^2}} \lesssim
      L_j^{\frac{1}{2} + \gamma_0} R_j^{- \frac{1}{2}},
    \end{split}
  \end{equation}
  for $0< \dl \ll \gamma_0 \ll \gamma \ll \delta_0 \ll \al - 1$. Moreover, from
  \eqref{Eqn:2} we may assume that $h_{k_j k_j'}^{L_j, R_j}$ is supported in 
  \{$| k_j - k_j' | \lesssim L_j^{\eps} R_j$\} for any $\eps > 0$.

  \smallskip
  
  \item Type (D), where $(w_{L_j})_{k_j}$ is supported in $\{k_j \in \Z;  \jb{ k_j } \le 
  L_j \}$, and satisfies
  \begin{equation}
    \label{Eqn:ee3} 
    \| (w_{L_j})_{k_j} \|_{\ell_{k_j}^2} \lesssim L_j^{-
    \frac{1}{2} - \gamma} .
  \end{equation}
\end{itemize}
A similar argument allows us to reduce \eqref{Eqn:7} to
\begin{equation}
  \label{Eqn:eykk} 
  \| \mathscr{Y}_{k, k'}^+ \|_{k k'} \lesssim N^{\frac{1}{2} + \frac{\g_0}3 } L^{{ {- \frac{1}{2} -  3 \g_0}}}_{\max},
\end{equation}
where $\mathscr{Y}_{k, k'}^+$ is as in \eqref{Eqn:ykk} and $0 < \varepsilon' \ll 1$.

The same argument as above also works for the estimate \eqref{Eqn:6} with the operator $\mathcal P^-$, for which we only need to prove the following time-independent random matrix estimates.
\begin{align}
  \label{Eqn:ktok-} 
  \| \mathscr{Y}_{k, k'}^{-} \|_{k \to  k'} & \lesssim L^{- 4
  \delta_0 }_{\max}, \\
    \label{Eqn:eykk-} 
  \| \mathscr{Y}_{k, k'}^{-} \|_{k k'} & \lesssim N^{\frac{1}{2} +
  \frac{\gamma_0}3 } L^{{ {- \frac{1}{2} - 3\g_0}}}_{\max},
\end{align}

\noi 
where the corresponding random tensor $\mathscr{Y}_{k, k'}^{-}$ is given by
\begin{equation}
  \label{Eqn:ykk-} 
  \mathscr{Y}_{k, k'}^{-} = \sum_{k_1, k_3} \mathrm{T}_{k k_1 k'
  k_3}^{\text{b}, m} \cdot (w_{L_1})_{k_1} \left( {w_{L_3}}
  \right)_{k_3}.
\end{equation}

\noi 
Here $w_{L_1}$ and $w_{L_3}$ in \eqref{Eqn:ykk-} are of either type (C) satisfying \eqref{Eqn:ee1}, or type (D) satisfying \eqref{Eqn:ee3}.

The proof of \eqref{Eqn:ktok}, \eqref{Eqn:eykk}, \eqref{Eqn:ktok-}, and \eqref{Eqn:eykk-} will be detailed in Subsections \ref{SUB:LLH} and \ref{SUB:LHL}.
The above argument yields the following conditional proof of \eqref{Eqn:6} -- \eqref{Eqn:8} in $\textup{{\textsf{Local}}} \left( {M} \right)$.

\begin{proof}[Proof of \eqref{Eqn:6} -- \eqref{Eqn:8} in $\textup{{\textsf{Local}}} \left( {M} \right)$]
Given $\text{{\textsf{Local}}} \left( \frac{M}{2} \right)$, we may decompose $y_{L_i}$ into $w_{L_i}$ of type (C) or type (D) for $L_i \le \frac{M}2$. 
From the above argument, 
we have seen that \eqref{Eqn:6} and \eqref{Eqn:7} follows from \eqref{Eqn:ktok}, \eqref{Eqn:eykk}, \eqref{Eqn:ktok-}, and \eqref{Eqn:eykk-}. 
Finally, the estimate \eqref{Eqn:8} follows from \eqref{Eqn:7} and the fact $|k-k'| \les L_{\max}$.
\end{proof}

We conclude this subsection by recording a consequence of \eqref{Eqn:6} -- \eqref{Eqn:8}.

\begin{corollary}
\label{COR:RAO}
Given $\textup{{\textsf{Local}}} \left( {M} \right)$ defined in Definition \ref{Def:LocalM} and the above notations, we have
  \begin{align} 
    \| \mathcal{P}^{\pm} \|_{Y^{b,b}} & \lesssim T^{c\theta} L_{\max}^{- 2 
     \delta_0 }, \label{r_OP} \\
    \| \mathcal{P}^{\pm} \|_{Z^{b,b}} & \lesssim T^{c\theta} N^{\frac{1}{2} + \frac{2 \gamma_0}3 } L^{- \frac{1}{2} - \g_0}_{\max},  \label{R_HS} \\
    \bigg\| \bigg( 1 + \frac{| k - k' |}{L_{\max}} \bigg)^{\kappa}
    \widehat{P_{k k'}^{\pm}} \bigg\|_{Z^{b,b}} & \lesssim T^{c\theta} N, \label{R_W}
  \end{align}

  \noi 
  provided $b > \frac12$ and close to $\frac12$,
  where $Y^{b,b}$ and $Z^{b,b}$ are defined in \eqref{Ybb} and \eqref{Zbb}, respectively.
\end{corollary}

\begin{proof}
From \eqref{Ybb} and \eqref{Eqn:6}, 
we have that 
\begin{align}
\label{r_OPp}
\begin{split}
\|\mathcal P^{\pm}\|_{Y^{1-b,b}} 
\le \bigg( \int_{\mathbb{R}^2} \langle \tau \rangle^{2 (1 - b)}
\langle \tau' \rangle^{- 2 b} \| \widehat{P^{\pm}_{k k'}} (\tau, \tau')
\|_{k \to  k'}^2 d \tau d \tau' \bigg)^{\frac12}  \les T^{\theta} L_{\max}^{- 3\dl_0 },
\end{split}
\end{align}

\noi 
for some $b > \frac12$.
We recall that 
\[
\mathcal{P}^+ [\psi] (t) = - i  \chi(t) \int^t_0 \Pi_N \mathcal{M} (y_{L_1}, y_{L_2}, \psi) (t') d t'.
\]

\noi 
It then follows that 
\begin{align}
    \label{embed1}
    \begin{split}
    \| \mathcal{P}^+ [\psi] (t) &  \|_{X^1}\sim \|\mathcal{P}^+ [\psi] (t) \|_{L^2_{t,x} ([-T,T] \times \T)} + \|\partial_t (\mathcal{P}^+ [\psi] (t) ) \|_{L^2_{t,x} ([-T,T] \times \T)} \\
    & \les L_{\max}^4 \|y_{L_{1}} \|_{X^b}\|y_{L_{2}} \|_{X^b}  \|\psi \|_{X^0}   \les L_{\max}^{12} \|\psi \|_{X^0} ,
    \end{split}
\end{align}

\noi 
where we used that $y_L = u_L - u_{L/2}$ is Fourier supported in $[-L,L]$, and that $u_L = \sum_{L' \le L} y_{L'}$ satisfies $\|y_{L'}\|_{X^b} \les L'$ due to the induction argument.
In particular,
\eqref{embed1} implies that
\begin{align*} 
\|\mathcal P^{\pm}\|_{Y^{1,0}} \les L_{\max}^{12} ,
\end{align*}

\noi 
which together with \eqref{r_OPp} implies \eqref{r_OP} by interpolation.

We turn to the proof of \eqref{R_HS}.
By using \eqref{Eqn:bk} and Minkowski
inequality, it follows that
\[ 
  \begin{split}
    & \big\| \langle \tau \rangle \langle \tau' \rangle^{- b} \ft{ P_{k k'}^{\pm}} (\tau, \tau') \big\|_{L^2_{\tau \tau'}}\\
    & \quad \lesssim \sum_{k_1 - k_2 = k - k'} \int_{\mathbb{R}^2} \left\|
    \frac{1}{\langle \tau' \rangle \langle \tau - \tau_1 + \tau_2 - \tau' -
    \Phi \rangle} \right\|_{L^2_{\tau \tau'}}  | (\widehat{y_{L_1}})_{k_1} (\tau_1) \overline{(
    \widehat{y_{L_2}} )_{k_2} } (\tau_2) | d \tau_1
    d \tau_2\\
    & \quad \lesssim \int_{\mathbb{R}^2} \| (\widehat{y_{L_1}})_{k_1} (\tau_1)
    \|_{\ell_{k_1}^2} \| \overline{(\widehat{y_{L_2}} )_{k_2}}
    (\tau_2) \|_{\ell_{k_2}^2} d \tau_1 d \tau_2\\
    & \quad \lesssim L_{\max}^{12}.
  \end{split} 
\]
Thus by summing above over $| k |, | k' | \lesssim N$ we arrive at
\begin{align}
    \label{embed3}
    \| \mathcal{P}^{\pm} \|_{Z^{1, b}} \lesssim N^2 L_{\max}^{12}.
\end{align} 

\noi 
Then we conclude \eqref{R_HS} by interpolating \eqref{Eqn:7} and \eqref{embed3}. 

The proof for \eqref{R_W} is similar; thus, we omit it. 
\end{proof}

In the next two subsections, we will establish \eqref{Eqn:ktok} and \eqref{Eqn:eykk} for the low-low-high random matrix $\mathscr{Y}_{k, k'}^{+} $ given in \eqref{Eqn:ykk}; and \eqref{Eqn:ktok-} and \eqref{Eqn:eykk-} for the low-low-high random matrix $\mathscr{Y}_{k, k'}^{-} $ given in \eqref{Eqn:ykk-}.

\subsection{Low-low-high random matrix} 
\label{SUB:LLH}
In this subsection, we focus on the low-low-high random matrix $\mathscr{Y}_{k, k'}^{+} $ given in \eqref{Eqn:ykk}, and establish \eqref{Eqn:ktok} and \eqref{Eqn:eykk}.
We consider \eqref{Eqn:ykk} with input  $(w_{L_1}, w_{L_2})$ of types (a) (C, C), (b) (C, D), (c) (D, C), (d) (D, D).

\subsubsection{Case (a)}  
\label{Subsub:cc}
Type (C, C). 
In this case, the matrix \eqref{Eqn:ykk} can be written as
\[ 
  \mathscr{Y}_{k, k'}^{+} = \sum_{k_1, k_2} {\rm T}^{{\rm b},m}_{k k_1 k_2
  k'} \cdot \sum_{k_1', k_2'} h_{k_1 k_1'}^{L_1, R_1} \overline{h_{k_2 k_2'}^{L_2, R_2}}   \frac{g_{k_1'} \cj{g_{k_2'}}}{\jbb{k_1'}^{\frac\al2} \jbb{k_2'}^{\frac\al2}} 
\]
where $\wt{\mathfrak{h}}^{L_j, R_j} = \{ h_{k_j k'_j}^{L_j, R_j} \}$ ($j \in \{ 1,
2 \}$) are either identity map over $L_j / 2 < \jb{k_j} \le  L_j$ or satisfying \eqref{Eqn:ee1}.

We first consider the non-pairing case, i.e. $k'_1 \neq k'_2$, for which we apply Proposition \ref{Prop:te} followed by Proposition \ref{PROP:contr} to get 
  \begin{align} 
  \label{P_CC}
  \begin{split}
    & \| \mathscr{Y}_{k, k'}^{+} \|_{k \to  k'}   \lesssim (L_1 L_2)^{- \alpha / 2} \big ( \| \mathrm{T}_{k k_1
    k_2 k'}^{\text{b}, m} \|_{k k_1 \to  k_2 k'} + \|
    \mathrm{T}_{k k_1 k_2 k'}^{\text{b}, m} \|_{k k_2 \to  k_1
    k'} \\
    & \hspace{2cm} + \| \mathrm{T}_{k k_1 k_2
    k'}^{\text{b}, m} \|_{k k_1 k_2 \to  k'} + \|
    \mathrm{T}_{k k_1 k_2 k'}^{\text{b}, m} \|_{k \to  k_1 k_2
    k'} \big )  \prod_{j=1}^2 \| h_{k_j k_j'}^{L_j, R_j} \|_{k_j \to  k'_j} ,
  \end{split} 
\end{align}

\noi 
which is enough for our purpose.
When $L_1 = L_2$, \eqref{P_CC} follows from Proposition \ref{Prop:te} and then Proposition \ref{PROP:contr} as $k_1' \neq k_2'$.
In what follows, we only consider \eqref{P_CC} for the case where $L_1 \neq L_2$. Suppose $L_1 < L_2$ (the case $L_2 < L_1$ is similar). Then the coefficients $h_{k_2 k_2'}^{L_2, R_2}$ may not be independent of $(g_{k_1'})_{L_1/2<\jb{k_1'} \le L_1}$. Therefore, we cannot use Proposition \ref{Prop:te} to $k_1'$ and  $k_2'$ directly. 
To overcome this obstacle, we first use Proposition \ref{Prop:te} only with respect to $k_2'$ and then Proposition \ref{PROP:contr} to get
\begin{align}
\label{tensor_e1}
\begin{split}
\| &  \mathscr{Y}_{k, k'}^{+} \|_{k \to  k'} 
\les L_2^{- \alpha / 2} 
 \bigg\| \sum_{k_1, k_2} {\rm T}^{{\rm b},m}_{k k_1 k_2 k'} \sum_{\substack{k_1' }} h_{k_1 k_1'}^{L_1, R_1}  
\overline{h_{k_2 k_2'}^{L_2, R_2}}   \frac{g_{k_1'}}{\jbb{k_1'}^{\frac\al2}} \bigg\|_{(kk_2' \to k')\cap (k \to k'k_2')} \\ 
& \hphantom{X} \les L_2^{- \alpha / 2} \bigg\| \sum_{k_1 } {\rm T}^{{\rm b},m}_{k k_1 k_2 k'} \sum_{\substack{ k_1' }} h_{k_1 k_1'}^{L_1, R_1} \frac{g_{k_1'}}{\jbb{k_1'}^{\frac\al2}} \bigg\|_{(kk_2 \to k')\cap (k \to k'k_2)}   \| {h_{k_2 k_2'}^{L_2, R_2}} \|_{k_2 \to  k'_2}  
  \end{split}
\end{align}

\noi 
where $\| \cdot \|_{X\cap Y} = \| \cdot \|_{X } + \| \cdot \|_{ Y}$.
Then, we note that the tensor ${\rm T}^{{\rm b},m}_{k k_1 k_2 k'}  h_{k_1 k_1'}^{L_1, R_1} $ is independent of $(g_{k_1'})_{L_1/2<\jb{k_1'} \le L_1}$, which enables us to apply Proposition \ref{Prop:te} again with respect to $k_1'$,
\begin{align}
\label{tensor_e2}
\begin{split}
\bigg\| & \sum_{k_1 } {\rm T}^{{\rm b},m}_{k k_1 k_2 k'} \sum_{ k_1' } h_{k_1 k_1'}^{L_1, R_1} \frac{g_{k_1'}}{\jbb{k_1'}^{\frac\al2}} \bigg\|_{(kk_2 \to k')\cap (k \to k'k_2)}  \\
& \les L_1^{- \alpha / 2}  \bigg\| \sum_{k_1 } {\rm T}^{{\rm b},m}_{k k_1 k_2 k'}  h_{k_1 k_1'}^{L_1, R_1}  \bigg\|_{(kk_2k_1' \to k')\cap (kk_1' \to k'k_2)\cap (kk_2 \to k'k_1')\cap (k \to k'k_2k_1')} \\
& \les L_1^{- \alpha / 2}  \| {\rm T}^{{\rm b},m}_{k k_1 k_2 k'}  \|_{(kk_2k_1 \to k')\cap (kk_1 \to k'k_2)\cap (kk_2 \to k'k_1)\cap (k \to k'k_2k_1)} \| h_{k_1 k_1'}^{L_1, R_1} \|_{k_1 \to  k'_1} .
\end{split}
\end{align}

\noi 
By collecting \eqref{tensor_e1} and \eqref{tensor_e2}, we conclude \eqref{P_CC} for $L_1 < L_2$.
By  Lemma \ref{LEM:tensor2}, Lemma \ref{LEM:tensor3}, with $N_1 = L_1$, $N_2 = L_2$, and $N_3 = N$, we obtain
\begin{align}
    \label{CC_e3}
    \begin{split}
    \|& {\rm T}^{{\rm b},m}_{k k_1 k_2 k'}  \|_{(kk_2k_1 \to k')\cap (kk_1 \to k'k_2)\cap (kk_2 \to k'k_1)\cap (k \to k'k_2k_1)} \\ 
    & = \|  \mathrm{T}_{k k_1
    k_2 k'}^{\text{b}, m} \|_{k k_1 \to  k_2 k'} + \|
    \mathrm{T}_{k k_1 k_2 k'}^{\text{b}, m} \|_{k k_2 \to  k_1 k'} \\ 
    & \hphantom{XXXXX} + \| \mathrm{T}_{k k_1 k_2
    k'}^{\text{b}, m} \|_{k k_1 k_2 \to  k'} + \|
    \mathrm{T}_{k k_1 k_2 k'}^{\text{b}, m} \|_{k \to  k_1 k_2 k'} \\
    & \les L_2^{1 - \frac{\alpha}2 } L_1^{1 - \frac\al2}  + L_1^{\frac12 - \frac{\alpha}4 } L_2^{\frac12 - \frac\al4} + L_1^{1 - \frac\alpha2} (\log L_2)^{\frac12}  +  L_2^{\frac12}  + L_2^{1 - \frac\al2} (\log L_1)^{\frac12}  + L_1^{\frac12} \\ 
    & \les (L_1L_2)^{\frac12}. 
    \end{split}
\end{align}

\noi 
From \eqref{P_CC}, \eqref{CC_e3}, and \eqref{Eqn:ee1}, we conclude that
\[ 
  \| \mathscr{Y}_{k, k'}^{+} \|_{k \to  k'} \lesssim (L_1 L_2)^{- \alpha / 2} L_1^{\frac{1}{2}} L_2^{\frac{1}{2}} \lesssim
  L^{- \frac{ \alpha -1}{2}}_{\max}, 
\]
which is sufficient for \eqref{Eqn:ktok} since $\al > 1$ and $\dl_0 \ll \al -1$. 

A similar argument as above also gives
\[
  \begin{split}
  \| \mathscr{Y}_{k, k'}^{+} \|_{k k'} &  \lesssim (L_1 L_2)^{- \alpha / 2}  \|
    {\rm T}^{{\rm b},m}_{k k_1 k_2 k'} \|_{k k_1 k_2 k'}  \prod_{j=1}^2 \| h_{k_j k_j'}^{L_j, R_j} \|_{k_j \to  k'_j} \\
    & \les (L_1 L_2)^{- \alpha / 2}  \big(  N^{1-\frac\al2} L_{\min}^{\frac12} (\log L_{\max})^{\frac12} + (L_1 L_2)^{\frac12} \big) \\ 
    & \les (L_1 L_2)^{- \frac{\alpha}{2}}
    N^{\frac{1}{2}} L_{\min}^{\frac12}  L^{\varepsilon}_{\max} \lesssim N^{\frac{1}{2}} L^{- \frac{\alpha}{2} +
    \varepsilon}_{\max} ,
  \end{split} 
\]
where $\eps \ll 1$, $L_{\min} = \min(L_1,L_2)$, and $L_{\max} = \max(L_1,L_2)$,
which gives \eqref{Eqn:eykk}.

Now we consider the pairing case, i.e. $k'_1 = k'_2$. Note that $L_1 = L_2 = L$.
In this case, we have
\begin{align} 
\label{ykk1}
  \begin{split}
    \mathscr{Y}_{k, k'}^{+} & = \sum_{k_1, k_2} {\rm T}^{{\rm b},m}_{k k_1
    k_2 k'} \sum_{k_1'} h_{k_1 k_1'}^{L_1, R_1}
    \overline{h_{k_2 k_1'}^{L_1, R_2}} \frac{|g_{k_1'}|^2}{\jbb{k_1'}^{\alpha}}\\
    & = \sum_{k_1, k_2} {\rm T}^{{\rm b},m}_{k k_1
    k_2 k'} \sum_{k_1'} h_{k_1 k_1'}^{L_1, R_1}
    \overline{h_{k_2 k_1'}^{L_1, R_2}} \frac{1}{\jbb{k_1'}^{\alpha}}  \\ 
    & \hphantom{XXXX} + \sum_{k_1, k_2} {\rm T}^{{\rm b},m}_{k k_1 k_2 k'}
    \sum_{k_1'} h_{k_1 k_1'}^{L_1, R_1}
    \overline{h_{k_2 k_1'}^{L_1, R_2}} \frac{| g_{k_1'} |^2 -
    1}{\jbb{k_1'}^{\alpha}}\\
    & =: \mathscr{Y}_{k, k'}^{(1)} +\mathscr{Y}_{k, k'}^{(2)}.
  \end{split} 
\end{align}

\noi 
We first consider the term $ \mathscr{Y}_{k, k'}^{(1)}$. 
By using Proposition \ref{PROP:contr}, and then Lemma \ref{LEM:tensor3}, we have
\begin{align} 
\label{ykk1_1}
  \begin{split}
    \big \| & \mathscr{Y}_{k, k'}^{(1)} \big \|_{k \to  k'}\les L^{- \alpha} \bigg\| \sum_{k_1, k_2, k_1'}  {\rm T}^{{\rm b},m}_{k k_1
    k_2 k'} h_{k_1 k_1'}^{L_1, R_1}
    \overline{h_{k_2 k_1'}^{L_1, R_2}} \bigg\|_{k \to k'} \\ 
    & \les L^{- \alpha} \min \big( \|
    {\rm T}^{{\rm b},m}_{k k_1 k_2 k'} \|_{k k_1 k_2 \to  k'}, \| {\rm T}^{{\rm b},m}_{k k_1 k_2 k'} \|_{k
    \to k_1 k_2 k' } \big) \Big\| \sum_{k_1'} h_{k_1 k_1'}^{L_1, R_1}
    \overline{h_{k_2 k_1'}^{L_1, R_2}} \Big\|_{k_1k_2} \\  
    & \les L^{- \alpha} \min \big( \|
    {\rm T}^{{\rm b},m}_{k k_1 k_2 k'} \|_{k k_1 k_2 \to  k'}, \| {\rm T}^{{\rm b},m}_{k k_1 k_2 k'} \|_{k' k_1 k_2
    \to  k} \big) \|h_{k_1 k_1'}^{L_1, R_1} \|_{k_1 \to 
    k_1'} \| {h_{k_2 k_1'}^{L_1, R_2}} \|_{k_2 k_1'}\\ 
    & \lesssim  L^{- \alpha} \big( L^{1-\frac\al2 +\eps} + L^{\frac12} \big)  L^{\frac12 + \g_0}   \lesssim  L^{1-\al+\g_0 } ,
  \end{split} 
\end{align}

\noi 
which is sufficient for \eqref{Eqn:ktok} provided $\g_0 \ll \al -1$. 
For the term $\mathscr{Y}_{k, k'}^{(2)} $ in \eqref{ykk1}, 
we apply Proposition
\ref{Prop:te} with $\eta_{k_1'} = | g_{k_1'} |^2 - 1$, then Lemma \ref{LEM:tech}, and Lemma \ref{LEM:tensor3}, to get
\begin{align} 
\label{ykk1_2}
  \begin{split}
    \| & \mathscr{Y}_{k, k'}^{(2)} \|_{k \to  k'}\lesssim \bigg\|
    \sum_{k_1, k_2} \sum_{ k_1' } 
    {\rm T}^{{\rm b},m}_{k k_1 k_2 k'} h_{k_1 k_1'}^{L_1, R_1}
    \overline{h_{k_2 k_1'}^{L_1, R_2}} \frac{\eta_{k_1'}}{\jbb{k_1'}^{\alpha}} \bigg\|_{k \to  k'}\\
    & \lesssim L^{- \alpha}  \bigg\| \sum_{k_1, k_2}
    {\rm T}^{{\rm b},m}_{k k_1 k_2 k'} h_{k_1 k_1'}^{L_1, R_1}
    \overline{h_{k_2 k_1'}^{L_1, R_2}} \bigg\|_{(k k'_1 \to  k') \cap (k
    \to  k' k'_1)}  \\
    & \lesssim L^{- \alpha} \Big( \| {\rm T}^{{\rm b},m}_{k k_1
    k_2 k'} \|_{k k_1 k_2 \to  k'}  +\|
    {\rm T}^{{\rm b},m}_{k k_1 k_2 k'} \|_{k \to  k_1 k_2 k'} \Big)   \prod_{j=1}^2 \| h_{k_j k_j'}^{L_j, R_j} \|_{k_j \to  k'_j} \\
    & \lesssim  L^{\frac12-\al} , 
  \end{split} 
\end{align}
which is sufficient for \eqref{Eqn:ktok}. 

We proceed to consider the estimate \eqref{Eqn:eykk} for the pairing case, where $\mathscr{Y}_{k, k'}^+$ is given by \eqref{ykk1}. 
We begin with $\mathscr{Y}_{k, k'}^{(1)}$ in \eqref{ykk1}. Here we have to use the cancellation that arises from the unitary property of $\wt H^{N,L}_{kk'}$, namely Lemma \ref{LEM:uni}. Specifically, by Corollary \ref{COR:cancel} and Remark \ref{RMK:cancel}, we can redefine $ \mathscr{Y}_{k, k'}^{(1)} $, keeping the same notation for $ \mathscr{Y}_{k, k'}^{(1)} $, as
\[
 \mathscr{Y}_{k, k'}^{(1)} =  \sum_{k_1, k_2} {\rm T}^{{\rm b},m}_{k k_1
    k_2 k'} \sum_{k_1'} h_{k_1 k_1'}^{L, R_1}
    \overline{h_{k_2 k_1'}^{L, R_2}} \bigg( \frac{1}{\jbb{k_1'}^{\alpha}} - \frac{1}{\jbb{k_1}^{\alpha}}  \bigg),
\]

\noi 
since $k_1 \neq k_2$ and $L_1 = L_2 = L$. 
From \eqref{Eqn:2}, we may assume that $| k_1 - k_1' | \lesssim
L^{\eps} R_1$, which implies that
\begin{equation}
  \label{Eqn:pg1} 
  \left| \frac{1}{\jbb{k_1'}^{\alpha}} - \frac{1}{\jbb{k_1}^{\alpha}}
  \right| \lesssim L^{- \alpha - 1 + \eps} R_1 .
\end{equation}

\noi 
Then, a similar argument as in \eqref{ykk1_1} together with \eqref{Eqn:pg1} yields
\begin{align}  
  \begin{split}
    \big \|\mathscr{Y}_{k, k'}^{(1)} \big \|_{k k'} & \les  L^{- \alpha - 1 + \eps} R_1  \bigg\| \sum_{k_1, k_2, k_1'}  {\rm T}^{{\rm b},m}_{k k_1
    k_2 k'} h_{k_1 k_1'}^{L_1, R_1}
    \overline{h_{k_2 k_1'}^{L_1, R_2}} \bigg\|_{k k'} \\  
    & \les L^{- \alpha - 1 + \eps} R_1   \|
    {\rm T}^{{\rm b},m}_{k k_1 k_2 k'} \|_{k k_1 k_2 k'} \|h_{k_1 k_1'}^{L_1, R_1} \|_{k_1 
    k_1'} \| {h_{k_2 k_1'}^{L_1, R_2}} \|_{k_2 \to k_1'}.
  \end{split} 
\end{align}

\noi 
By Lemma \ref{LEM:tensor} and \eqref{Eqn:ee1}, we get
\begin{align} 
\label{ykk2_2}
  \begin{split}
    \big \|\mathscr{Y}_{k, k'}^{(1)} \big \|_{k k'} 
    & \lesssim   L^{- \alpha - 1 + \eps} R_1  \big( N^{1-\frac\al2} L^{\frac12 + \eps} + L\big)  L^{\frac12 + \g_0} R_1^{-\frac12}\\
    & \lesssim  N^{1-\frac\al2 } L^{\frac12 -\al + 2\eps + \g_0} + L^{1-\al + \eps + \g_0},
  \end{split} 
\end{align}

\noi 
which is sufficient for the estimate \eqref{Eqn:eykk}.

Finally, we turn to $\mathscr{Y}_{k, k'}^{(2)}$.
Proceeding as in \eqref{ykk1_2},
we have 
\begin{align}
    \label{ykk2_3}
    \begin{split}
    \| \mathscr{Y}_{k, k'}^{(2)} \|_{k k'} & \lesssim   L^{-
    \alpha}  \bigg\| \sum_{k_1, k_2} {\rm T}^{{\rm b},m}_{k
    k_1 k_2 k'} h_{k_1 k_1'}^{L, R_1} \overline{h_{k_2
    k'_1}^{L, R_2}} \bigg\|_{(k k' \to  k_1')\cap (k
    k' k_1')}  \\
    & \lesssim L^{ - \alpha}\big( \|
    {\rm T}^{{\rm b},m}_{k k_1 k_2 k'} \|_{k k' \to  k_1 k_2}
    \|h_{k_1 k_1'}^{L, R_1} \overline{h_{k_2 k'_1}^{L, R_2}}
    \|_{k_1 k_2 \to  k_1'} \\
    & \hspace{2cm} + \| {\rm T}^{{\rm b},m}_{k k_1
    k_2 k'} \|_{k k_1 k_2 k'} \|h_{k_1 k_1'}^{L, R_1}
    \overline{h_{k_2 k'_1}^{L, R_2}} \|_{k_1' \to  k_1 k_2}
    \big)\\
    & \les L^{ - \alpha}  \| {\rm T}^{{\rm b},m}_{k k_1
    k_2 k'} \|_{k k_1 k_2 k'}   \prod_{j=1}^2 \| h_{k_j k_j'}^{L_j, R_j} \|_{k_j \to  k'_j}  \\
    & \lesssim L^{ - \alpha} ( N^{1-\frac\al2 + \eps} L^{\frac12} + L )   \lesssim N^{1 - \frac\al2 + \eps} L^{\frac{1}{2} - \alpha} + L^{1-\al},
    \end{split}
\end{align}

\noi 
which is again sufficient for \eqref{Eqn:eykk}.

Finally, by collecting \eqref{ykk1}, \eqref{ykk1_1}, \eqref{ykk1_2}, \eqref{ykk2_2}, and \eqref{ykk2_3}, we finish the proof of \eqref{Eqn:ktok} and \eqref{Eqn:eykk}
for Case (a).

\subsubsection{Case (b) and Case (c)}
\label{SUB:DC+}
Type (C, D) and type (D, C).
Without loss of generality, we only consider the case when $(w_{L_1}, w_{L_2})$ are of type (D, C). 
The random matrix \eqref{Eqn:ykk} can be written as
\[ 
  \mathscr{Y}_{k, k'}^{+} = \sum_{k_1, k_2} {\rm T}^{{\rm b},m}_{k k_1 k_2
  k'} (w_{L_1})_{k_1} \cdot \sum_{k_2'}
  \overline{h_{k_2 k_2'}^{L_2, R_2}} \frac{\cj{g_{k_2'}}}{\jbb{k_2'}^{\frac\al2}} , 
\]
where $w_{L_1}$ satisfies \eqref{Eqn:ee3} and $\wt{\mathfrak{h}}^{L_2, R_2} = \{
h_{k_2 k'_2}^{L_2, R_2} \}$ are either identity map over $L_2 / 2 < \jb{k_2} \le  L_2$ or satisfying
\eqref{Eqn:ee1}. 
We apply Proposition \ref{PROP:contr} to get
\begin{align}
\label{DC1}
\begin{split}
\|\mathscr{Y}_{k, k'}^{+} \|_{k \to k'} & \les \min \bigg(  \bigg\| \sum_{ k_2} {\rm T}^{{\rm b},m}_{k k_1 k_2 k'} \sum_{k_2'}
  \overline{h_{k_2 k_2'}^{L_2, R_2}} \frac{\cj{g_{k_2'}}}{\jbb{k_2'}^{\frac\al2}}  \bigg\|_{k\to k_1k'}, \\
  & \hphantom{XXXX,} \bigg\| \sum_{ k_2} {\rm T}^{{\rm b},m}_{k k_1 k_2 k'} \sum_{k_2'}
  \overline{h_{k_2 k_2'}^{L_2, R_2}} \frac{\cj{g_{k_2'}}}{\jbb{k_2'}^{\frac\al2}}  \bigg\|_{kk_1 \to k'} \bigg) \|(w_{L_1})_{k_1}\|_{\l^2_{k_1}}.
\end{split}
\end{align}

\noi 
Recall that the tensor ${\rm T}^{{\rm b},m}_{k k_1 k_2 k'} \overline{h_{k_2 k_2'}^{L_2, R_2}} $ is independent of $F_{L_2}$. We can apply Proposition \ref{Prop:te} to \eqref{DC1}, and then apply Proposition \ref{PROP:contr} to get
\begin{align}
\label{DC2}
\begin{split}
\|\mathscr{Y}_{k, k'}^{+} \|_{k \to k'} & \les L_2^{-\frac\al2} \min\bigg(  \Big\| \sum_{k_2} {\rm T}^{{\rm b},m}_{k k_1 k_2 k'} \overline{h_{k_2 k_2'}^{L_2, R_2}} \Big\|_{(kk_2'\to k_1k') \cap (k\to k_1k_2'k')},  \\
& \hspace{2.2cm} \Big\| \sum_{k_2} {\rm T}^{{\rm b},m}_{k k_1 k_2 k'} \overline{h_{k_2 k_2'}^{L_2, R_2}} \Big\|_{(kk_1k_2'\to k') \cap (k k_1\to k_2'k')} \bigg)  \|(w_{L_1})_{k_1}\|_{\l^2_{k_1}}\\ 
& \les  L_2^{-\frac\al2} \big\| {h_{k_2 k_2'}^{L_2, R_2}} \big\|_{k_2\to k_2'} \min\Big(  \big\| {\rm T}^{{\rm b},m}_{k k_1 k_2 k'} \big\|_{(kk_2\to k_1k') \cap (k\to k_1k_2k')},  \\
& \hspace{2.2cm} \big\| {\rm T}^{{\rm b},m}_{k k_1 k_2 k'} \big\|_{(kk_1k_2\to k') \cap (k k_1\to k_2k')} \Big)  \|(w_{L_1})_{k_1}\|_{\l^2_{k_1}} \\ 
& \les  L_2^{-\frac\al2} \big\| {h_{k_2 k_2'}^{L_2, R_2}} \big\|_{k_2\to k_2'} \big\| {\rm T}^{{\rm b},m}_{k k_1 k_2 k'} \big\|_{(kk_1k_2\to k') \cap (k k_1\to k_2k')}   \|(w_{L_1})_{k_1}\|_{\l^2_{k_1}}.
\end{split}
\end{align}

\noi
By Lemma \ref{LEM:tensor2} and Lemma \ref{LEM:tensor3}, we have
\begin{align*}
    \begin{split}
    \big\|  {\rm T}^{{\rm b},m}_{k k_1 k_2 k'}   & \big\|_{(kk_1k_2\to k') \cap (k k_1\to k_2k')}= \big\| {\rm T}^{{\rm b},m}_{k k_1 k_2 k'} \big\|_{kk_1k_2\to k'} + \big\| {\rm T}^{{\rm b},m}_{k k_1 k_2 k'} \big\|_{k k_1\to k_2k'} \\ 
    & \les L_1^{1-\frac\al2} (\log L_2)^{\frac12} + L_2^{\frac12} + L_1^{1-\frac\al2} L_2^{1-\frac\al2}  \les L_1^{1-\frac\al2} L_2^{\frac12},
    \end{split}
\end{align*}

\noi 
which together with \eqref{DC2}, \eqref{Eqn:ee1}, and \eqref{Eqn:ee3}, implies that 
\begin{align*}
\begin{split}
\|\mathscr{Y}_{k, k'}^{+} \|_{k \to k'} & \les  L_2^{-\frac\al2} \big\| {h_{k_2 k_2'}^{L_2, R_2}} \big\|_{k_2\to k_2'} L_1^{1-\frac\al2} L_2^{\frac12} \|(w_{L_1})_{k_1}\|_{\l^2_{k_1}} \\
& \les L_1^{1-\frac\al2} L_1^{-\frac12 - \g} L_2^{\frac12 - \frac\al2} \les (L_1 L_2)^{\frac12 - \frac\al2} ,
\end{split}
\end{align*}

\noi 
which proves \eqref{Eqn:ktok} since  $\dl_0  \ll \al - 1$.

Now we turn to the estimate \eqref{Eqn:eykk}. Similar computation as in \eqref{DC1} and \eqref{DC2} yields
\begin{align}
\label{DC5}
\begin{split}
\|\mathscr{Y}_{k, k'}^{+} \|_{k k'} & \les \min \bigg( \bigg\| \sum_{ k_2} {\rm T}^{{\rm b},m}_{k k_1 k_2 k'} \sum_{k_2'}
  \overline{h_{k_2 k_2'}^{L_2, R_2}} \frac{\cj{g_{k_2'}}}{\jbb{k_2'}^{\frac\al2}} \bigg\|_{k k_1k' }, \\
  & \hphantom{XXXX} \bigg\| \sum_{ k_2} {\rm T}^{{\rm b},m}_{k k_1 k_2 k'} \sum_{k_2'}
  \overline{h_{k_2 k_2'}^{L_2, R_2}} \frac{\cj{g_{k_2'}}}{\jbb{k_2'}^{\frac\al2}} \bigg\|_{kk' \to k_1} \bigg) \|(w_{L_1})_{k_1}\|_{\l^2_{k_1}} \\ 
& \les L_2^{-\frac\al2} \min\bigg(  \Big\|  \sum_{ k_2} {\rm T}^{{\rm b},m}_{k k_1 k_2 k'} \overline{h_{k_2 k_2'}^{L_2, R_2}} \Big\|_{(kk_1k_2'k') \cap (kk_1k'\to k_2')},  \\
& \hspace{2.2cm} \Big\|  \sum_{ k_2} {\rm T}^{{\rm b},m}_{k k_1 k_2 k'} \overline{h_{k_2 k_2'}^{L_2, R_2}} \Big\|_{(kk_2'k'\to k_1) \cap (k k'\to k_1k_2')} \bigg)  \|(w_{L_1})_{k_1}\|_{\l^2_{k_1}}\\ 
& \les  L_2^{-\frac\al2} \big\| {h_{k_2 k_2'}^{L_2, R_2}} \big\|_{k_2\to k_2'} \min\Big(  \big\| {\rm T}^{{\rm b},m}_{k k_1 k_2 k'} \big\|_{(kk_1k_2k') \cap (kk_1k'\to k_2)},  \\
& \hspace{2.2cm} \big\| {\rm T}^{{\rm b},m}_{k k_1 k_2 k'} \big\|_{(kk_2k'\to k_1) \cap (k k'\to k_1k_2)} \Big)  \|(w_{L_1})_{k_1}\|_{\l^2_{k_1}} \\ 
& \les  L_2^{-\frac\al2} \big\| {h_{k_2 k_2'}^{L_2, R_2}} \big\|_{k_2\to k_2'} \big\| {\rm T}^{{\rm b},m}_{k k_1 k_2 k'} \big\|_{(kk_2k'\to k_1) \cap (k k'\to k_1k_2)}   \|(w_{L_1})_{k_1}\|_{\l^2_{k_1}}.
\end{split}
\end{align}

\noi
By Lemma \ref{LEM:tensor2} and Lemma \ref{LEM:tensor3}, we have
\begin{align}
    \label{DC6}
    \begin{split}
    \big\| {\rm T}^{{\rm b},m}_{k k_1 k_2 k'}   &  \big\|_{(kk_2k'\to k_1) \cap (k k'\to k_1k_2)}= \big\| {\rm T}^{{\rm b},m}_{k k_1 k_2 k'} \big\|_{kk_2k'\to k_1} + \big\| {\rm T}^{{\rm b},m}_{k k_1 k_2 k'} \big\|_{k k'\to k_1k_2} \\ 
    & \les N^{1-\frac\al2} (\log L_2)^{\frac12} + L_2^{\frac12} + L_{\min}^{1-\frac\al2} N^{1-\frac\al2} \\
    & \les N^{1-\frac\al2} L_2^{\frac{\al -1}2} + L_{\min}^{1-\frac\al2} N^{1-\frac\al2}.
    \end{split}
\end{align}

\noi 
By collecting  \eqref{DC2}, \eqref{Eqn:ee1},  \eqref{Eqn:ee3}, and \eqref{DC6}, we conclude that 
\begin{align}
\label{DC7}
\begin{split}
\|\mathscr{Y}_{k, k'}^{+}\|_{k k'} & \les  L_2^{-\frac\al2} \big\| {h_{k_2 k_2'}^{L_2, R_2}} \big\|_{k_2\to k_2'} \big( N^{1-\frac\al2} L_2^{\frac{\al -1}2} + L_{\min}^{1-\frac\al2} N^{1-\frac\al2} \big) \|(w_{L_1})_{k_1}\|_{\l^2_{k_1}} \\
& \les \big( N^{1-\frac\al2} L_2^{\frac{\al -1}2} + L_{\min}^{1-\frac\al2} N^{1-\frac\al2} \big) L_2^{-\frac\al2}  L_1^{-\frac12 - \g}\\ 
&  \les N^{1-\frac\al2} L_2^{-\frac{1}2} L_1^{-\frac12 - \g} + N^{1-\frac\al2} L_{\min}^{1-\frac\al2}  L_2^{-\frac\al2}  L_1^{-\frac12 - \g} \\ 
&  \les N^{\frac12 } L^{{- \frac{\al}{2} }}_{\max},
\end{split}
\end{align}

\noi 
which is sufficient for \eqref{Eqn:eykk}.

Therefore, we finish the proof of \eqref{Eqn:ktok} and \eqref{Eqn:eykk} 
for Case (b): type (C, D). The proof for Case (c): type (D, C) follows similarly. 

\subsubsection{Case (d)} Type: (D, D).
In this case, the matrix \eqref{Eqn:ktok} can be written as
\[ 
  \mathscr{Y}_{k, k'}^{+} = \sum_{k_1, k_2} \mathrm{T}_{k k_1 k_2 k'}^{\text{b},
  m}  \cdot (w_{L_1})_{k_1} \cdot  \overline{(w_{L_2})_{k_2}}, 
\]
where $w_{L_1}$ and $w_{L_2}$ satisfy \eqref{Eqn:ee3}. 
By using
Proposition \ref{PROP:contr}, followed by Lemma
\ref{LEM:tensor2} and \eqref{Eqn:ee3}, we obtain
\[ 
  \begin{split}
    \| \mathscr{Y}_{k, k'}^{+} \|_{k \to  k'} & \lesssim \|
    {\rm T}^{{\rm b},m}_{k k_1 k_2 k'} \|_{k k_1
    \to  k_2 k'} \| (w_{L_1})_{k_1} \|_{k_1} \| \overline{(
    w_{L_2})_{k_2}}  \|_{k_2}\\
    &  \lesssim (L_1 L_2)^{1 - \frac{\alpha}{2}} (L_1 L_2)^{- \frac{1}{2}
    - \gamma} \\  
    & \lesssim L^{\frac{1 - \alpha}{2} - \gamma} \les L^{- 4 \dl_0 },
  \end{split} 
\]
which is sufficient for \eqref{Eqn:ktok}, provided $\dl_0  \ll \al - 1$. 

Similarly, we have
\[ 
  \begin{split}
    \| \mathscr{Y}_{k, k'}^{+} \|_{k k'} & \lesssim \| \mathrm{T}_{k k_1 k_2
    k'}^{\text{b}, m} \|_{k k' \to  k_1 k_2} \| (w_{L_1})_{k_1}
    \|_{k_1} \| \overline{(
    w_{L_2})_{k_2}} \|_{k_2} \\
    & \lesssim N^{1 - \frac{\alpha}{2}} L_{\min}^{1 -
    \frac{\alpha}{2}} (L_1 L_2)^{- \frac{1}{2} - \gamma} \\ 
    & \lesssim N^{1-\frac{\al}{2}} L_{\max}^{- \frac{1}{2} - \gamma} \les N^{\frac12} L_{\max}^{-\frac\al2 -\g} ,
  \end{split} 
\]
which again proves \eqref{Eqn:eykk}.

Therefore,we have finished the proof of \eqref{Eqn:ktok} and \eqref{Eqn:ykk}
for Case (d).

\subsection{Low-high-low random averaging operator}
\label{SUB:LHL}
In this subsection, we consider the low-low-high random matrix $\mathscr{Y}_{k, k'}^{-} $ defined in \eqref{Eqn:ykk-}, and prove \eqref{Eqn:ktok-} and \eqref{Eqn:eykk-}.
We follow a similar strategy as in the previous subsection, and consider four cases for \eqref{Eqn:ykk-} depending on whether $(w_{L_1}, w_{L_3})$ are of types (a) (C, C), (b) (C, D), (c) (D, C), (d) (D, D).

\subsubsection{Case (a)}  
\label{Subsub:cc-}
Type (C, C). 
In this case, the matrix \eqref{Eqn:ykk-} can be written as
\[ 
  \mathscr{Y}_{k, k'}^{-} = \sum_{k_1, k_3} {\rm T}^{{\rm b},m}_{k k_1 k'
  k_3} \sum_{k_1', k_3'} h_{k_1 k_1'}^{L_1, R_1} 
  {h_{k_3 k_3'}^{L_3, R_3}} \frac{g_{k_1'}{g_{k_3'}}}{\jbb{k_1'}^{\frac\al2} \jbb{k_3'}^{\frac\al2}} 
\]
where $\mathfrak{h}^{L_j, R_j} = \{ h_{k_j k'_j}^{L_j, R_j} \}$ ($j \in \{ 1,
3 \}$) are either identity map over $L_j / 2 < \jb{k_j'} \le  L_j$ or satisfying \eqref{Eqn:ee1}.
Similar argument as in Subsection \ref{Subsub:cc} yields
  \begin{align} 
  \label{P_CC-}
  \begin{split}
    & \| \mathscr{Y}_{k, k'}^{-} \|_{k \to  k'}   \lesssim (L_1 L_3)^{- \alpha / 2} \big ( \| \mathrm{T}_{k k_1
    k' k_3}^{\text{b}, m} \|_{k k_1 \to  k' k_3} + \|
    \mathrm{T}_{k k_1 k' k_3}^{\text{b}, m} \|_{k \to  k_1 k' k_3} \\
    & \hspace{2.5cm} + \| \mathrm{T}_{k k_1 
    k' k_3}^{\text{b}, m} \|_{k k_1 k_3 \to  k'} + \|
    \mathrm{T}_{k k_1 k' k_3}^{\text{b}, m} \|_{k k_3 \to  k_1 k'} \big )\\
    & \hspace{1.3cm} \times \| h_{k_1 k_1'}^{L_1, R_1} \|_{k_1 \to 
    k'_1} \| {h_{k_3 k_3'}^{L_3, R_3}} \|_{k_3
    \to  k'_3} .
  \end{split} 
\end{align}

\noi 
By using Lemma \ref{LEM:tensor2}, Lemma \ref{LEM:tensor3}, with $N_1 = L_1$, $N_2 \sim N$, and $N_3 = L_3$, we obtain
\begin{align}
    \label{CC_e3-}
    \begin{split}
     \| & \mathrm{T}_{k k_1
    k' k_3}^{\text{b}, m} \|_{k k_1 \to  k' k_3} + \|
    \mathrm{T}_{k k_1 k' k_3}^{\text{b}, m} \|_{k \to  k_1 k' k_3} + \| \mathrm{T}_{k k_1 k'
    k_3}^{\text{b}, m} \|_{k k_1 k_3 \to  k'} + \|
    \mathrm{T}_{k k_1 k' k_3}^{\text{b}, m} \|_{kk_3 \to  k_1 k'} \\
    & \les (L_1L_3)^{1 - \frac\al2}  + L_{\max}^{1 - \frac{\alpha}2 } (\log L_{\min})^{\frac12} + L_{\min}^{\frac12}  \\ 
    & \les (L_1L_3)^{\frac12}, 
    \end{split}
\end{align}

\noi 
where $L_{\min} = \min \{L_1,L_3\}$ and $L_{\max} = \max \{L_1,L_3\}$.
Thus it follows from \eqref{P_CC-}, \eqref{CC_e3-}, and \eqref{Eqn:ee1}, that
\[ 
  \| \mathscr{Y}_{k, k'}^{-} \|_{k \to  k'} \lesssim (L_1 L_3)^{- \alpha / 2} L_1^{\frac{1}{2}} L_3^{\frac{1}{2}} \lesssim
  L^{\frac{1 - \alpha}{2}}_{\max}, 
\]
which is sufficient for \eqref{Eqn:ktok-}. 
A similar argument also gives
\begin{align} 
\label{CC_e4-}
  \begin{split}
   \| \mathscr{Y}_{k, k'}^{-} \|_{k k'}   & \lesssim (L_1 L_3)^{- \alpha / 2} \big( \|{\rm T}^{{\rm b},m}_{k k_1 k' k_3}
    \|_{k k' k_3 \to  k_1} + \| \mathrm{T}^{\text{b},
    m}_{k k_1 k' k_3} \|_{k k_1 k' \to  k_3}\\ 
    & \hspace{2em} + \|
    {\rm T}^{{\rm b},m}_{k k_1 k'  k_3} \|_{k k_1 k' k_3} +
    \| {\rm T}^{{\rm b},m}_{k k_1 k' k_3} \|_{k k' \to  k_1 k_3} \big)\\
    & \hspace{1em} \times \| h_{k_1 k_1'}^{L_1, R_1} \|_{k_1 \to  k'_1} \|
    h_{k_3 k_3'}^{L_3, R_3} \|_{k_3 \to k'_3} .
  \end{split} 
\end{align}

\noi 
By using Lemma \ref{LEM:tensor2}, Lemma \ref{LEM:tensor3}, with $N_1 = L_1$, $N_2 = N$, and $N_3 = L_3$, we obtain
\begin{align}
\label{CC_e5-}
\begin{split}
    \| & {\rm T}^{{\rm b},m}_{k k_1 k' k_3}
    \|_{k k' k_3 \to  k_1} + \| \mathrm{T}^{\text{b},
    m}_{k k_1 k' k_3} \|_{k k_1 k' \to  k_3}  + \|
    {\rm T}^{{\rm b},m}_{k k_1 k' k_3} \|_{k k_1 k' k_3} + \| {\rm T}^{{\rm b},m}_{k k_1 k' k_3} \|_{k k' \to  k_1 k_3}  \\ 
    & \les L_3^{1-\frac\al2} (\log N)^{\frac12} + N^{\frac12} + L_1^{1-\frac\al2} (\log N)^{\frac12} + L_{\max}^{1-\frac\al2} (L_{\min} \log N)^{\frac12} + (L_{\min} N)^{\frac12} \\ 
    & \les N^{\frac 12} L_{\min}^{\frac12}  .
\end{split}
\end{align}

\noi 
By collecting \eqref{CC_e4-} and \eqref{CC_e5-}, we conclude that
\begin{align*}
    \begin{split}
    \| \mathscr{Y}_{k, k'}^{-} \|_{k k'}   & \les (L_1 L_3)^{- \alpha / 2}  N^{\frac 12} L_{\min}^{\frac12}   \les  N^{\frac 12} L_{\max}^{-\frac\al2},
    \end{split}
\end{align*}

\noi 
from which by choosing $\g_0 \ll \al -1$ we finish the proof of \eqref{Eqn:eykk-}.

\subsubsection{Case (b) and Case (c)}
Type (C, D) and type (D, C).
Without loss of generality, we only consider $(w_{L_1}, w_{L_3})$ of type (D, C). Thus, the random matrix \eqref{Eqn:ykk-} can be written as
\[ 
  \mathscr{Y}_{k, k'}^{-} = \sum_{k_1, k_3} {\rm T}^{{\rm b},m}_{k k_1 k' k_3} (w_{L_1})_{k_1} \cdot \sum_{k_3'} {h_{k_3 k_3'}^{L_3, R_3}} \frac{{g_{k_3'}}}{\jbb{k_3'}^{\frac\al2}} , 
\]
where $w_{L_1}$ satisfies \eqref{Eqn:ee3} and $\mathfrak{h}^{L_3, R_3} = \{
h_{k_3 k'_3}^{L_3, R_3} \}$ are either identity map over $L_3 / 2 < \jb{k_3'} \le  L_3$ or satisfying
\eqref{Eqn:ee1}. 
Similar argument as in Subsection \ref{SUB:DC+} yields 
\begin{align*}
\begin{split}
\|\mathscr{Y}_{k, k'}^{-} \|_{k \to k'} 
& \les  L_3^{-\frac\al2} \big\| {h_{k_3 k_3'}^{L_3, R_3}} \big\|_{k_3\to k_3'} \big\| {\rm T}^{{\rm b},m}_{k k_1 k' k_3} \big\|_{(kk_1k_3\to k') \cap (k k_1\to k_3k')}   \|(w_{L_1})_{k_1}\|_{\l^2_{k_1}} \\
& \les L_1^{-\frac12 -\g} 
L_3^{-\frac\al2}  \big(L_{\max}^{1-\frac\al2} (\log L_{\min})^{\frac12} + L_{\min}^{\frac12} +  (L_1L_3)^{1-\frac\al2} \big)\\
& \les L_1^{-\frac12 -\g} 
L_3^{-\frac\al2}  (L_1L_3)^{1-\frac\al2}  \les L_{\max}^{\frac12 - \frac\al 2 - \g},
\end{split}
\end{align*}

\noi 
which is again sufficient for \eqref{Eqn:ktok-} provided $\dl_0  \ll \al - 1$.

Similar computation as in \eqref{DC5}, \eqref{DC6}, and \eqref{DC7} yields
\begin{align*} 
\begin{split}
\|\mathscr{Y}_{k, k'}^{-} \|_{k k'} 
& \les  L_3^{-\frac\al2} \big\| {h_{k_3 k_3'}^{L_3, R_3}} \big\|_{k_3\to k_3'} \big\| {\rm T}^{{\rm b},m}_{k k_1 k' k_3} \big\|_{(kk'k_3 \to k_1) \cap (k k'\to k_1k_3)}   \|(w_{L_1})_{k_1}\|_{\l^2_{k_1}} \\
& \les L_1^{-\frac12 -\g} L_3^{-\frac\al2}  \big( L_3^{1-\frac\al2} (\log N)^{\frac12} + N^{\frac12} + L_{\min}^{\frac12 - \frac\al4} N^{\frac12 - \frac\al4} \big)  \\
& \les N^{\eps} L_1^{-\g -\frac12} L_3^{1-\al} + N^{\frac12} L_1^{-\frac12 -\g} L_3^{-\frac\al2}  + N^{\frac12 - \frac\al4} L_1^{-\g -\frac\al4} L_3^{\frac12-\frac{3\al}4} \\
& \les N^{\frac12} L_{\max}^{-\frac12 -\g},
\end{split}
\end{align*}

\noi 
which is sufficient for \eqref{Eqn:eykk-} as $\g_0 \ll \g$. Here we used that $\al \in (1,2)$.

Therefore, we finish the proof of \eqref{Eqn:ktok-} and \eqref{Eqn:eykk-} 
for Case (b): type (C, D). The proof for Case (c): type (D, C) follows similarly. 

\subsubsection{Case (d)} Type: (D, D).
In this case, the matrix \eqref{Eqn:ykk-} can be written as
\[ 
  \mathscr{Y}_{k, k'}^{-} = \sum_{k_1, k_3} \mathrm{T}_{k k_1 k' k_3}^{\text{b},
  m}  \cdot (w_{L_1})_{k_1} \cdot \left( {w_{L_3}} \right)_{k_3}, 
\]
where $w_{L_1}$ and $w_{L_3}$ satisfy \eqref{Eqn:ee3}. 
By using
Proposition \ref{PROP:contr}, followed by Lemma
\ref{LEM:tensor2} and \eqref{Eqn:ee3}, we obtain
\[ 
  \begin{split}
    \| \mathscr{Y}_{k, k'}^{-}  & \|_{k \to  k'}\lesssim \|
    {\rm T}^{{\rm b},m}_{k k_1 k' k_3} \|_{k k_1
    \to  k' k_3} \| (w_{L_1})_{k_1} \|_{k_1} \| (
    {w_{L_3}} )_{k_3} \|_{k_3}\\
    &  \lesssim (L_1 L_3)^{1 - \frac{\alpha}{2}} (L_1 L_3)^{- \frac{1}{2}
    - \gamma}  \lesssim L_{\max}^{\frac{1 - \alpha}{2} - \gamma},
  \end{split} 
\]
which is sufficient for \eqref{Eqn:ktok}. 

Similarly, we have
\[ 
  \begin{split}
    \| \mathscr{Y}_{k, k'}^{-} \|_{k k'} & \lesssim \| \mathrm{T}_{k k_1 k' k_3}^{\text{b}, m} \|_{k k' \to  k_1 k_3} \| (w_{L_1})_{k_1} \|_{k_1} 
    \| ( {w_{L_3}} )_{k_3} \|_{k_3} \\
    & \lesssim (N  L_{\min})^{\frac12 -
    \frac{\alpha}{4}} (L_1 L_2)^{- \frac{1}{2} - \gamma} \\ 
    & \lesssim N^{\frac12 -
    \frac{\alpha}{4}} L_{\max}^{- \frac{1}{2} - \gamma} \les N^{\frac12} L_{\max}^{- \frac{1}{2} -\frac\al4 - \gamma} ,
  \end{split} 
\]
where we used that $N \ge L_{\max}$, which proves \eqref{Eqn:eykk-}.

Therefore,we have finished the proof of \eqref{Eqn:ktok-} and \eqref{Eqn:ykk-}
for Case (d): type (D, D).

\subsection{Random averaging operator}
\label{SUB:RAO}
The main purpose of this subsection is to 
prove \eqref{Eqn:1} and \eqref{Eqn:2} in $\text{{\textsf{Local}}} (M)$ by
assuming $\text{{\textsf{Local}}} \left( \frac{M}{2} \right)$.  
Let 
\begin{equation}
  \label{Eqn:smallpnl}
  \mathfrak{p}^{N, L} = \mathcal{P}^{N, L} - \mathcal{P}^{N, L / 2},
\end{equation}
where $\mathcal{P}^{N, L}$ be the
linear operator defined in \eqref{Eqn:l}.
Recall from
\eqref{Eqn:calH} and \eqref{Eqn:smallpnl} that
\begin{equation}
  \label{Eqn:hpnl} 
  \mathcal{H}^{N, L} = \Pi_N (1 - \mathcal{P}^{N, L})^{- 1} = \Pi_N +
  \sum_{n = 1}^{\infty} (\mathcal{P}^{N, L})^n,
\end{equation}
and thus the operator $\mathfrak{h}^{N,L}$ in \eqref{Eqn:hNL} has the following expansion,
\begin{equation}
  \label{Eqn:hnlse} 
  \mathfrak{h}^{N, L} = \sum_{n = 1}^{\infty} (- 1)^{n - 1}
  (\mathcal{H}^{N, L} \mathfrak{p}^{N, L})^n \mathcal{H}^{N, L}.
\end{equation} 
Recalling the operator $\mathcal{P}^+ := \mathcal{P}^+_{N, L_1, L_2}$ from \eqref{Eqn:5}, we have
\begin{equation}
  \label{Eqn:ind} 
  \begin{split}
    \mathcal{P}^{N, L} [\psi] & = - 2i  \chi(t) \int_0^t \Pi_N \mathcal{M} (u_L,
    u_L, \psi) (t') d t' \\
    & = - 2i  \sum_{L_1, L_2 \le  L} \chi(t) \int_0^t \Pi_N \mathcal{M}
    (y_{L_1}, y_{L_2}, \psi) (t') d t' \\ 
    & = \mathcal{P}^{N, \frac{L}{2}}
    \psi + 2 \sum_{L_{\max} = L} \mathcal{P}_{N, L_1, L_2}^+ \psi ,
  \end{split}
\end{equation}

\noi 
where $\chi(t)$ is the time localization given in \eqref{chi}.
Recall that $\mathcal{P}_{N, L_1, L_2}^+$ satisfies the estimates
\eqref{Eqn:6}-\eqref{Eqn:7}, 
which implies that $\mathfrak{p}^{N, L}$ also satisfies the estimates \eqref{Eqn:6} and \eqref{Eqn:7} but with a logarithm loss coming from the summation over $L_i$. 
This logarithm loss is harmless since it can be absorbed. 
Namely, from \eqref{Eqn:smallpnl}, \eqref{Eqn:ind}, and Corollary \ref{COR:RAO}, we have
\begin{equation}
  \label{Eqn:pnlbounds} 
  \begin{cases}
    \|\mathfrak {p}^{N, L} \|_{X^b \to  X^{b}} \lesssim T^{c\theta}
    L^{- 2 \delta_0 } (\log L)^2 ;\\
    \|\mathfrak {p}^{N, L} \|_{Z^{b, b}} \lesssim T^{c\theta} N^{\frac{1}{2}
    + \frac{2 \gamma_0}3 } L^{- \frac{1}{2} - \g_0 } (\log L)^2,
  \end{cases}
\end{equation}
which together with \eqref{Eqn:smallpnl} implies
\begin{equation}
  \label{Eqn:bigpnlop} 
  \begin{cases}
  \| \mathcal{P}^{N, L} \|_{X^{b} \to 
  X^{b}} \lesssim T^{c \theta} ,\\
  \|\mathcal {P}^{N, L} \|_{Z^{b, b}} \lesssim T^{c\theta} N^{\frac{1}{2}
    + \frac{2\gamma_0}3}
  \end{cases}
\end{equation}

\noi 
for some $c > 0$.

From the definition of the operator norm, we note that
\begin{equation}
  \label{Eqn:ABxaxa} 
  \| \mathcal{A} \mathcal{B} \|_{Y^{b,b}} =  \| \mathcal{A} \mathcal{B} \|_{X^{b} \to 
  X^{b}} \le  \| \mathcal{A} \|_{X^{b} \to  X^{b}}
  \| \mathcal{B} \|_{X^{b} \to  X^{b}}
\end{equation}
which together with \eqref{Eqn:hpnl} and \eqref{Eqn:bigpnlop} implies that
\begin{equation}
  \label{Eqn:hnl11} 
  \begin{split}
    \|\mathcal{H}^{N, L} - \Pi_N \|_{X^{b} \to  X^{b}}
     &  \le  \sum_{n = 1}^{\infty} \| (\mathcal{P}^{N, L})^n \|_{X^{b}
    \to  X^{b}}  \le  \sum_{n = 1}^{\infty} \| \mathcal{P}^{N, L} \|^n_{X^{b}
    \to  X^{b}}\\
    & \le  \| \mathcal{P}^{N, L} \|_{X^{b} \to  X^{b}}
    \sum_{n \ge  1} (C T^{c \theta})^{n - 1}  \lesssim T^{c \theta}.
  \end{split}
\end{equation}

\noi 
Also, note that $\| \Pi_N \|_{X^b (J) \to X^b (J)}  \les 1$.
In particular, we conclude from \eqref{Eqn:hnl11} that
\[
\| \mathcal{H}^{N, L} \|_{X^{b} \to  X^{b}} \les 1,
\]

\noi 
which together with \eqref{Eqn:hnlse}, \eqref{Eqn:ABxaxa}, and \eqref{Eqn:pnlbounds}, implies 
\[ 
  \begin{split}
    \| \mathfrak{h}^{N, L}  &\|_{X^{b} \to  X^{b}} \le 
    \sum_{n = 1}^{\infty} \| (\mathcal{H}^{N, L} \mathfrak{p}^{N, L})^n
    \mathcal{H}^{N, L} \|_{X^{b} \to  X^{b}} \le  \sum_{n = 1}^{\infty} \| \mathcal{H}^{N, L} \|_{X^{b}
    \to  X^{b}}^{n + 1} \|\mathfrak{p}^{N, L} \|^n_{X^{b}
    \to  X^{b}}\\
    & \le  \|\mathfrak{p}^{N, L} \|_{X^{b} \to  X^{b}}
    \sum_{n = 1}^{\infty} C^{n + 1} (T^{c \theta} L^{-
    \delta_0})^{n - 1} \les T^{c\theta}
    L^{-   \delta_0 } ,
  \end{split} 
\]
which gives the first bound in \eqref{Eqn:1} by choosing $T\ll 1$.

We also note that 
\begin{equation*}
  \| \mathcal{A} \mathcal{B} \|_{Z^{b,b}} \le \min\big( \|
  \mathcal{A} \|_{X^{b} \to  X^{b}} \| \mathcal{B}
  \|_{Z^{b,b}}, \| \mathcal{A}
  \|_{Z^{b,b}} \|
  \mathcal{B} \|_{X^{b} \to  X^{b}} \big),
\end{equation*}
which together with \eqref{Eqn:hnlse} and \eqref{Eqn:bigpnlop} implies
\[ 
  \begin{split}
    \| \mathfrak{h}^{N, L} \|_{Z^{b,b}} & \le  \sum_{n = 1}^{\infty}
    \| (\mathcal{H}^{N, L} \mathfrak{p}^{N, L})^n \mathcal{H}^{N, L}
    \|_{Z^{b,b}}\\
    & \le  \sum_{n = 1}^{\infty} \| \mathcal{H}^{N, L} \|^{n+1}_{X^{b}
    \to  X^{b}} \| \mathfrak{p}^{N, L} \|^{n-1}_{X^{b}
    \to  X^{b}} \| \mathfrak{p}^{N, L} \|_{Z^{b,b}}\\
    & \le  \| \mathfrak{p}^{N, L} \|_{Z^{b,b}} \sum_{n = 1}^{\infty}
    C^{n+1} (T^{c\theta} L^{- \delta_0 })^{n-1}\\
    & \le  T^{c\theta} N^{\frac{1}{2} + \gamma_0 } L^{-\frac12},
  \end{split} 
\]
provided $T \ll 1$, 
which completes the proof of the second bound of \eqref{Eqn:1}.

Finally, from Proposition \ref{Prop:co}, we have
\[ 
  \bigg\| \bigg( 1 + \frac{| k - k' |}{L} \bigg)^{\kappa}
  (\widehat{\mathcal{A} \mathcal{B}})_{k k'} \bigg\|_{Z^{b,b}} \le 
  \| \mathcal{A} \|_{X^{b} \to  X^{b}} \bigg\| \bigg( 1 +
  \frac{| k - k' |}{L} \bigg)^{\kappa} \widehat{\mathcal{B}}_{k k'}
  \bigg\|_{Z^{b,b}}, 
\]
which, together with a similar argument as above, gives
\begin{equation*} 
  \bigg\| \bigg( 1 + \frac{| k - k' |}{L} \bigg)^{\kappa} h^{N, L}_{k k'} \bigg\|_{Z^{b,b}} \le 
  T^{c\theta} N,
\end{equation*}
provided $T \ll 1$, which proves \eqref{Eqn:2}.

\section{The remainder terms}\label{Sec:rems}
In this section, we continue the proof of Proposition \ref{PROP:main}.
We have shown 
\eqref{Eqn:1},  \eqref{Eqn:2}, 
\eqref{Eqn:6}, 
\eqref{Eqn:7}, and \eqref{Eqn:8} of $\text{{\textsf{Local}}} \left( M
\right)$ in Section \ref{Sec:raos}. 
It remains to prove \eqref{Eqn:3}, which will be the main focus of this section.
To be more precise, we shall prove the bound \eqref{Eqn:3} with $N = M$,
i.e.
\begin{align} 
\label{Eqn:3m}
  \| z_M \|_{X^b (J)} \le  M^{- \frac{1}{2}
  - \gamma} ,
\end{align}
by assuming $\text{{\textsf{Local}}} \left( \frac{M}{2} \right)$ from Definition \ref{Def:LocalM}. 
The main strategy of this section is to show \eqref{Eqn:3m} by a continuity argument, with the help of \eqref{Eqn:1},  \eqref{Eqn:2}, 
\eqref{Eqn:6}, 
\eqref{Eqn:7}, and \eqref{Eqn:8} of $\text{{\textsf{Local}}} \left( M
\right)$.
Recall that $z_M$ is given by the equation \eqref{Eqn:newzN}, whose right-hand side can be written as the combination of the expressions 
\begin{equation}
  \label{termA} 
  \mathscr{A} (w_{N_1}, w_{N_2}, w_{N_3}) (t) =
  { {\chi}} (t) \int^t_0 \Pi_N \mathcal{M} (w_{N_1}, w_{N_2},
  w_{N_3}) (s) d s,
\end{equation}
and
\begin{align}
\label{termB}
  \mathscr{B} (w_{N_1}, w_{N_2}, w_{N_3}) (t) = { {\chi}} (t)
  \int^t_0 \Delta_N \mathcal{M} (w_{N_1}, w_{N_2}, w_{N_3}) (s) d s,
\end{align}

\noi 
\noi 
where $w_{N_i}$ are of either type $(C)$ or type $(D)$ defined in Subsection \ref{Sub:refo}.
In particular, we can group the right-hand side of \eqref{Eqn:newzN} as follows
\begin{itemize}
  \item[\textup{(I)}] $\mathscr{A} (w_{N_1}, w_{N_2}, w_{N_3})$ with $w_{N_j}$ for $j\in \{1,2,3\}$ of any type
  and $N_{\med} = N_{\max} = M$.
  
  \item[\textup{(II)}] $\mathscr{A} (w_{N_1}, w_{N_2}, z_M)$ with $w_{N_j}$ for $j\in \{1,2\}$ of any type
  and $N_j \le  \frac{M}{2}$ for $j \in \{ 1, 2 \}$.
  
  \item[\textup{(III)}] $\mathscr{A} (w_{N_1}, w_{N_2}, \psi_{M, L_M})$ with $w_{N_j}$ for $j
  \in \{ 1, 2 \}$ of any type with $L_M < \max (N_1, N_2) \le 
  \frac{M}{2}$.
  
  \item[\textup{(IV)}] $\mathscr{A} (w_{N_1}, z_M, w_{N_3})$ with $w_{N_j}$ for $j\in \{1,3\}$ of any type
  and $N_j \le  \frac{M}{2}$ for $j \in \{ 1, 3 \}$.
  
  \item[\textup{(V)}] $\mathscr{A} (w_{N_1}, \psi_{M, L_M}, w_{N_3})$ where $w_{N_j}$ for $j
  \in \{ 1, 3 \}$ of any type with $\max (N_1, N_3) \le 
  \frac{M}{2}$.
  
  \item[\textup{(VI)}] $\mathscr{B} (w_{N_1}, w_{N_2}, w_{N_3})$ with $w_{N_j}$ for $j\in \{1,2,3\}$ of any type
  and $N_{\max} \le  \frac{M}{2}$.
  
  \item[\textup{(VII)}] The term
  \[ \int_0^t \mathcal{R} (w_N, w_N, w_N) (t') d t' \]
  with $w_N$ of any type and { {$N \in \{ M, M / 2 \}$}}.
\end{itemize}

Our goal in this section is to recover the bound for $z_M$ in \eqref{Eqn:newzN} for each of
the terms (I)--(VII) above, i.e.
\begin{align}
  \| \mathscr{A} (w_{N_1}, w_{N_2}, w_{N_3}) \|_{X^b}
  \lesssim M^{- \frac{1}{2} - \gamma}, \label{Eqn:Awww1} \\
\| \mathscr{B} (w_{N_1}, w_{N_2}, w_{N_3}) \|_{X^b}
  \lesssim M^{- \frac{1}{2} - \gamma}, \label{Eqn:Bwww1}
\end{align}

\noi 
for all possible types of $(w_{N_1}, w_{N_2}, w_{N_3})$ as above.

\subsection{Preparation of the proof}\label{Sub:pre}

We first see that
\[ 
  \| \mathscr{A} (w_{N_1}, w_{N_2}, w_{N_3}) \|_{X^1} \lesssim \| \mathcal{M}
  (w_{N_1}, w_{N_2}, w_{N_3}) \|_{L^2_{t, x}} \lesssim {{M^2}}
  \prod_{i = 1}^3 \| w_{N_i} \|_{X^b} \lesssim M^{14} .
\]

\noi 
Therefore, to prove \eqref{Eqn:Awww1}, from  interpolation, it suffices to show
\begin{equation}
  \label{Eqn:Awww} 
  \| \mathscr{A} (w_{N_1}, w_{N_2}, w_{N_3}) \|_{X^{1 - b}}
  \lesssim M^{- \frac{1}{2} {{- 2\g  }}},
\end{equation}
provided $b - \frac{1}{2}  \ll \gamma_0$. By \eqref{Eqn:M},
\eqref{Eqn:Ik}, and \eqref{termA}, we have
\begin{align} 
\label{Eqn:FAwww}
  \begin{split}
    & \mathcal{F}_{t, x} (\mathscr{A} (w_{N_1}, w_{N_2}, w_{N_3})) (\tau, k)\\
    & \quad = { {- i }} \sum_{\substack{
      k_1 - k_2 + k_3 = k\\
      k_2 \not\in \{ k_1, k_3 \}
    }} \int_{\mathbb{R}^3} \mathcal{K} (\tau, \Phi + \tau_1 -
    \tau_2 + \tau_3)\\
    & \hspace{1.5cm} \times ( \widehat{w_{N_1}} )_{k_1} (\tau_1)
    \cdot \overline{( \widehat{w_{N_2}} )_{k_2}} (\tau_2) \cdot
    ( \widehat{w_{N_3}} )_{k_3} (\tau_3) d \tau_1
    d \tau_2 d \tau_3\\
    & = { {- i }} \sum_m \int_{\mathbb{R}^3} \sum_{k_1, k_2,
    k_3} \mathcal{K} (\tau, m + \Phi - [\Phi] + \tau_1 - \tau_2 + \tau_3)
    \times \mathrm{T}_{k k_1 k_2 k_3}^{\text{b}, m}\\
    & \hspace{1.5cm} \times ( \widehat{w_{N_1}} )_{k_1} (\tau_1)
    \cdot \overline{( \widehat{w_{N_2}} )_{k_2}} (\tau_2) \cdot
    \left( \widehat{w_{N_3}} \right)_{k_3} (\tau_3) d \tau_1
    d \tau_2 d \tau_3,
  \end{split} 
\end{align}
where $\Phi$ is defined in \eqref{Eqn:Phi}, $[\Phi]$ is the integer part of $\Phi$, and the tensor $\mathrm{T}_{k k_1
k_2 k_3}^{\text{b}, m}$ is the base tensor given in \eqref{baseT}. 
We note that $\Phi - [\Phi] = O(1)$. 
Then from \eqref{Eqn:FAwww}, \eqref{Eqn:bk}, and Minkowski inequality, we obtain
\[ 
  \begin{split}
    & \| \mathscr{A} (w_{N_1}, w_{N_2}, w_{N_3}) \|_{X^{1 - b}}^2\\
    & \quad \lesssim \int_{\mathbb{R}} \langle \tau \rangle^{- 2 b} \bigg(
    \sum_m \int_{\mathbb{R}^3}  \bigg\| \sum_{k_1, k_2, k_3}  \langle \tau - m +
    O(1) - \tau_1 + \tau_2 - \tau_3 \rangle^{- 1}\\
    & \hspace{2cm} \times  \langle \tau_1 \rangle^{- b} \langle \tau_2
    \rangle^{- b} \langle \tau_3 \rangle^{- b} 
    \mathrm{T}_{k k_1
    k_2 k_3}^{\text{b}, m} \prod_{i=1}^3 \jb{\tau_i}^b ( \widehat{w_{N_i}} )^{\zeta_i}_{k_i} (\tau_i)  
    \bigg\|_{\ell_k^2} d \tau_1 d \tau_2 d \tau_3\bigg)^2
    d \tau\\
    & \quad \lesssim \int_{\mathbb{R}} \langle \tau \rangle^{- 2 b} \bigg(
    \sum_{m } \int_{\mathbb{R}^3} \langle \tau - m - \tau_1 +
    \tau_2 - \tau_3 \rangle^{- 1} \langle \tau_1 \rangle^{- b} \langle \tau_2
    \rangle^{- b} \langle \tau_3 \rangle^{- b} \\
    & \hspace{2cm} \times \bigg\| \sum_{k_1, k_2, k_3} \mathrm{T}_{k k_1
    k_2 k_3}^{\text{b}, m} \prod_{i=1}^3 \jb{\tau_i}^b ( \widehat{w_{N_i}} )^{\zeta_i}_{k_i} (\tau_i) \bigg\|_{\ell_k^2}d \tau_1 d \tau_2 d \tau_3\bigg)^2
    d \tau,
  \end{split} 
\]

\noi 
where $\zeta_1, \zeta_3 = +$ and $\zeta_2 = -$.
Due to the fact that $|{k_1}|^{\alpha} - |{k_2}|^{\alpha} +
|{k_3}|^{\alpha} - |{k}|^{\alpha} - O(1) = m$, we see that $| m |
\lesssim N^{\al}$, which leads to 
\[ 
\sum_{m \in \mathbb{Z}} \langle \tau - m -
\tau_1 + \tau_2 - \tau_3 \rangle^{- 1} \lesssim 1 + \log N. 
\]
Therefore, to
prove \eqref{Eqn:Awww} it only needs to show
\begin{align}
\label{Eqn:ex1}
\begin{split}
\bigg\| \sum_{k_1, k_2, k_3} \mathrm{T}_{k k_1
k_2 k_3}^{\text{b}, m} \prod_{i=1}^3 \jb{\tau_i}^b ( \widehat{w_{N_i}} )^{\zeta_i}_{k_i} (\tau_i)\bigg\|_{L^2_{\tau_1 \tau_2 \tau_3} \ell_k^2} \les M^{- \frac{1}{2} -
3\gamma },
\end{split}
\end{align}

\noi  
where $w_{N_i}$ are of types (C) or (D).
In view of \eqref{Eqn:e1} and \eqref{Eqn:e3}, we may further reduce \eqref{Eqn:ex1} to the following estimate,
\begin{equation}
  \label{Eqn:ex} 
  \| \mathscr{X}_k \|_{\ell_k^2} \lesssim M^{- \frac{1}{2} -
  3\gamma },
\end{equation}
where $\mathscr{X}_k $ is given by
\begin{equation}
  \label{Eqn:Xk} 
  \mathscr{X}_k = \sum_{k_1, k_2, k_3} \mathrm{T}_{k k_1 k_2
  k_3}^{\text{b}, m} \cdot (w_{N_1})_{k_1} \cdot   \overline{(w_{N_2}) _{k_2}} \cdot (w_{N_3})_{k_3}
\end{equation}
and $w_{N_j}$ are either of the following two types (with a slight abuse of notation, we still call them type (C) and (D)). 
\begin{itemize}
  \item Type (C), where
  \[ 
    (w_{N_j})_{k_j} = \sum_{ k_j' } h_{k_j
    k_j'}^{N_j, L_j} \frac{g_{k_j'} (\omega)}{\jbb{k_j'}^{\alpha / 2}}, 
  \]
  with $h_{k_j k_j'}^{N_j, L_j} (\omega)$ supported in the set $ \{
  \frac{N_j}{2} < \jb{k_j'} \le  N_j\}$,
  $\mathcal{B}_{\le  L_j}$-measurable for some $L_j \le  N_j^{1 -
  \delta}$, and satisfying the bounds
  \begin{equation}
    \label{Eqn:eee1} 
    \begin{split}
      &\| h^{N_j, L_j}_{k_j k_j'} \|_{\ell_{k_j}^2 \to  \ell^2_{k_j'}}
      \lesssim L_j^{- \delta_0},\\
      &\| h^{N_j, L_j}_{k_j k_j'} \|_{\ell_{k_j k_j'}^2} \lesssim
      N^{\frac{1}{2} + \gamma_0}_j L_j^{- \frac{1}{2}},
    \end{split}
  \end{equation}
  for $ \dl \ll \gamma_0 \ll  \gamma \ll \delta_0 \ll \al -1$. Moreover using
  \eqref{Eqn:2} we may assume that $h^{N_j, L_j}_{k_j k_j'}$ is supported in $\{| k_j
  - k_j' | \lesssim N^{\eps} L_j\}$.
  
  \item Type (D), where $(w_{N_j})_{k_j}$ is supported in $\{ \jb{k_j} \lesssim
  N_j \}$, and satisfies
  \begin{equation}
    \label{Eqn:eee2} 
    \| (w_{N_j})_{k_j} \|_{\ell_{k_j}^2} \lesssim N_j^{-     \frac{1}{2} - \gamma} .
  \end{equation}
\end{itemize}

\medskip

We will now analyze the terms (I)-(VII) in the expression of $\mathscr X_k$ given by \eqref{Eqn:Xk}. We will consider eight possible scenarios, 
based on the possible types of $(w_1, w_2, w_3)$ in \eqref{Eqn:Xk}, i.e.
(a) (C, C, C), (b) (C, C, D),
(c) (C, D, C), (d) (D, C, C), (e) (D, D, C), (f) (C, D, D), (g) (D, C, D), and
(h) (D, D, D).

\subsection{High-high interaction}\label{Sub:highhigh}\label{Sub:hh}

In this subsection, we focus on the scenarios where $N_{\med} \gtrsim N_{\max}^{1 - \delta} \gtrsim M^{1 - \delta}$ for a small positive constant $0 < \delta \ll 1$. In particular, these scenarios cover terms (I) and (III) from the previous section.  
As we have seen before, the base tensor ${\rm T}^{{\rm b},m}$ is defined by \eqref{baseT}, which is essentially the indicator function of the set \eqref{Eqn:S} with some constraints and $N_{\med} \gtrsim N_{\max}^{1 - \delta}$.

\subsubsection{Case (a) : (C, C, C)}
\label{SUB:HHCCC}
First, let us consider the non-pairing cases, where $k_1' \neq k_2'$ and $ k_2' \neq k_3'$. In this case, we have
\begin{equation}
  \label{Eqn:CCCn} 
  \mathscr{X}_k = \sum_{k_1, k_2, k_3} \mathrm{T}^{\text{b},
  m}_{k k_1 k_2 k_3} \cdot \sum_{\substack{
    k_1', k'_2, k_3'\\
    k_1' \neq k_2', k_2' \neq k_3'
  }} h_{k_1 k'_1}^{N_1, L_1} \overline{h_{k_2 k'_2}^{N_2, L_2}}
  h_{k_3 k'_3}^{N_3, L_3} \frac{g_{k'_1} \overline{g_{k'_2}}g_{k'_3}}{\jbb{k'_1}^{\frac{\alpha}{2}} \jbb{k'_2}^{\frac{\alpha}{2}}
  \jbb{k'_3}^{\frac{\alpha}{2}}}
\end{equation}
where $h^{N_j, L_j}_{k_j k_j'} = h^{N_j, L_j}_{k_j k_j'} (\omega)$ satisfies
\eqref{Eqn:eee1} and ${\rm T}^{{\rm b},m}_{k k_1 k_2 k_3}$ is defined in
\eqref{baseT}. 

Applying Proposition \ref{Prop:te}, Remark \ref{RMK:abs}, and Proposition \ref{PROP:contr} yields\footnote{Here we only consider the case $N_{\min} = N_{\max}$ for simplicity. When $N_{\min} < N_{\max}$, we need to apply Proposition \ref{Prop:te} and Proposition \ref{PROP:contr} repeatedly as in Subsection \ref{Subsub:cc}, as where $h^{N_j, L_j}_{k_jk_j'}$ may depend on $g_{k_i'}$ for $i\neq j$. }
\begin{align}
\label{CCC2}
\begin{split}
\| \mathscr{X}_k \|_k & \lesssim (N_1 N_2 N_3)^{- \alpha / 2} \bigg\| \sum_{k_1, k_2, k_3} \mathrm{T}^{\text{b},
m}_{k k_1 k_2 k_3} \cdot h_{k_1 k'_1}^{N_1, L_1} \overline{h_{k_2 k'_2}^{N_2, L_2}}
h_{k_3 k'_3}^{N_3, L_3} 
\bigg\|_{kk_1'k_2'k_3'} \\ 
& \les (N_1 N_2 N_3)^{- \alpha / 2} \big\| \mathrm{T}^{\text{b},
  m}_{k k_1 k_2 k_3} 
 \big\|_{kk_1k_2 k_3} \prod_{j=1}^3 \big\| 
  h_{k_j k'_j}^{N_j, L_j} \big\|_{k_j \to k_j'}.        
\end{split}
\end{align}

\noi 
Then from \eqref{Eqn:eee1} and Lemma \ref{LEM:tensor}, we can continue with
\begin{align}
\label{CCC3}     \begin{split}
\| \mathscr{X}_k \|_k & \lesssim (N_1 N_2 N_3)^{- \alpha / 2}  \min  \big(N_3^{2 - \alpha} \log (N_1 \vee
N_2) (N_1 \wedge N_2) + N_1 N_2, \\ 
& \hspace{1.5cm}  N_1^{2 -
\alpha} \log (N_2 \vee N_3) (N_2 \wedge N_3) + N_2 N_3 \big)^{\frac12} \\ 
& \les (N_1 N_2 N_3)^{- \alpha / 2} N_{\max}^{\frac12} (\log N_{\max} )^{\frac12} N_{\min}^{\frac12} \\ 
& \les N_{\max}^{\frac12 - \frac\al2 + \eps} N_{\med}^{-\frac\al2} ,
\end{split}
\end{align}

\noi 
which is sufficient for \eqref{Eqn:ex} since $N_{\med} \ges N_{\max}^{1-\dl}$, provided $\eps, \dl, \g \ll \al - 1$.

Secondly, we will consider the case when $k_1' = k_2'$ (the case when $k_2' =
k_3'$ is similar, and we omit its proof), where we have $N_1 = N_2$.  
For this case, we have
\begin{align}
    \label{Eqn:CCCp}
    \begin{split}
    \mathscr{X}_k & = \sum_{k_1, k_2, k_3} \mathrm{T}^{\text{b},
  m}_{k k_1 k_2 k_3} \cdot \sum_{\substack{ 
    k_1', k_3'
  }} h_{k_1 k'_1}^{N_1, L_1} \overline{h_{k_2 k'_1}^{N_2, L_2}}
  h_{k_3 k'_3}^{N_3, L_3} \frac{|g_{k'_1} |^2 g_{k'_3}}{\jbb{k'_1}^{{\alpha}}  
  \jbb{k'_3}^{\frac{\alpha}{2}}} \\
  & = \sum_{k_1, k_2, k_3} \mathrm{T}^{\text{b},
  m}_{k k_1 k_2 k_3} \cdot \sum_{\substack{ 
    k_1', k_3'
  }} h_{k_1 k'_1}^{N_1, L_1} \overline{h_{k_2 k'_1}^{N_2, L_2}}
  h_{k_3 k'_3}^{N_3, L_3} \frac{(|g_{k'_1} |^2 -1)g_{k'_3}}{\jbb{k'_1}^{{\alpha}}  
  \jbb{k'_3}^{\frac{\alpha}{2}}} \\
  & \hspace{2cm} + \sum_{k_1, k_2, k_3} \mathrm{T}^{\text{b},
  m}_{k k_1 k_2 k_3} \cdot \sum_{\substack{ 
    k_1', k_3'
  }} h_{k_1 k'_1}^{N_1, L_1} \overline{h_{k_2 k'_1}^{N_2, L_2}}
  h_{k_3 k'_3}^{N_3, L_3} \frac{ g_{k'_3}}{\jbb{k'_1}^{{\alpha}}  
  \jbb{k'_3}^{\frac{\alpha}{2}}} \\
  & = \mathscr{X}_k^{(1)} + \mathscr{X}_k^{(2)}.
    \end{split}
\end{align}

\noi 
For the term $\mathscr{X}_k^{(1)} $, we can apply Proposition \ref{Prop:te} with $\eta_{k_1'} = (|g_{k'_1} |^2 -1)$, $\eta_{k_3'} = g_{k'_3}$, or $\eta_{k_1'} = |g_{k'_1}|^2 g_{k'_1}$ for over paired cases, and then similar argument as above, to get
\begin{align}
    \label{CCCp1}
    \begin{split}
   \| \mathscr{X}_k^{(1)}  \|_{k} & \les N_1^{-\al} N_3^{-\frac\al2} \bigg\| \sum_{k_1, k_2, k_3} \mathrm{T}^{\text{b},
  m}_{k k_1 k_2 k_3} \cdot  h_{k_1 k'_1}^{N_1, L_1} \overline{h_{k_2 k'_1}^{N_2, L_2}}
  h_{k_3 k'_3}^{N_3, L_3} 
 \bigg\|_{kk_1'k_3'} \\
 & \les  N_1^{-\al} N_3^{-\frac\al2} \big\| \mathrm{T}^{\text{b},
  m}_{k k_1 k_2 k_3} 
 \big\|_{kk_1k_2 k_3}  \big\| 
  h_{k_1 k'_1}^{N_1, L_1} h_{k_2 k'_1}^{N_2, L_2}  \big\|_{k_1k_2 \to k_1'}  \big\| 
  h_{k_3 k'_3}^{N_3, L_3} \big\|_{k_3 \to k_3'} \\
 & \les  N_1^{-\al} N_3^{-\frac\al2} \big\| \mathrm{T}^{\text{b},
  m}_{k k_1 k_2 k_3} 
 \big\|_{kk_1k_2 k_3}  \prod_{j=1}^3 \big\| 
  h_{k_j k'_j}^{N_j, L_j} \big\|_{k_j \to k_j'} , 
  \end{split}
\end{align}

\noi 
where we used Lemma \ref{LEM:tech} in the last step. By using the same estimate as in \eqref{CCC3}, we have
\begin{align}
    \label{CCCp2}
     \| \mathscr{X}_k^{(1)}  \|_{k} \les N_{\max}^{\frac12 - \frac\al2 + \eps} N_{\med}^{-\frac\al2} ,
\end{align}

\noi 
which is again sufficient for \eqref{Eqn:ex}.
Now we turn to the second term in \eqref{Eqn:CCCp}.
By using the unitary property of $\wt H_{k k'}^{N, L_{N}}$ in Lemma \ref{LEM:uni} (also see Corollary \ref{COR:cancel} and Remark \ref{RMK:cancel}), we may redefine $\mathscr{X}_k^{(2)} $ as follows, still denoted by $\mathscr{X}_k^{(2)}$,
\begin{align}
  \label{Eqn:CCCp2} 
    \begin{split}
    \mathscr{X}_k^{(2)} & = \sum_{k_1, k_2, k_3} \mathrm{T}^{\text{b},
  m}_{k k_1 k_2 k_3} \cdot \sum_{\substack{ 
    k_1', k_3'
  }} \bigg( \frac1{\jbb{k'_1}^{{\alpha}}} - \frac1{\jbb{k_1}^{{\alpha}}}\bigg) h_{k_1 k'_1}^{N_1, L_1} \overline{h_{k_2 k'_1}^{N_2, L_2}}
  h_{k_3 k'_3}^{N_3, L_3} \frac{ g_{k'_3}}{  
  \jbb{k'_3}^{\frac{\alpha}{2}}} \\
  & \hspace{1cm} - \sum_{k_1, k_2, k_3} \mathrm{T}^{\text{b},
  m}_{k k_1 k_2 k_3} \cdot \sum_{\substack{ 
    k_1' 
  }} h_{k_1 k'_1}^{N_1, L_1} \overline{h_{k_2 k'_1}^{N_2, L_2}}
  h_{k_3 k'_1}^{N_3, L_3} \frac{ g_{k'_1}}{  
  \jbb{k'_1}^{\frac{3\alpha}{2}}} \\
  & = : \mathscr{X}_k^{(21)} +\mathscr{X}_k^{(22)} .
  \end{split}
\end{align}

\noi 
Recall that the support of $h_{k_1 k'_1}^{N_1, L_1}$ is in $\{| k_1 - k_1' | \lesssim N_1^{\eps}
L_1\}$. Therefore,
\begin{equation}
  \label{Eqn:gain2} 
  \left| \frac{1}{\jbb{ k_1'}^{\alpha}} -
  \frac{1}{\jbb{k_1}^{\alpha}} \right| \lesssim N^{- \alpha - 1 + \eps}_1 L_1 .
\end{equation}
Let us consider $\mathscr{X}_k^{(21)}$ and $\mathscr{X}_k^{(22)}$
one-by-one.
Let us first consider $\mathscr{X}_k^{(21)}$ first. From \eqref{Eqn:CCCp2} and \eqref{Eqn:gain2}, we apply  Proposition \ref{Prop:te} and
Proposition \ref{PROP:contr} to get
\begin{equation}
  \label{Eqn:pa1} 
  \begin{split}
\| \mathscr{X}_k^{(21)} \|_k & \lesssim N^{- \alpha - 1 + \eps}_1 L_1 N_3^{- \frac{\alpha}{2}} \bigg\| \sum_{k_1, k_2, k_3} {\rm T}^{{\rm b},m}_{k k_1 k_2 k_3} 
    \sum_{\substack{
      k_1' 
    }} h_{k_1 k'_1}^{N_1, L_1} \overline{h_{k_2 k'_1}^{N_1, L_2}}
    h_{k_3 k'_3}^{N_3, L_3}   \bigg\|_{kk_3'} \\
    & \lesssim N^{- \alpha - 1 + \eps}_1 L_1
    N_3^{- \frac{\alpha}{2}} \big\| {\rm T}^{{\rm b},m}_{k k_1 k_2 k_3}  \big\|_{kk_1k_2k_3}  \big\| h_{k_1 k'_1}^{N_1, L_1}  
 \big\|_{k_1  k_1'} \prod_{j=2}^3 \big\| h_{k_j k'_j}^{N_j, L_j}  
 \big\|_{k_j \to k_j'} ,
  \end{split}
\end{equation}

\noi 
which together with Lemma \ref{LEM:tensor} and \eqref{Eqn:eee1} implies that 
\begin{align}
\label{PA11}
\| \mathscr{X}_k^{(21)} \|_k \lesssim N^{- \alpha - 1 + \eps}_1 L_1 N_3^{- \frac{\alpha}{2}} N_{\max}^{\frac12 + \eps} N_{\min}^{\frac12} N_1^{\frac12 + \g_0} L_1^{-\frac12} \les N_1^{-\al + \eps + \g_0} N_3^{-\frac\al2} N_{\max}^{\frac12 + \eps} N_{\min}^{\frac12},
\end{align}

\noi 
which is sufficient for \eqref{Eqn:ex} again since $N_{1}, N_3 \ges N_{\max}^{1-\dl}$.
For the second term of \eqref{Eqn:CCCp}, by applying Proposition \ref{Prop:te} and Proposition \ref{PROP:contr} similarly as above, we obtain
\begin{align}
\label{PA2}
\begin{split}
\|\mathscr{X}_k^{(22)}\|_{k} & \les \bigg\| \sum_{k_1, k_2, k_3} \mathrm{T}^{\text{b},
  m}_{k k_1 k_2 k_3} \cdot \sum_{\substack{ 
    k_1' 
  }} h_{k_1 k'_1}^{N_1, L_1} \overline{h_{k_2 k'_1}^{N_2, L_2}}
  h_{k_3 k'_1}^{N_3, L_3} \frac{ g_{k'_1}}{ \jbb{k'_1}^{\frac{3\alpha}{2}}}  \bigg\|_k \\ 
  & \les N_{\max}^{-\frac{3\alpha}{2}} \bigg\| \sum_{k_1, k_2, k_3} \mathrm{T}^{\text{b},
  m}_{k k_1 k_2 k_3} \cdot  h_{k_1 k'_1}^{N_1, L_1} \overline{h_{k_2 k'_1}^{N_2, L_2}}
  h_{k_3 k'_1}^{N_3, L_3}  \bigg\|_{kk_1'}\\ 
  & \les N_{\max}^{-\frac{3\alpha}{2}} \big\| {\rm T}^{{\rm b},m}_{k k_1 k_2 k_3}  \big\|_{kk_1k_2k_3}  \prod_{j=1}^3 \big\| h_{k_j k'_j}^{N_j, L_j}  
 \big\|_{k_j \to k_j'} ,
    \end{split}
\end{align}

\noi 
where we used Lemma \ref{LEM:tech} twice in the last step.
By using Lemma \ref{LEM:tensor} and \eqref{Eqn:eee1}, from \eqref{PA2} we get
\begin{align}
\label{PA3}    \|\mathscr{X}_k^{(22)}\|_{k} \les N_{\max}^{-\frac{3\alpha}{2}} N_{\max}^{1+\eps},
\end{align}

\noi
where we used $N_1 = N_2 = N_3 = N_{\max}$, which is again sufficient for \eqref{Eqn:ex}.

Finally,
by collecting \eqref{Eqn:CCCp}, \eqref{CCCp2}, \eqref{Eqn:CCCp2}, \eqref{PA11}, and \eqref{PA3}, we conclude that 
\[
\|\mathscr{X}_k\|_{k} \les N_{\max}^{-\frac1{2} - 4 \g } \les M^{-\frac1{2} - 4\g } .
\]

\noi 
Thus we finish the proof of \eqref{Eqn:ex}.

\subsubsection{Case (b) : (C, C, D)}
\label{CCD}

First, consider the non-pairing case when $k_1' \neq k_2'$. In this case, we have
\[ 
  \mathscr{X}_k = \sum_{k_1, k_2, k_3} {\rm T}^{{\rm b},m}_{k k_1 k_2
  k_3}  \sum_{\substack{
    k_1' \neq k_2'
  }} h_{k_1 k'_1}^{N_1, L_1} \overline{h_{k_2 k'_2}^{N_2, L_2}}
  \frac{g_{k'_1} \overline{g_{k'_2}}}{\jbb{k'_1}^{\frac{\alpha}{2}}
  \jbb{k'_2}^{\frac{\alpha}{2}}} (w_{N_3})_{k_3}, 
\]
where $h^{N_j, L_j}_{k_j k_j'} = h^{N_j, L_j}_{k_j k_j'} (\omega)$ satisfies
\eqref{Eqn:eee1}, $(w_{N_3})_{k_3}$ satisfies \eqref{Eqn:eee2} and
${\rm T}^{{\rm b},m}_{k k_1 k_2 k_3}$ is defined in \eqref{baseT}. 
From Proposition \ref{PROP:contr}, we have
\begin{align}
\label{CCD0}
\|\mathscr{X}_k\|_k & \les  \bigg\| \sum_{k_1, k_2} {\rm T}^{{\rm b},m}_{k k_1 k_2   k_3}  \sum_{\substack{
    k_1' \neq k_2'
  }} h_{k_1 k'_1}^{N_1, L_1} \overline{h_{k_2 k'_2}^{N_2, L_2}}
  \frac{g_{k'_1} \overline{g_{k'_2}}}{\jbb{k'_1}^{\frac{\alpha}{2}}
  \jbb{k'_2}^{\frac{\alpha}{2}}} \bigg\|_{k\to k_3} \|  (w_{N_3})_{k_3} 
 \|_{k_3}.
\end{align}

\noi 
Then we apply Proposition \ref{Prop:te} and Proposition \ref{PROP:contr} repeatedly as in Subsection \ref{Subsub:cc} to the right-hand side of \eqref{CCD0} to get
\begin{align}
    \label{CCD1}
    \begin{split}
    \|\mathscr{X}_k\|_k & \les (N_1N_2)^{-\frac{\al}2} \Big( \big\| {\rm T}^{{\rm b},m}_{k k_1 k_2
  k_3} 
 \big\|_{kk_1k_2 \to k_3} + \big\| {\rm T}^{{\rm b},m}_{k k_1 k_2
  k_3} 
 \big\|_{k \to k_1k_2k_3}  \\
 & \hphantom{XXXXXXXX} + \big\| {\rm T}^{{\rm b},m}_{k k_1 k_2
  k_3} 
 \big\|_{kk_1 \to k_2k_3} + \big\| {\rm T}^{{\rm b},m}_{k k_1 k_2
  k_3} 
 \big\|_{kk_2 \to k_1k_3} \Big) \\
 & \hphantom{XXX} \times \prod_{j=1}^2 \big\|h_{k_j k'_j}^{N_j, L_j}  \big\|_{k_j \to k_j'} \|  (w_{N_3})_{k_3} 
 \|_{k_3}.
    \end{split}
\end{align}

\noi 
Then by using Lemma \ref{LEM:tensor3} and \eqref{Eqn:eee1}, from \eqref{CCD1} we obtain
\begin{align*} 
    \|\mathscr{X}_k\|_k & \les (N_1N_2)^{-\frac{\al}2} N_3^{-\frac12 - \g} \big(N_1^{1-\frac\al2} N_2^\eps + N_2^{\frac12 } + (N_1\wedge N_2)^{1-\frac\al2} N_3^\eps + N_3^{\frac12} \\ 
    & \hphantom{XX} + (N_2\wedge N_3)^{1-\frac\al2} N_1^{1-\frac\al2} + (N_1\wedge N_3)^{\frac12-\frac\al4} N_2^{\frac12-\frac\al4} \big) \\ 
 & \les (N_1N_2)^{-\frac{\al}2} N_3^{-\frac12 - \g} N_{\min}^{1-\frac\al2} N_{\max}^{\frac12},
\end{align*}

\noi 
which is sufficient for our purpose since $N_{\med} \ges N_{\max}^{1-\dl}$ and $\eps, \dl, \g \ll \al - 1$.

Let us turn to the pairing case, i.e. $k_1' = k'_2$ (which implies $N_1 = N_2$).
By using Lemma \ref{LEM:uni}, we may consider
\[ 
  \begin{split}
\mathscr{X}_k 
    & = \sum_{k_1, k_2, k_3}
    {\rm T}^{{\rm b},m}_{k k_1 k_2 k_3} \cdot \sum_{k_1'} \left( \frac{1}{\jbb{k'_1}^{\alpha}} -
    \frac{1}{\jbb{k_1}^{\alpha}} \right) h_{k_1 k'_1}^{N_1, L_{1}}
    \overline{h_{k_2 k'_1}^{N_1, L_{2}}} (w_{N_3})_{k_3} \\
    & \quad + \sum_{k_1, k_2, k_3}
    {\rm T}^{{\rm b},m}_{k k_1 k_2 k_3} \cdot \sum_{k_1'} h_{k_1 k'_1}^{N_1, L_{1}} \overline{h_{k_2 k'_1}^{N_1,
    L_{2}}} \frac{(|g_{k'_1} |^2 - 1)}{\jbb{k_1'}^\al} (w_{N_3})_{k_3}\\
    & = : \mathscr{X}_k^{(3)}  + \mathscr{X}_k^{(4)} .
  \end{split} 
\]

\noi 
Similar argument as in \eqref{Eqn:pa1} gives
\begin{align*} 
    \begin{split}
\| \mathscr{X}_k^{(3)} \|_{k} & \les N^{- \alpha - 1 + \eps}_1 L_1 \bigg\| \sum_{k_1, k_2, k_3}
{\rm T}^{{\rm b},m}_{k k_1 k_2 k_3} \cdot \sum_{k_1'}  h_{k_1 k'_1}^{N_1, L_{1}}
\overline{h_{k_2 k'_1}^{N_1, L_{2}}} (w_{N_3})_{k_3} 
\bigg\|_{k} \\
 & \les  N^{- \alpha - 1 + \eps}_1 L_1 \big\| {\rm T}^{{\rm b},m}_{k k_1 k_2 k_3} 
 \big\|_{kk_1k_2 \to k_3}  \big\| h_{k_1 k'_1}^{N_1, L_1}  \big\|_{k_1  k_1'}   \big\| h_{k_2k'_2}^{N_2, L_2}  
 \big\|_{k_2 \to k_2'} \|(w_{N_3})_{k_3} \|_{k_3} \\
 & \les  N^{- \alpha - 1 + \eps}_1 L_1  N_1^{\frac12 } N_1^{\frac12 + \g_0} L_1^{-\frac12} N_3^{-\frac12 - \g} \\
 & \les N^{- \alpha +\frac12 + \eps + \g_0}_1  N_3^{-\frac12 - \g},
    \end{split}
\end{align*}

\noi 
where we used $N_1 = N_2$, which is sufficient for our purpose since $N_{1} , N_3 \ges N_{\max}^{1-\dl}$.
For the term $\mathscr{X}_k^{(4)} $, by Proposition \ref{Prop:te} with $\eta_{k_1'} = |g_{k'_1} |^2 - 1$, we have
\[ 
  \begin{split}
\|\mathscr{X}_k^{(4)}\|_k  & =   \bigg\|\sum_{k_1, k_2, k_3}
    {\rm T}^{{\rm b},m}_{k k_1 k_2 k_3} \cdot \sum_{k_1'} h_{k_1 k'_1}^{N_1, L_{1}} \overline{h_{k_2 k'_1}^{N_1,
    L_{2}}} \frac{(|g_{k'_1} |^2 - 1)}{\jbb{k_1'}^\al} (w_{N_3})_{k_3} \bigg\|_{k} \\
    & \les N_1^{-\al} 
    \bigg(
    \bigg\|\sum_{k_1, k_2}
    {\rm T}^{{\rm b},m}_{k k_1 k_2 k_3} \cdot  h_{k_1 k'_1}^{N_1, L_{1}} \overline{h_{k_2 k'_1}^{N_1,
    L_{2}}}   \bigg\|_{kk_1' \rightarrow k_3} \\
    & \hphantom{XXXX}  +
    \bigg\|\sum_{k_1, k_2}
    {\rm T}^{{\rm b},m}_{k k_1 k_2 k_3} \cdot  h_{k_1 k'_1}^{N_1, L_{1}} \overline{h_{k_2 k'_1}^{N_1,
    L_{2}}}   \bigg\|_{k \rightarrow k_1' k_3} 
    \bigg)
    \| (w_{N_3})_{k_3} \|_{k_3} .
  \end{split} 
\]

\noi 
Then similar argument as in \eqref{CCD1} together with Lemma \ref{LEM:tensor3} and Lemma \ref{LEM:tech}  yields 
\[ 
  \begin{split}
\|\mathscr{X}_k^{(4)}\|_k  
     \les & N_1^{-\al} \big ( \big\| 
    {\rm T}^{{\rm b},m}_{k k_1 k_2 k_3}   \big\|_{kk_1k_2 \to k_3} + \big\| 
    {\rm T}^{{\rm b},m}_{k k_1 k_2 k_3}   \big\|_{k\to k_1k_2k_3} \big)\\
    & \times \big\| h_{k_1 k'_1}^{N_1, L_{1}} \overline{h_{k_2 k'_1}^{N_1,
    L_{2}}} \big\|_{k_1'\to k_1k_2} \|(w_{N_3})_{k_3} \|_{k_3}  \\ 
    \les & N_1^{-\al + \frac12} N_3^{-\frac12 -\g},
  \end{split} 
\]

\noi 
which is again sufficient for our purpose.

The proof for Case (d): (D, C, C) is similar to that of Case (b) (C, C, D); we omit the proof.

\subsubsection{Case (c) : (C, D, C)}

In this case, there is no pairing. Therefore, we have
\[ 
  \mathscr{X}_k = \sum_{k_1, k_2, k_3} {\rm T}^{{\rm b},m}_{k k_1 k_2
  k_3} \cj{(w_{N_2})_{k_2}} \sum_{\substack{
    k_1', k_3'
  }} h_{k_1 k'_1}^{N_1, L_1} {h_{k_3 k'_3}^{N_3, L_3}}
  \frac{g_{k'_1} {g_{k'_3}}}{\jbb{k'_1}^{\frac{\alpha}{2}}
  \jbb{k'_3}^{\frac{\alpha}{2}}} , 
\]
where $h^{N_j, L_j}_{k_j k_j'} = h^{N_j, L_j}_{k_j k_j'} (\omega)$ satisfies
\eqref{Eqn:eee1}, $(w_{N_2})_{k_2}$ satisfies \eqref{Eqn:eee2} and
${\rm T}^{{\rm b},m}_{k k_1 k_2 k_3}$ is defined in \eqref{baseT}. 
Applying Proposition \ref{Prop:te} and Propostion \ref{PROP:contr} similarly as in the previous subsection, we have
\begin{align*} 
    \begin{split}
    \|\mathscr{X}_k\|_k & \les (N_1N_3)^{-\frac{\al}2} \Big( \big\| {\rm T}^{{\rm b},m}_{k k_1 k_2 k_3} \big\|_{kk_1k_3 \to k_2} + \big\| {\rm T}^{{\rm b},m}_{k k_1 k_2 k_3} \big\|_{kk_1 \to k_2k_3}  \\ 
    &\hphantom{XXXXXXX} + \big\| {\rm T}^{{\rm b},m}_{k k_1 k_2 k_3} \big\|_{kk_3 \to k_1 k_2}  + \big\| {\rm T}^{{\rm b},m}_{k k_1 k_2 k_3} \big\|_{k \to k_1k_2k_3 } \Big) \\ 
    &\hphantom{XXX} \times \big\|h_{k_1 k'_1}^{N_1, L_1} \big\|_{k_1' \to k_1} \big\|{h_{k_3 k'_3}^{N_3, L_3}} \big\|_{k_3' \to k_3} \|  (w_{N_2})_{k_2} 
 \|_{k_2}.
    \end{split}
\end{align*}

\noi 
Then by using Lemma \ref{LEM:tensor3}, \eqref{Eqn:eee1}, and \eqref{Eqn:eee2}, we get
\begin{align*} 
    \|\mathscr{X}_k\|_k \les (N_1N_3)^{-\frac{\al}2} N_2^{-\frac12 - \g} N_{\min}^{1-\frac\al2} N_{\max}^{\frac12},
\end{align*}

\noi 
which is sufficient for our purpose since $N_{\med} \ges N_{\max}^{1-\dl}$ and $\eps, \g, \dl \ll \al -1$.

\subsubsection{Case (e) : (D, D, C)}
\label{DDC}
In this case, we have
\begin{equation*}
  \mathscr{X}_k = \sum_{k_1, k_2, k_3} \mathrm{T}^{\text{b},
  m}_{k k_1 k_2 k_3} \cdot (w_{N_1})_{k_1}  \overline{(w_{N_2}
  )_{k_2}} \sum_{ k_3' } h_{k_3 k'_3}^{N_3, L_3}
  \frac{g_{k'_3}}{\jbb{k_3'}^{\frac{\alpha}{2}}},
\end{equation*}
where $h^{N_3, L_3}_{k_3 k_3'} $ satisfies
\eqref{Eqn:eee1}, $(w_{N_1})_{k_1}$ and $\overline{(w_{N_2} )_{k_2}}$ satisfy \eqref{Eqn:eee2}, and the base tensor
${\rm T}^{{\rm b},m}_{k k_1 k_2 k_3}$ is defined in \eqref{baseT}.
From Proposition \ref{Prop:te} and Propostion \ref{PROP:contr}, we have
\begin{align*} 
    \begin{split}
    \|\mathscr{X}_k\|_k & \les N_3^{-\frac{\al}2} \big( \big\| {\rm T}^{{\rm b},m}_{k k_1 k_2 k_3} \big\|_{kk_3 \to k_1 k_2} + \big\| {\rm T}^{{\rm b},m}_{k k_1 k_2 k_3} \big\|_{k\to k_1 k_2k_3 }\big) \\
    & \hphantom{XXX} \times \big\| {h_{k_3 k'_3}^{N_3, L_3}}  \big\|_{k_3' \to k_3}  \|  (w_{N_1})_{k_1} 
 \|_{k_1} \|  (w_{N_2})_{k_2} \|_{k_2}.
    \end{split}
\end{align*}

\noi 
Then by using Lemma \ref{LEM:tensor2} and \eqref{Eqn:eee1}-\eqref{Eqn:eee2} to get
\begin{align*} 
\|\mathscr{X}_k\|_k \les  N_3^{-\frac{\al}2} (N_1N_2)^{-\frac12 - \g} \big( (N_{\min} N_{\max})^{1 -\frac\al2} + (N_1 \wedge N_2)^{1-\frac\al 2} (\log N_3)^{\frac12} + N_3^{\frac12}\big) ,
\end{align*}

\noi 
which is again sufficient for our purpose since { $N_{\med} \ges N_{\max}^{1-\dl}$}.

The proof for Case (f): (C, D, D) is similar to that of Case (e): (D, D, C); we omit the proof.

\subsubsection{Case (g) : (D, C, D)}

In this case, we have
\begin{equation*}
  \mathscr{X}_k = \sum_{k_1, k_2, k_3} \mathrm{T}^{\text{b},
  m}_{k k_1 k_2 k_3} \cdot (w_{N_1})_{k_1} \sum_{k_2'} \overline{h_{k_2 k'_2}^{N_2, L_2}}
  \frac{ \overline{g_{k'_2}}}{
  \jbb{k'_2}^{\frac{\alpha}{2}}} (w_{N_3})_{k_3},
\end{equation*}
where $h^{N_2, L_2}_{k_2 k_2'} = h^{N_2, L_2}_{k_2 k_2'} (\omega)$ satisfies
\eqref{Eqn:eee1}, $(w_{N_1})_{k_1}$ and $\left( {w_{N_3}}
\right)_{k_3}$ satisfy \eqref{Eqn:eee2}, and the base tensor
${\rm T}^{{\rm b},m}_{k k_1 k_2 k_3}$ is defined in \eqref{baseT}.
From Proposition \ref{Prop:te} and Proposition \ref{PROP:contr}, we have
\begin{align} 
\label{DCD2}
    \begin{split}
    \|\mathscr{X}_k\|_k   
  & \les N_2^{-\frac{\al}2} \big( \big\| {\rm T}^{{\rm b},m}_{k k_1 k_3 k_2} \big\|_{kk_1k_2 \to  k_3} + \big\| {\rm T}^{{\rm b},m}_{k k_1 k_3 k_2} \big\|_{kk_1 \to k_2 k_3} \big) \\ 
  & \hphantom{XXX} \times \big\| h_{k_2 k'_2}^{N_2, L_2} \big\|_{k_2' \to k_2}  \|  (w_{N_1})_{k_1} 
 \|_{k_1} \|  (w_{N_3})_{k_3} \|_{k_3}.
    \end{split}
\end{align}

\noi 
Then by using Lemma \ref{LEM:tensor2}, Lemma \ref{LEM:tensor3}, \eqref{Eqn:eee1}, and \eqref{Eqn:eee2} to get
\begin{align*} 
\|\mathscr{X}_k\|_k  \les  N_2^{-\frac{\al}2} (N_1N_3)^{-\frac12 - \g} \big(  N_1^{1-\frac\al 2} (\log N_2)^{\frac12} + N_2^{\frac12} + (N_{\min} N_{\max})^{1 -\frac\al2} \big) ,
\end{align*}

\noi 
which is again sufficient for our purpose since { $N_{\med} \ges N_{\max}^{1-\dl}$}.

\subsubsection{Case (h) : (D, D, D)}

In this case, we have
\begin{equation*} 
  \mathscr{X}_k = \sum_{k_1, k_2, k_3} \mathrm{T}^{\text{b},
  m}_{k k_1 k_2 k_3} \cdot (w_{N_1})_{k_1} \overline{(w_{N_2}
  )_{k_2}} (w_{N_3})_{k_3},
\end{equation*}
where $(w_{N_j})_{k_j}$ for $j \in \{ 1, 2, 3 \}$ satisfy \eqref{Eqn:eee2},
and the base tensor ${\rm T}^{{\rm b},m}_{k k_1 k_2 k_3}$ is defined in \eqref{baseT}. 
By applying Proposition \ref{PROP:contr}, we obtain
\[ 
  \begin{split}
    & \| \mathscr{X}_k \|_k \lesssim 
    \min \big( \| {\rm T}^{{\rm b},m}_{k k_1 k_2 k_3} \|_{kk_3 
    \to  k_1k_2}, \| {\rm T}^{{\rm b},m}_{k k_1 k_2 k_3} \|_{k k_1
    \to k_2 k_3}\big) \prod_{j=1}^3 \|  (w_{N_j})_{k_j} \|_{k_j}.
  \end{split} 
\]

\noi 
By using Lemma \ref{LEM:tensor2} and \eqref{Eqn:eee2}, we get
\[
\| \mathscr{X}_k \|_k \lesssim (N_1 N_2 N_3)^{- \frac{1}{2} - \gamma} (N_{\min} N_{\med})^{1-\frac\al2},
\]

\noi 
which is sufficient for \eqref{Eqn:eee2}.

\subsection{random averaging operator}\label{Sub:rao}

In this section, we estimate terms (II) and (IV), which correspond to
$\mathcal{P}^+$ and $\mathcal{P}^-$ of \eqref{Eqn:5}, respectively. We only consider the term (II) since the argument for the term  (IV) is similar. 
To show \eqref{Eqn:Awww1} for  the term (II), it
suffices to show
\begin{equation}
  \label{Eqn:ppn} 
  \| \mathcal{P}^{+} [z_M] \|_{X^b} \lesssim M^{-
  \frac{1}{2} - \gamma} .
\end{equation}
We see from the definition of terms (II) that $N_{\med} = \max (N_1, N_2)$ since both $N_1$ and $N_2$ are smaller than $N_3$ and $N_3$ is approximately equal to $M$. The case where $N_{\med} \gtrsim M^{1 - \delta}$ corresponds to the high-high interaction, and we have already established \eqref{Eqn:ppn} for this case in the previous subsection. Therefore, we will only concentrate on the case where $N_{\med} \ll M^{1 - \delta}$ in the following discussion. Applying \eqref{r_OP} and \eqref{Eqn:eee2}, we obtain
\[
\| \mathcal{P}^{\pm} [z_M] \|_{X^b} \lesssim  \| \mathcal{P}^{\pm}   \|_{Y^{b,b}} \| z_M\|_{X^b} \lesssim  M^{-\frac12 - \g}.
\]

\noi 
This completes the proof of \eqref{Eqn:Awww1} for this case.

\subsection{The Term (V)}

This subsection shows that the term (V) satisfies the bound \eqref{Eqn:Awww}. It suffices to show \eqref{Eqn:ex}. We
may assume that $N_1, N_3 \ll N_2^{1 - \delta}$; otherwise, it has been dealt with in Subsection \ref{Sub:highhigh}. 
As $w_{N_2}$ is of type (C), we only need to consider the cases where the types of $(w_{N_1}, w_{N_2}, w_{N_3})$ are
(a) (C, C, C), (b) (C, C, D), (d) (D, C, C), (g) (D, C, D). We set $N_2 = M$
in this subsection.

Before proceeding, let us recall that
\begin{align}
    \label{termV}
   \mathscr{X}_k = \sum_{k_1, k_2, k_3} \mathrm{T}_{k k_1 k_2
  k_3}^{\text{b}, m} \cdot (w_{N_1})_{k_1} \cdot \sum_{ k_2' } \cj{h_{k_2
 k_2'}^{N_2, L_2}} \frac{\cj{g_{k_2'} (\omega)}}{\jbb{k_2'}^{\alpha / 2}} \cdot (w_{N_3})_{k_3}    
\end{align}

\noi 
where we may take that $N_1, N_3 \ll N_2^{1 -\delta}$ with $h_{k_2 k_2'}^{N_2, L_2} (\omega)$ supported in the set $ \{ (k_2,k_2');
\frac{N_2}{2} < \jb{k_2'} \le  N_2, \, | k_2 - k_2' | \lesssim N_2^{\eps} L_2 \}$,
$\mathcal{B}_{\le  L_2}$-measurable for some $L_2 \le  N_2^{1 -
\delta}$, and satisfying the bounds \eqref{Eqn:eee1}.
We note that $| k_2 - k_2' | \lesssim N_2^{\eps} L_2 \les N_2^{\delta / 2} N^{1 - \delta}_2$, which together with $\jb{k_2'} \ge N_2/2$ implies that  $ | k_2 | \ge {N_2}/{4}$, provided $N_2$ is sufficiently large. Furthermore, from $N_1, N_3 \ll N_2^{1 -\delta}$, we have\begin{equation}
  \label{Eqn:vc} 
  |k_1 + k_3| \ll  |k_2| .
\end{equation}
From \eqref{Eqn:lowb+} and \eqref{Eqn:vc}, we have that
\[ 
|2k - (k_1 + k_3)| = |2k_2 - (k_1 + k_3)| \ges |k_1+ k_3|, 
\]
which, together with \eqref{Eqn:bad1}, implies that the summand of \eqref{termV} is non-zero only for those $(k,k_1,k_2,k_3)$ such that 
\begin{equation}
  \label{Eqn:importantfact} 
   (k,k_2) \in S_{k_1 k_3}^{\rm good} ,
\end{equation}

\noi  
where $S_{k_1 k_3}^{\rm good}$ is given in \eqref{Eqn:bad1}.
The above observation will be crucial in our later analysis. In particular, instead of \eqref{termV}, we will consider the following 
\begin{align}
    \label{termV1}
    \mathscr{X}_k = \sum_{\substack{k_1, k_2, k_3\\ |k_1 + k_3| \ll  |k_2|}} \mathrm{T}_{k k_1 k_2
  k_3}^{\text{b}, m} \cdot (w_{N_1})_{k_1} \cdot \sum_{ k_2' } \cj{h_{k_2
 k_2'}^{N_2, L_2}} \frac{\cj{g_{k_2'} (\omega)}}{\jbb{k_2'}^{\alpha / 2}} \cdot (w_{N_3})_{k_3} .
\end{align} 

Now, let us consider \eqref{termV1} with the restriction \eqref{Eqn:vc} (or \eqref{Eqn:importantfact}) case-by-case.

\subsubsection{Case (a) : (C, C, C)}

There is no pairing since $N_1, N_3 \ll N_2 = M$. So we have
\begin{equation*} 
  \mathscr{X}_k = \sum_{\substack{k_1, k_2, k_3\\ |k_1 + k_3| \ll  |k_2|}} \mathrm{T}^{\text{b},
  m}_{k k_1 k_2 k_3} \cdot \sum_{\substack{
    k_1', k'_2, k_3' 
  }} h_{k_1 k'_1}^{N_1, L_1} \overline{h_{k_2 k'_2}^{N_2, L_2}}
  h_{k_3 k'_3}^{N_3, L_3} \frac{g_{k'_1} \overline{g_{k'_2}}g_{k'_3}}{\jbb{k'_1}^{\frac{\alpha}{2}} \jbb{k'_2}^{\frac{\alpha}{2}}
  \jbb{k'_3}^{\frac{\alpha}{2}}}.
\end{equation*}

\noi 
Same argument as in Subsection \ref{SUB:HHCCC}, followed by Proposition \ref{PROP:contr} and Lemma \ref{LEM:tensor4}, yields
\begin{align*}    
    \begin{split}
    \| \mathscr{X}_k \|_k  & \les  (N_1 N_2 N_3)^{- \frac\alpha2}  \bigg\| \sum_{\substack{k_1, k_2, k_3\\ |k_1 + k_3| \ll  |k_2|}} 
    \mathrm{T}^{\text{b},
  m}_{k k_1 k_2 k_3} \cdot h_{k_1 k'_1}^{N_1, L_1} \overline{h_{k_2 k'_2}^{N_2, L_2}}
  h_{k_3 k'_3}^{N_3, L_3} 
 \bigg\|_{kk_1'k_2'k_3'}   \\ 
 & \les  (N_1 N_2 N_3)^{- \frac\alpha2} \big\| \ind_{|k_1 + k_3| \ll  |k_2|}  \mathrm{T}^{\text{b},
  m}_{k k_1 k_2 k_3} 
 \big\|_{kk_1k_2 k_3} \prod_{j=1}^3 \big\| 
  h_{k_j k'_j}^{N_j, L_j} \big\|_{k_j \to k_j'} \\
  & \les (N_1 N_2 N_3)^{- \frac\alpha2}  (N_1N_3)^{\frac12}  \les N_{2}^{-\frac\al2},
\end{split}
\end{align*}

\noi 
which is sufficient for our purpose. 

\subsubsection{Case (b) : (C, C, D)}

Similar to the above, there is no pairing since $N_1 \ll N_2 = M$. From \eqref{termV1}, we have
\[ 
  \mathscr{X}_k = \sum_{\substack{k_1, k_2, k_3\\ |k_1 + k_3| \ll  |k_2|}} {\rm T}^{{\rm b},m}_{k k_1 k_2
  k_3} \cdot \sum_{\substack{
    k_1', k_2'
  }} h_{k_1 k'_1}^{N_1, L_1} \overline{h_{k_2 k'_2}^{N_2, L_2}}
  \frac{g_{k'_1} \overline{g_{k'_2}}}{\jbb{k'_1}^{\frac{\alpha}{2}}
  \jbb{k'_2}^{\frac{\alpha}{2}}} (w_{N_3})_{k_3}, 
\]
where $h^{N_j, L_j}_{k_j k_j'}$ for $j = 1,2$ satisfy
\eqref{Eqn:eee1}, $(w_{N_3})_{k_3}$ satisfies \eqref{Eqn:eee2}, and
${\rm T}^{{\rm b},m}_{k k_1 k_2 k_3}$ is defined in \eqref{baseT}. 
By a similar argument as in \eqref{CCD1} together with \eqref{CSinq}, 
we have
\[ 
  \begin{split}
    \| \mathscr{X}_k \|_k & \les  \bigg\|\sum_{\substack{k_1, k_2\\ |k_1 + k_3| \ll  |k_2|}} {\rm T}^{{\rm b},m}_{k k_1 k_2
  k_3} \cdot \sum_{\substack{
    k_1', k_2'
  }} h_{k_1 k'_1}^{N_1, L_1} \overline{h_{k_2 k'_2}^{N_2, L_2}}
  \frac{g_{k'_1} \overline{g_{k'_2}}}{\jbb{k'_1}^{\frac{\alpha}{2}}
  \jbb{k'_2}^{\frac{\alpha}{2}}} 
 \bigg\|_{k\to k_3 } \| (w_{N_3})_{k_3} \|_{k_3} \\
 & \les  (N_1 N_2)^{- \frac{\alpha}{2}}  \big \|\ind_{|k_1 + k_3| < |k_2|}
    {\rm T}^{{\rm b},m}_{k k_1 k_2 k_3} \big \|_{k k_1 k_2 k_3} \prod_{j=1}^2 \big \| h_{k_j k'_j}^{N_j, L_j} \big \|_{k_j \to  k_j'}  \| (w_{N_3})_{k_3} \|_{k_3} .
  \end{split} 
\]

\noi 
Then we apply Lemma \ref{LEM:tensor4}, together with \eqref{Eqn:eee1} and \eqref{Eqn:eee2}, to get
\[ 
  \begin{split}
    \| \mathscr{X}_k \|_k 
    &  \lesssim (N_1 N_2)^{- \frac{\alpha}{2} } N_3^{- \frac{1}{2} -
    \gamma} (N_1N_3)^{\frac{1}{2}} \les N_2^{-\frac\al2},
  \end{split} 
\]

\noi 
which is enough for our purpose.

The proof for Case (d): (D, C, C) is similar to that of Case (b): (C, C, D); we omit the proof.

\subsubsection{Case (g) : (D, C, D)}

In this case, we consider
\[ 
  \mathscr{X}_k = \sum_{\substack{k_1, k_2, k_3\\ |k_1 + k_3| \ll  |k_2|}} {\rm T}^{{\rm b},m}_{k k_1 k_2
  k_3} \cdot \sum_{k_2'
  } \overline{h_{k_2 k'_2}^{N_2, L_2}}
  \frac{\overline{g_{k'_2}}}{\jbb{k'_2}^{\frac{\alpha}{2}}} (w_{N_1})_{k_1}
  (w_{N_3})_{k_3} . 
\]

\noi 
By a similar argument as in \eqref{DCD2}, we have
\begin{align*}
    \begin{split}
        \|\mathscr{X}_k\|_k & \les N_2^{-\frac{\al}2} \big\| \ind_{|k_1 + k_3| < |k_2|} {\rm T}^{{\rm b},m}_{k k_1 k_3 k_2} \big\|_{kk_1k_2 k_3} \big\| h_{k_2 k'_2}^{N_2, L_2} \big\|_{k_2' \to k_2}  \|  (w_{N_1})_{k_1} 
 \|_{k_1} \|  (w_{N_3})_{k_3} \|_{k_3}.
    \end{split}
\end{align*}

\noi 
Then from Lemma \ref{LEM:tensor4}, \eqref{Eqn:eee1}, and \eqref{Eqn:eee2}, we have
\[ 
  \begin{split}
   \| \mathscr{X}_k \|_k  & \lesssim N_2^{- \frac{\alpha}{2}} N^{- \frac{1}{2}
    - \gamma}_1 N_3^{- \frac{1}{2} - \gamma}\big\| \ind_{|k_1 + k_3| < |k_2|} {\rm T}^{{\rm b},m}_{k k_1 k_3 k_2} \big\|_{kk_1 k_2  k_3}  \\
    & \lesssim N^{- \frac{1}{2} - \gamma}_1 N_2^{- \frac{\alpha}{2}} N_3^{-
    \frac{1}{2} - \gamma} N_1^{\frac{1}{2}} N_3^{\frac{1}{2}}\lesssim N_2^{-
    \frac{\alpha}{2}},
  \end{split} 
\]
which is again sufficient for \eqref{Eqn:ex}.

\subsection{The \texorpdfstring{$\Gamma$}{Lg}-condition case}

In this subsection, we consider \eqref{Eqn:Bwww1}, which shows that the term (VI), defined at the beginning of Section \ref{Sec:rems}, satisfies the bound
\begin{equation}
  \label{Eqn:termB} 
  \| \mathscr{B} (w_{N_1}, w_{N_2}, w_{N_3}) \|_{X^b} \lesssim
  M^{- \frac{1}{2} - \gamma} ,
\end{equation}

\noi 
where $\mathscr B$ is given in \eqref{termB}.
Same argument as in Subsection \ref{Sub:pre} reduces \eqref{Eqn:termB} to the
the following estimate
\begin{equation}
  \label{Eqn:gag} 
  \| \mathscr{X}_k \|_k \lesssim M^{- \frac{1}{2} - 3\gamma},
\end{equation}
where $\mathscr{X}_k$ is given in \eqref{Eqn:Xk}.

To prove \eqref{Eqn:gag}, we make several assumptions. To make the term (VI) non-trivial, we should have $N_{\max} \gtrsim M$. If $N_2 \gtrsim M$, then
term (VI) can be handled in the same way as the terms (IV) and (V). Therefore, we only need to consider the cases that $N_1$ or  $N_3 \ges M$. 
Due to the symmetry between $w_{N_1}$ and $w_{N_3}$, we may assume { {$N_3
\gtrsim M$}} in what follows. 
To prove \eqref{Eqn:gag}, it only suffices to show
\begin{equation}
  \label{Eqn:gag1} 
  \| \mathscr{X}_k \|_k \lesssim N_3^{- \frac{1}{2} - 3\gamma },
\end{equation}

\noi 
where $\mathscr X_k $ is given in \eqref{Eqn:Xk}.
We may further assume that ${ {N_1,
N_2 \ll N_3^{1 - \delta}}}$ as otherwise it has been dealt with in Subsection
\ref{Sub:hh}. Furthermore, we may assume that ${ {w_{N_3} =
\psi_{N_3, L_{N_3}}}}$; otherwise, it can be handled by using the random averaging operator as in Subsection \ref{Sub:rao}. 

Due to the projections $\Delta_M$
($\frac{M}{2} < \jb{k} \le  M$), the requirement $N_{\max} \le 
\frac{M}{2}$ in the definition of the term (VI), we observe that
\[ \begin{split} 
  \jb{k} > \frac{M}{2} \ge  \jb{k_3} & \Longrightarrow | k | \ge   
  \Big( \Big( \frac{M}{2} \Big)^{2} - 1 \Big)^{1 /2}
  \ge  | k_3 | = | k - k_1 + k_2 | \ge  | k |
  - 2 N_{\med} \\
& \Longrightarrow | k_3 | + 2
  N_{\med} \ge  | k | \ge  \Gamma
  \ge  | k_3 | ,
\end{split}
\] 
where $\Gamma := ( ( {M}/{2} )^{2} - 1 )^{1 /2}$. 
Therefore, to make $\mathscr X_k$ is non-trivial, we need  $(k,k_1,k_2, k_3) \in B_{\Gamma}$, where $B_\Gamma$ is given in \eqref{gammaB}. 

With the above argument, we focus on \eqref{Eqn:gag1} with 
\begin{align}
\label{Eqn:gag2}
    \mathscr{X}_k = \sum_{ {k_1, k_2, k_3 }}  \ind_{B_{\Gamma}} (k,k_1,k_2, k_3)  \mathrm{T}_{k k_1 k_2
  k_3}^{\text{b}, m} \cdot (w_{N_1})_{k_1} \cdot \overline{(w_{N_2}
  )_{k_2}} \cdot \sum_{\substack{ k_3' }}  h_{k_3 k'_3}^{N_3, L_3} \frac{ g_{k'_3}}{ \jbb{k'_3}^{\frac{\alpha}{2}}} ,
\end{align} 

\noi 
where $N_1, N_2 \ll N_3$.
In the following, we distinguish several cases for the possible types of $( w_{N_1}, w_{N_2})$: Case (a) (C, C), Case (c)
(C, D), Case (d) (D, C), Case (e) (D, D).

\subsubsection{Case (a) : (C, C, C) }
We start with the non-pairing cases.
Same argument as in Subsection \ref{SUB:HHCCC}, followed by Corollary \ref{COR:gammaT}, yields
\[ 
  \begin{split}
    \| \mathscr{X}_k \|_k   \lesssim (N_1 N_2 N_3)^{- \frac{\alpha}{2}} \| \ind_{B_\Gamma}
    \mathrm{T}_{k k_1 k_2 k_3}^{\text{b}, m} \|_{k k_1 k_2 k_3} \lesssim (N_1 N_2 M)^{- \frac{\alpha}{2}} N^{\frac{1}{2}}_{\min}
    N^{\frac{1}{2}}_{\med} \les  N_3^{-
    \frac{\alpha}{2}},
  \end{split} 
\]

\noi 
which is sufficient for \eqref{Eqn:gag1}.

We turn to the pairing cases. 
Similar as in Subsection \ref{SUB:HHCCC}, we can write \eqref{Eqn:gag2} (note that $w_{N_1}$ and $w_{N_2}$ are both of type (C)) as
\begin{align*} 
    \begin{split}
    \mathscr{X}_k & = \sum_{k_1, k_2, k_3} \ind_{B_\Gamma} \mathrm{T}^{\text{b},
  m}_{k k_1 k_2 k_3} \cdot \sum_{\substack{ 
    k_1' \neq k_3'
  }} h_{k_1 k'_1}^{N_1, L_1} \overline{h_{k_2 k'_1}^{N_2, L_2}}
  h_{k_3 k'_3}^{N_3, L_3} \frac{|g_{k'_1} |^2 g_{k'_3}}{\jbb{k'_1}^{{\alpha}}  
  \jbb{k'_3}^{\frac{\alpha}{2}}} \\
  & = \sum_{k_1, k_2, k_3} \ind_{B_\Gamma} \mathrm{T}^{\text{b},
  m}_{k k_1 k_2 k_3} \cdot \sum_{\substack{ 
    k_1' \neq k_3'
  }} h_{k_1 k'_1}^{N_1, L_1} \overline{h_{k_2 k'_1}^{N_2, L_2}}
  h_{k_3 k'_3}^{N_3, L_3} \frac{(|g_{k'_1} |^2 -1)g_{k'_3}}{\jbb{k'_1}^{{\alpha}}  
  \jbb{k'_3}^{\frac{\alpha}{2}}} \\
  & \hspace{2cm} + \sum_{k_1, k_2, k_3} \ind_{B_\Gamma} \mathrm{T}^{\text{b},
  m}_{k k_1 k_2 k_3} \cdot \sum_{\substack{ 
    k_1' \neq k_3'
  }} h_{k_1 k'_1}^{N_1, L_1} \overline{h_{k_2 k'_1}^{N_2, L_2}}
  h_{k_3 k'_3}^{N_3, L_3} \frac{ g_{k'_3}}{\jbb{k'_1}^{{\alpha}}  
  \jbb{k'_3}^{\frac{\alpha}{2}}} \\
  & = \mathscr{X}_k^{(1)} + \mathscr{X}_k^{(2)}.
    \end{split}
\end{align*}

\noi 
For the first term $\mathscr{X}_k^{(1)}$, we follow the computation in \eqref{CCCp1} and use Corollary \ref{COR:gammaT} to get 
\begin{align*}
    \begin{split}
   \| \mathscr{X}_k^{(1)}  \|_{k}  
 & \les  N_1^{-\al} N_3^{-\frac\al2} \big\| \ind_{B_\Gamma}\mathrm{T}^{\text{b},
  m}_{k k_1 k_2 k_3} 
 \big\|_{kk_1k_2 k_3}  \prod_{j=1}^3 \big\| 
  h_{k_j k'_j}^{N_j, L_j} \big\|_{k_j \to k_j'} \les N_1^{-\al + 1} N_3^{-\frac\al2} , 
  \end{split}
\end{align*}

\noi 
which is enough for our purpose. As to the second term $\mathscr{X}_k^{(2)}$, we follow the computation in \eqref{Eqn:CCCp2}, \eqref{Eqn:pa1}, and \eqref{PA2} to get 
\begin{align*}
    \begin{split}
   \| \mathscr{X}_k^{(2)}  \|_{k}  
 & \les  N^{- \alpha - 1 + \eps}_1 L_1
    N_3^{- \frac{\alpha}{2}} \big\|\ind_{B_\Gamma} {\rm T}^{{\rm b},m}_{k k_1 k_2 k_3}  \big\|_{kk_1k_2k_3}  \big\| h_{k_1 k'_1}^{N_1, L_1}  
 \big\|_{k_1  k_1'} \prod_{j=2}^3 \big\| h_{k_j k'_j}^{N_j, L_j}  
 \big\|_{k_j \to k_j'} \\ 
 & \les  N_1^{-\al + \eps + \g_0} N_3^{-\frac\al2} N_{\med}^{\frac12} N_{\min}^{\frac12}, 
  \end{split}
\end{align*}

\noi 
which is again sufficient since $N_{\min} = N_{\med} = N_1$.

\subsubsection{Case (b) : (C, D, C) and Case (c) : (D, C, C) }

We consider Case (b) only as the argument for Case (c) is
similar. In this case, we can write \eqref{Eqn:gag2} as
\[
  \mathscr{X}_k = \sum_{k_1, k_2, k_3} \ind_{B_\Gamma} {\rm T}^{{\rm b},m}_{k k_1 k_2
  k_3} \cdot \overline{(w_{N_2}
  )_{k_2}} \sum_{\substack{
    k_1', k_3' }} h_{k_1 k'_1}^{N_1, L_1} h_{k_3 k'_3}^{N_3, L_3} \frac{g_{k'_1}g_{k'_3}}{\jbb{k'_1}^{\frac{\alpha}{2}} \jbb{k'_3}^{\frac{\alpha}{2}}}.
\]

\noi 
By a similar argument as in Subsection \ref{CCD},
we obtain
\begin{align*} 
    \begin{split}
    \|\mathscr{X}_k\|_k  
  & \les (N_1N_3)^{-\frac{\al}2} \big\| \ind_{B_\Gamma}  {\rm T}^{{\rm b},m}_{k k_1 k_3 k_2} \big\|_{kk_1k_3 k_2} \big\|h_{k_1 k'_1}^{N_1, L_1} \big\|_{k_1' \to k_1} \big\|{h_{k_3 k'_3}^{N_3, L_3}} \big\|_{k_3' \to k_3} \|  (w_{N_2})_{k_2} 
 \|_{k_2}.
    \end{split}
\end{align*}

\noi 
Then by using Corollary \ref{COR:gammaT}, \eqref{Eqn:eee1}, and \eqref{Eqn:eee2}, we get
\begin{align*} 
    \|\mathscr{X}_k\|_k \les  (N_1N_3)^{-\frac{\al}2} N_2^{-\frac12 - \g} (N_{\min} N_{\med})^{\frac12 } \les N_3^{-\frac\al2},
\end{align*}

\noi 
which is sufficient for our purpose.

\subsubsection{Case (d) : (D, D, C)}

Similar argument as in Subsection \ref{DDC}, from \eqref{Eqn:gag2} we have
\[ 
  \begin{split}
  \| \mathscr{X}_k \|_k 
  &  \les  N_3^{-\frac{\al}2} \big\| \ind_{B_\Gamma} {\rm T}^{{\rm b},m}_{k k_1 k_3 k_2} \big\|_{k k_1 k_2k_3 } \big\| {h_{k_3 k'_3}^{N_3, L_3}}  \big\|_{k_3' \to k_3}  \|  (w_{N_1})_{k_1} 
 \|_{k_1} \|  (w_{N_2})_{k_2} \|_{k_2} \\  
 & \lesssim (N_1 N_2)^{- \frac{1}{2} - \gamma} N_3^{- \frac{\alpha}{2}} \big\| \ind_{B_\Gamma}  {\rm T}^{{\rm b},m}_{k k_1 k_3 k_2} \big\|_{k k_1 k_2k_3}  \\
 & \lesssim (N_1 N_2)^{- \frac{1}{2} - \gamma} N_3^{- \frac{\alpha}{2}}
   (N_{\min} N_{\med})^{ \frac{1}{2}} \le  N_3^{-
    \frac{\alpha}{2}},
  \end{split} 
\]

\noi 
where we applied Proposition \ref{PROP:contr} in the second step,
which is again sufficient for \eqref{Eqn:gag}.

\subsection{Resonant case}

Finally, we estimate term (VII), the resonant case. 
In this case, we may assume { {$M = N_1 = N_2 = N_3$}}. If $(w_{N_1}, w_{N_2}, w_{N_3})$ is of type (D, D, D),
then we have
\[ 
  \begin{split}
    \| \chi (t) (\ensuremath{\operatorname{VII}}) \|_{X^b} & \lesssim \|
    \mathcal{R} (w_{N_1}, w_{N_2}, w_{N_3}) (t) \|_{L^2_{t, x}} \\
    & = \bigg( \int_{\mathbb{R}} \sum_k | (w_{N_1})_k \cj{(w_{N_2})_k} (w_{N_3})_k
    |^2 d t \bigg)^{\frac{1}{2}} \\
    &  \lesssim \prod_{i = 1}^3 \| w_{N_i} \|_{X^{\frac{1}{3}}} \lesssim \prod_{i = 1}^3 \| w_{N_i} \|_{X^{b}} \les   M^{- \frac{3}{2} - 3 \gamma},
  \end{split} 
\]
where we used \eqref{Eqn:e1}, 
which is sufficient for \eqref{Eqn:3}.

If $(w_{N_1}, w_{N_2}, w_{N_3})$ is of type (C, C, C), i.e. Case (a), we have
\[ 
\| \chi (t)(\ensuremath{\operatorname{VII}}) \|_{X^b} \lesssim \|
\mathcal{R} (w_{N_1}, w_{N_2}, w_{N_3}) (t) \|_{L^2_{t, x}},
\] 

\noi 
where $w_N$ is given by \eqref{Eqn:C1}. 
We only consider the non-pairing cases, i.e.
\[
\begin{split}
\big( & \mathcal{R} (w_{N_1}, w_{N_2}, w_{N_3}) (t) \big )_k \\
& =  \chi(t) \sum_{k_1,k_2,k_3} {\rm T}_{kk_1k_2k_3}^{{\rm b},m} 
\sum_{\substack{ k_1', k_2', k_3'\\  k_2' \notin \{k_1', k_3'\} }} 
{h^{N_1, L_1}_{k k_1'}} (t)  
\cj{{h^{N_2, L_2}_{k k_2'}}} (t) 
{h^{N_3, L_3}_{k k_3'}} (t)  
\frac{g_{k_1'} \cj{g_{k_2'}} g_{k_3'}}{\jbb{k_1'}^{\frac{\alpha}{2}}
\jbb{k_2'}^{\frac{\alpha}{2}} \jbb{ k_3'}^{\frac{\alpha}{2}}}  .
\end{split}
\] 

\noi 
Then, from Proposition \ref{Prop:te} and then Proposition \ref{PROP:contr}, we have
\[
\begin{split}
 \| & \mathcal{R} (w_{N_1}, w_{N_2}, w_{N_3}) (t) \|_{L^2_{t, x}} \\
& \lesssim  M^{-\frac{3\al}2}  \big \| \chi(t) {h^{N_1, L_1}_{k k_1'}} \cj{{h^{N_2, L_2}_{k k_2'}}} {h^{N_3,
    L_3}_{k k_3'}} \big \|_{L^2_T( kk_1'k_2'k_3')} \\
& \lesssim T^{\frac12}  M^{-\frac{3\al}2} \|
     {h^{N_1, L_1}_{k k_1'}} \|_{L^{\infty}_T (k_1' \to  k)} \|
    {h^{N_2, L_2}_{k k_2'}} \|_{L^{\infty}_T (k_2' \to  k)} \|
    {h^{N_3, L_3}_{k k_3'}} \|_{L^{\infty}_T (kk_3')}\\
    & \les T^{\frac12} M^{-\frac{3\al}2}  \prod_{j=1}^2 \big\| \jb{\tau}^b \ft{h^{N_j, L_j}_{k_j k_j'}} (\tau) \big\|_{L_{\tau}^2 (\l^2_{k_j} \to \l^2_{k_j'})}  
    \big\| \jb{\tau}^b \ft{h^{N_3, L_3}_{k_3 k_3'}} (\tau) \big\|_{L_{\tau}^2
      (\l^2_{k_3}  \l^2_{k_3'})}  \\ 
    & \lesssim T^{\frac12} M^{-\frac{3\al}2}
    N_3^{\frac{1}{2} + \gamma_0} \lesssim T^{\frac12} M^{- \alpha},
  \end{split} 
\]
where we used the fact $b > \frac12$ and \eqref{Eqn:e1}, which is sufficient for \eqref{Eqn:3}.

All the intermediate cases, from Case (b) to Case (g), can be controlled
similarly.

\section{Proof of Theorem \ref{THM:div}}
\label{Sec:alphagtr1}

This section is devoted to the proof of Theorem \ref{THM:div}, which roughly means that the second Picard
iterate given in \eqref{Picard} converges in $H^{\frac{1}{2} +}$
when $\alpha > 1$, while fails to converge in $L^2$ when $\alpha = 1$.

We start with Theorem \ref{THM:div} (i), i.e. the case when $\alpha > 1$. 
The proof of this part actually follows from a similar argument as in Subsection \ref{SUB:HHCCC}.
We give a proof here for completeness. 
We rewrite
\[ 
(Z_N^{(2)} (t))_k = \sum_{\substack{
     k_1 - k_2 + k_3 = k\\
     k_2 \not\in \{k_1, k_3 \}
   }} {\rm T}^\dl_{kk_1k_2k_3} \frac{g_{k_1} \overline{g_{k_2}} g_{k_3}}{\jbb{k_1}^{\frac{\alpha}{2}} \jbb{k_2}^{\frac{\alpha}{2}} 
   \jbb{k_3}^{\frac{\alpha}{2}}} \cdot
   \Theta (t, \Phi), \]
where
\[ 
{\rm T}^\dl_{kk_1k_2k_3} = \mathbf{1}_{\langle k \rangle < N} \prod_{j = 1}^3 \mathbf{1}_{N^{1 - \delta} < k_j < N} , 
\]
and
\[
\Theta (t, \Phi) = \int^t_0 e^{it' \Phi (\bar{k})} dt'.
\]
Here $\Phi (\bar k) = |k_1 |^{\alpha} - |k_2 |^{\alpha} + |k_3 |^{\alpha} - |k|^{\alpha}$ is given in \eqref{Eqn:Phi}. 
We note that $\Theta (t, \Phi) $ is bounded and also $\Theta (t, \Phi) = \frac{e^{it \Phi} - 1}{i \Phi}$ when $\Phi \neq 0$.
Therefore,
\begin{equation}
  \label{Eqn:shaalpha}
  \begin{split}
    \mathbb{E} \big [\|Z_N^{(2)} (t)\|_{L^2 (\mathbb{T})}^2 \big] 
    & \quad = \sum_{k \in \mathbb{Z}} \sum_{\substack{
      k_1 - k_2 + k_3 = k\\
      k_2 \not\in \{k_1, k_3 \}
    }} {\rm T}^\dl_{kk_1k_2k_3}  \frac{| \Theta (t, \Phi) |^2}{\jbb{k_1}^{\alpha} \jbb{k_2}^{\alpha}
    \jbb{k_3}^{\alpha}} \\
    & \quad \les \sum_{k \in \mathbb{Z}}
    \sum_{\substack{
      k_1 - k_2 + k_3 = k\\
      k_2 \not\in \{k_1, k_3 \}
     }} {\rm T}^\dl_{kk_1k_2k_3}  \frac{\jb{ \Phi (\bar{k}) }^{- 2}}{\jbb{k_1}^{\alpha} \jbb{k_2}^{\alpha}
    \jbb{k_3}^{\alpha}},
    \end{split}
\end{equation}

\noi 
provided $|t| \sim 1$.
By a similar argument as in \eqref{Eqn:pkkp}, we obtain
\[
\begin{split}
 \mathbb{E} \big[\|Z_N^{(2)} (t)\|_{L^2 (\mathbb{T})}^2 \big]
 & \quad \les \sum_{m \in \mathbb{Z}} \frac{1}{\jb{m}^2} \sum_{k \in \mathbb{Z}}
    \sum_{\substack{
      k_1 - k_2 + k_3 = k\\
      k_2 \not\in \{k_1, k_3 \}
     }} \frac{ {\rm T}^\dl_{kk_1k_2k_3}  \mathbf{1}_{\Phi = m + O
    (1)}}{\jbb{k_1}^{\alpha} \jbb{k_2}^{\alpha} \jbb{k_3}^{\alpha}} \\
    & \quad \les \sup_{m \in \mathbb{Z}} \sum_{k \in \mathbb{Z}}
    \sum_{\substack{
      k_1 - k_2 + k_3 = k\\
      k_2 \not\in \{k_1, k_3 \}
     }} \frac{ {\rm T}^\dl_{kk_1k_2k_3}  \mathbf{1}_{\Phi = m + O
    (1)} }{\jbb{k_1}^{\alpha} \jbb{k_2}^{\alpha} \jbb{k_3}^{\alpha}} .
    \end{split}
\]

\noi 
Then we dyadically decompose $k_1$, $k_2$, and $k_3$ to get
\[ 
   \begin{split}
     \mathbb{E} \big[\|Z_N^{(2)} (t)\|^2_{L^2 (\mathbb{T})} \big] \lesssim \sup_{m \in \mathbb{Z}} \sum_{N_1,
     N_2, N_3 } \frac{1}{N_1^{\alpha} N_2^{\alpha} N_3^{\alpha}} \sum_{k \in \mathbb{Z}} \sum_{\substack{
       k_1 - k_2 + k_3 = k\\
       k_2 \not\in \{k_1, k_3 \}
      }} {\rm T}^{\dl,m}_{kk_1k_2k_3},
   \end{split} \]
where 
\[
{\rm T}^{\dl,m}_{kk_1k_2k_3} = {\rm T}^{\dl}_{kk_1k_2k_3} 
       \mathbf{1}_{\Phi = m + O
    (1)}  \prod_{j = 1}^3  \mathbf{1}_{ \frac{N_j}{2} < \langle k_j \rangle \leq N_j} .
\]

\noi 
From Lemma \ref{LEM:S}, we have 
\[
 \sum_{k \in \mathbb{Z}} \sum_{\substack{
       k_1 - k_2 + k_3 = k\\
       k_2 \not\in \{k_1, k_3 \}
      }} {\rm T}^{\dl,m}_{kk_1k_2k_3} \le |S| \les N_{\max} N_{\min}.
\]

\noi 
Recall that $N_1$, $N_2$, and $N_3$ are dyadic numbers bounded from below by $N^{1 - \delta}$.
Then, we may continue with
\[ \begin{split}
     \mathbb{E} \big[\|Z_N^{(2)} (t)\|_{L^2 (\mathbb{T})}^2 \big]  \lesssim \sum_{N_1, N_2, N_3} \frac{N_{\max} N_{\min}}{N_1^{\alpha}
     N_2^{\alpha} N_3^{\alpha}} = \sum_{N^{1 - \delta} <
     N_{\med}}
     N_{\med}^{- \alpha} \lesssim N^{- \alpha (1 -
     \delta)} \lesssim N^{- 1},
   \end{split} \]
provided $\al (1- \dl) > 1$,
which finishes the proof of Theorem \ref{THM:div} (i).

Now we turn to Theorem \ref{THM:div} (ii), where $\alpha = 1$. For this case,
we observe that if $k \ge 0$ and $k_i \ge 0$, then $\Phi (\bar{k}) = 0$ on
the hyperplane $k = k_1 - k_2 + k_3$. We also note that $| \Theta (t, 0) | =
|\int^t_0 dt'| = |t| \sim 1$. Then starting from the first line of
\eqref{Eqn:shaalpha}, the proof will be changed to the following
\begin{equation}
  \label{Eqn:sha}
  \begin{split}
    \mathbb{E} & \big[\|Z_N^{(2)} (t)\|^2_{L^2 (\mathbb{T})}\big] \\
&  = \sum_{k \in \mathbb{Z}}
    \sum_{\substack{
      k_1 - k_2 + k_3 = k\\
      k_2 \not\in \{k_1, k_3 \}
     }} {\rm T}^\dl_{kk_1k_2k_3} \frac{| \Theta (t, \Phi) |^2}{\jbb{k_1} \jbb{k_2} \jbb{k_3}}  \\
    & \ge \sum_k \sum_{\substack{
      k_1 - k_2 + k_3 = k\\
      k_2 \not\in \{k_1, k_3 \},
     }} {\rm T}^\dl_{kk_1k_2k_3}  \frac{| \Theta (t, 0) |^2}{\jbb{k_1}
     \jbb{k_2} \jbb{k_3}} \\
    & \gtrsim | \Theta (t, 0) |^2  \sum_{N^{1 - \delta} < N_1, N_2, N_3
    \le N / 4} (N_1 N_2 N_3)^{- 1}  \bigg( \sum_k
    \sum_{\substack{
      k_1 - k_2 + k_3 = k\\
      k_2 \not\in \{k_1, k_3 \},
     }} {\rm T}^\dl_{kk_1k_2k_3} \bigg),
  \end{split}
\end{equation}
where $N_1, N_2, N_3$ are dyadic numbers. Let $S_N$ be the set given by
\[ 
S_N = \{ (k, k_1, k_2, k_3) \in \mathbb{Z}^4 ; k_1 - k_2 + k_3 = k, k_2
\not\in \{k_1, k_3 \}, 0 \le k < N, k_j \in (N_j / 2, N_j) \},
\]
{\noindent} such that $N^{1 - \delta} < N_1, N_2, N_3 \le N / 4$. Then it is
easy to see that
\[ |S_N | = | \{ (k_1, k_2, k_3) \in \mathbb{Z}^3 ; k_2 \not\in \{k_1, k_3 \},
   k_j \in (N_j / 2, N_j) \} | \gtrsim N_1 N_2 N_3, \]
{\noindent}which together with \eqref{Eqn:sha} implies
\[ \mathbb{E} \big[\|Z_N^{(2)} (t)\|^2_{L^2 (\mathbb{T})} \big] \gtrsim \sum_{N^{1 -
   \delta} < N_1, N_2, N_3 \leq N / 4} 1 \gtrsim \delta^3 (\log N)^3 . \]
{\noindent}This finishes the proof of Theorem \ref{THM:div} (ii).

\ackno{\rm 
The authors are grateful to Chenmin Sun for many helpful discussions, which resulted in Remarks \ref{RMK:phaset} and \ref{RMK:nosmooth}, and also for pointing out an error in a previous version of this text.
We also would like to thank Nikolay Tzvetkov for the useful comments. Additionally, we thank the anonymous reviewers for their comments, which greatly improved the presentation of the paper. 
Y.W. has been supported by the EPSRC New Investigator Award (grant no. EP/V003178/1).} 

\medskip

\noi 
{\bf Declarations}

\smallskip

\noi 
{\bf Conflict of interest}. 
{\rm The content of this paper is original and has not been published or submitted for publication elsewhere. The authors also declare that we have no conflict of interest.}

\smallskip 

\noi 
{\bf Data Availability}.
{\rm Data sharing is not applicable to this article as no data sets were generated or analyzed.}

\end{document}